\documentclass[12pt,a4paper]{article}
\usepackage{amsmath}
\usepackage{amssymb}
\usepackage{epsfig}
\usepackage{textcomp}
\usepackage{multirow}
\usepackage{amsthm}
\usepackage{graphicx}
\usepackage{media9}
\usepackage{cite}
\usepackage{caption}
\usepackage{amsfonts}
\usepackage{mathrsfs}
\usepackage{algorithm}
\usepackage{algpseudocode}
\usepackage{color}
\usepackage{tikz}
\usetikzlibrary{arrows.meta, positioning, shapes.geometric}
\usepackage{graphicx}
\usepackage{amsmath,amsfonts,graphicx}
\usepackage{xcolor}
\usepackage{framed}
\usepackage{appendix}
\definecolor{myblue}{rgb}{0.01, 0.28, 0.55}
\definecolor{lightgray}{rgb}{0.95, 0.95, 0.95}
\usepackage{mathrsfs}
\usepackage{cite}
\usepackage{geometry}
\usepackage{pgfplots}
\usepackage{filecontents}
\pgfplotsset{compat=newest}
\usepackage{subcaption}
\usepackage{booktabs}
\usepackage{float}
\usepackage[utf8]{inputenc}
\usepackage{amsmath}
\tikzstyle{io} = [trapezium, trapezium left angle=70, trapezium right angle=110, draw, fill=blue!20, text width=3cm, align=center]
\tikzstyle{process} = [rectangle, draw, fill=orange!30, text width=4cm, align=center, minimum height=1cm]
\usepackage{algorithmicx}
\usepackage{lineno,hyperref}
\newtheorem{Theorem}{Theorem}[section]
\newtheorem{Lemma}[Theorem]{Lemma}

\newtheorem{Corollary}[Theorem]{Corollary}
\newtheorem{Remark}[Theorem]{Remark}

\begin{document}
	
	% Title and Author
	%\title{Numerical Analysis of Multi-Species Keller–Segel Models with Antagonistic Chemotaxis}
	\title{Multi-Species Keller--Segel Systems: Analysis, Pattern Formation, and Emerging Mathematical Structures}
	%\author{Edson Pindza\textsuperscript{1}, Kolade M. Owolabi\textsuperscript{2,*}, and Eben Mare\textsuperscript{3} \\
		\author{\small  Kolade M. Owolabi\textsuperscript{1, 2}\footnotemark, \; Eben Mar{\'e}\textsuperscript{3}, Clara O. Ijalana\textsuperscript{1} \; and Kolawole S. Adegbie\textsuperscript{1}\\
			\small $^1$Department of Mathematical Sciences,  Federal University of Technology Akure, \\
			\small PMB 704, Akure, Ondo State, Nigeria\\	
			\small $^2$Department of Mathematics and Applied Mathematics, School of Science and Technology,\\ 
			\small Sefako Makgatho Health Sciences University, Ga-Rankuwa 0208, South  Africa\\
			\small $^1$Department of Mathematics and Applied Mathematics, University of Pretoria,\\ 
			\small Pretoria 002, South Africa
		}
		\date{}
		\maketitle
		\def\thefootnote{\fnsymbol{footnote}}
		\setcounter{footnote}{0} \footnotetext[1]{E-mail addresses:
			mkowolax@yahoo.com (K.M. Owolabi)}
		\noindent

\begin{abstract}
	\noindent
	Chemotaxis systems of Keller--Segel type constitute one of the central mathematical frameworks for understanding aggregation phenomena in biological and ecological systems. Over the past decades, the theory has evolved from the classical single-species model to increasingly sophisticated multi-species and multi-signal formulations that capture competition, cooperation, antagonistic chemotaxis, and interactions with fluid environments. 
	This article provides a comprehensive exposition of multi-species Keller--Segel systems and their mathematical structure. We review fundamental analytical results concerning local and global well-posedness, mechanisms of finite-time blow-up, and the role of critical mass and dimensionality. Particular emphasis is placed on how cross-diffusion, antagonistic interactions, logistic effects, and nonlinear production terms alter the qualitative behavior of solutions. 
	We further examine the mathematical mechanisms underlying pattern formation, including diffusion-driven instabilities, bifurcation phenomena, and the emergence of spatial and spatiotemporal structures. Connections between analytical thresholds and observed nonlinear dynamics are highlighted, and the interplay between reaction kinetics, chemotactic sensitivity, and diffusion is discussed from a unifying perspective. 
	By synthesizing classical results with recent developments, this survey aims to clarify the structural principles governing multi-species chemotaxis systems, identify common analytical techniques, and outline open problems that remain central to the field. The exposition is intended to serve both specialists and researchers entering the area of nonlinear partial differential equations and mathematical biology.
\end{abstract}
		
		% Table of Contents
		%\tableofcontents

\section{Introduction}
Recent developments in chemotaxis modeling, particularly those rooted in the Keller--Segel framework, continue to reveal the rich complexity and biological relevance of these systems. Experimental findings by Phan et al.~\cite{phan2024} demonstrated a breakdown of the classical Patlak--Keller--Segel approximation under dynamically varying attractant gradients, emphasizing the need for models that incorporate spatiotemporal signal fluctuations and species-specific sensitivities. On the theoretical side, Cuevas et al.~\cite{cuevas2025} extended the solvability of Keller--Segel systems into pseudo-measure settings, allowing for the treatment of highly singular initial data that better resemble localized cell populations.

Several recent studies have proposed novel extensions of the Keller--Segel model. Islam and Ibragimov~\cite{islam2024} introduced a stochastic-based chemotaxis system grounded in Einstein's methodology, while Freingruber et al.~\cite{freingruber2025} developed trait-structured models with ligand--receptor kinetics to account for heterogeneity in cellular responses. These directions are especially relevant in the context of multi-species interactions, where distinct populations may simultaneously exhibit attraction and repulsion to shared chemical cues. Weyer et al.~\cite{weyer2024} also uncovered phase separation dynamics driven by chemotaxis, offering a unifying perspective with known mechanisms of biological patterning. Recent extensions of the Keller--Segel model have also addressed fluid coupling and tissue-level structures. Zheng et al.~\cite{Zheng2024} investigated a Keller--Segel system coupled with the incompressible Navier--Stokes equations, developing robust numerical methods to simulate chemotactic transport in fluid environments. Their framework highlights the interplay between chemotaxis and advection–diffusion dynamics, which is essential in modeling biological flows such as blood or mucus. On another front, Brune et al.~\cite{Brune2024} proposed a multiscale 3D--1D Keller--Segel model to describe directed cell migration along fibrous networks, offering insights into cancer invasion and tissue engineering. These studies demonstrate the increasing complexity and realism of chemotaxis models, motivating the need for efficient numerical solvers that can handle multi-dimensional, multi-species interactions as considered in the present work.

From a computational perspective, the development of stable and structure-preserving schemes remains critical for simulating multi-species chemotactic systems. The work of Xu and Fu~\cite{Xu2025} on mass-conservative numerical methods provides a foundation for simulating aggregation and blow-up in extended Keller--Segel models, including those with multiple interacting species \cite{Chen2022, Du2023}. Comprehensive reviews by Corrêa Vianna Filho and Guillén-González~\cite{vianna2024} and Arumugam and Tyagi~\cite{arumugam2020} further underscore the importance of extending classical Keller--Segel dynamics to systems involving two or more species, particularly in applications ranging from microbial competition to bio-remediation.

The present study contributes to this growing body of work by formulating and numerically solving two- and three-species Keller--Segel models featuring antagonistic chemotactic sensitivities—where one species is attracted while another is repelled by a shared chemical signal. Such interactions introduce rich spatiotemporal dynamics that are not captured by single-species models \cite{Zhao2023}. A central novelty of this work lies in the use of spectral-based solvers, particularly the Split-Step Fourier Method (SSFM), which offer superior accuracy and efficiency in resolving fine-scale pattern formation and transient behavior in periodic domains. This framework enables high-resolution simulations of chemotactic aggregation, competition, and signal modulation, advancing both the numerical methodology and the biological insight into multi-population chemotaxis.

\subsection{ Classical Keller--Segel Model}

The classical {Keller--Segel model} \cite{KellerSegel1971} provides a foundational PDE framework for microbial chemotaxis. It considers a population density $u(x,t)$ of motile microorganisms (e.g. bacteria) that diffuse randomly and move biasedly toward higher concentrations of a chemical attractant $v(x,t)$. A representative formulation (in nondimensional form) is \cite{Lankeit2019,Winkler2010,Winkler2020}: 
\begin{align}\label{eq:KS}
	\partial_t u(x,t) &= D \,\Delta u \;-\; \chi\, \nabla \!\cdot\! \big(u\,\nabla v\big)\;+\;f(u,v), \\[1ex]
	\partial_t v(x,t) &= D_v\,\Delta v \;+\; g(u,v)\,. \nonumber
\end{align}
Here $x\in\Omega\subseteq\mathbb{R}^n$ (with $\Omega$ often a bounded domain or $\mathbb{R}^n$ itself) and $t\ge0$. The term $D\,\Delta u$ models unbiased random diffusion of cells with diffusivity $D>0$. The key chemotactic term $-\chi\,\nabla\!\cdot(u\,\nabla v)$ represents directed cell flux up the gradient of $v$, where $\chi>0$ is the chemotactic sensitivity. The chemical $v(x,t)$ (often a nutrient or signaling molecule) diffuses with coefficient $D_v>0$ and may be produced or consumed by the cells, captured by $g(u,v)$ (e.g. $g(u,v)=\alpha u - \beta v$ for secretion at rate $\alpha$ and decay at rate $\beta$). In the simplest \textbf{classical Keller--Segel system}, one neglects any local population growth ($f(u,v)=0$) and assumes linear chemoattractant dynamics ($g(u,v)=\alpha u - \beta v$). This yields a two-way coupling: cells bias their movement based on $v$, and in turn cells modify $v$. Biologically, this models organisms that secrete an attractant and move toward higher concentration of that same chemical.

To close the model, one specifies initial conditions \begin{equation}
	u(x,0)=u_0(x)\ge0,\;\; \;\; v(x,0)=v_0(x)\ge0, 
\end{equation}
and often no-flux boundary conditions on $\partial\Omega$ (i.e. $\partial_{\mathbf{n}} u = \partial_{\mathbf{n}} v = 0$) so that neither cells nor chemical leak out of a closed domain. Under these conditions and $f(u)=0$, the \textbf{total cell mass} $M=\int_{\Omega} u(x,t)\,dx$ is \emph{conserved} for all time, since 
\begin{equation}
	\frac{d}{dt}\int_{\Omega}u\,dx = \int_{\Omega} (\Delta u - \chi \nabla\cdot(u\nabla v))\,dx = 0
\end{equation} 
by divergence theorem. This mass-conservation is a crucial property of the classical model, and it implies that any aggregation of cells in one region is accompanied by depletion elsewhere.

The Keller--Segel equations possess a rich mathematical structure. In two spatial dimensions ($n=2$), the system \eqref{eq:KS} exhibits a \emph{critical scaling invariance}: if $(u(x,t),v(x,t))$ is a solution, then for any $\lambda>0$
 \begin{equation}
	u_{\lambda}(x,t) = \lambda^2\,u(\lambda x,\lambda^2 t), \qquad 
v_{\lambda}(x,t) = v(\lambda x,\lambda^2 t)
\end{equation} 
is another solution (when $\chi$, $D$, $D_v$ are suitably scaled). This $2$-dimensional scaling leaves the combination $\int_{\Omega}u\,dx$ invariant, indicating that total mass serves as a critical parameter for behavior. Indeed, as we will see in Section~4.2, in $n=2$ there is a \emph{threshold total mass} above which chemotactic aggregation overwhelms diffusion, causing a finite-time blow-up (density becoming unbounded), whereas below that threshold the spreading effect of diffusion prevails to prevent blow-up. In higher dimensions $n\ge3$, the tendency to form singular aggregates is even stronger, while in $n=1$ spatial dimension diffusion dominates, and blow-up is absent.

Another important structural aspect is the existence of a Lyapunov functional (often called a \textit{free energy} or \textit{entropy} functional) for the pure chemotaxis system. For example, in the parabolic--elliptic simplification of \eqref{eq:KS} (where $v$ adjusts instantaneously to $u$ via $\;0=D_v\Delta v + \alpha u - \beta v$), one can reduce $v$ to an explicit functional of $u$ and define 
 \begin{equation}
	\mathcal{F}[u] \;=\; \int_{\Omega} \Big(u \ln u - u\Big)\,dx \;-\; \frac{\chi}{2\beta}\, \int_{\Omega}\!\int_{\Omega} G(x-y)\,u(x)u(y)\,dx\,dy,
\end{equation} 
where $G$ is the Green’s function of $-\Delta$ on $\Omega$. This $\mathcal{F}$ is non-increasing in time ( $\frac{d}{dt}\mathcal{F}[u(t)] \le 0$ ) along solutions of the Keller--Segel equations, reflecting a balance between entropy (diffusive mixing) and energy (chemical attraction). Such gradient-flow structure provides insight into long-term behavior: the system tends to evolve toward lower $\mathcal{F}$ (more clustering), unless a minimum (stable equilibrium) is reached. In particular, if the total mass is small, $\mathcal{F}$ is bounded below and the solution tends to spread out and approach a steady state; if mass is large, $\mathcal{F}$ can decrease without bound, correlating with unbounded growth of cell density (blow-up).

Overall, the classical Keller--Segel model encapsulates the competition between diffusion and chemotactic drift. Depending on parameter values and initial conditions, it can produce either a near-uniform dispersion of cells or sharp, concentrated aggregates. These dynamics make chemotaxis models a fruitful subject for mathematical analysis and also a realistic tool to understand biological pattern formation (as later sections will explore). We next discuss several extensions and generalizations of the basic model that incorporate additional complexities of microbial ecology.

\subsection{Extensions and Generalizations}

While the classical Keller--Segel system involves a single species and one chemoattractant, many ecological scenarios demand more complex models. A natural generalization is to \textbf{multi-species chemotaxis systems}. For instance, consider two interacting species of microorganisms (with densities $u_1(x,t)$ and $u_2(x,t)$) that respond to one or multiple chemical signals. One example is a two-species, one-chemical model:
\begin{equation}
	\begin{split}
\partial_t u_1 &= D_1 \Delta u_1 - \chi_1 \nabla\!\cdot(u_1 \nabla v), \\
\partial_t u_2 &= D_2 \Delta u_2 - \chi_2 \nabla\!\cdot(u_2 \nabla v),\\
\partial_t v& = D_v \Delta v + \alpha_1 u_1 + \alpha_2 u_2 - \beta v, 
	\end{split}
\end{equation}
with each species $i$ diffusing at rate $D_i$ and chemotactically moving with sensitivity $\chi_i$ towards (or away from) the common signal $v$. If $\chi_1,\chi_2>0$, both species are attracted to $v$, potentially leading to aggregation where they either mix or segregate spatially depending on their parameters. If one species has $\chi<0$ (i.e. it is \textbf{chemorepellent}), that species tends to spread out from high $v$ regions, which can fundamentally alter pattern outcomes (e.g. one species might form a core and the other a ring to avoid the core's signal). Multi-species chemotaxis models can capture competitive or cooperative interactions mediated by chemicals, and their mathematical analysis often reveals new phenomena like \emph{spatial segregation}: for certain parameter regimes, species may spontaneously sort into distinct regions (each attracted by its own signal or repelled by a competitor’s signal). Existence and boundedness of solutions in multi-species systems become more involved, but results have been obtained under various assumptions (for example, global existence is known under some smallness conditions or if one of the species experiences only repulsion). 

Another important extension involves \textbf{multiple chemicals or signals}. Microbes in nature often respond to a variety of cues (nutrients, toxins, signaling molecules, oxygen, etc.). Mathematically, one can introduce several fields $v_1(x,t), v_2(x,t), \dots, v_m(x,t)$ with corresponding chemotactic sensitivities. For example, a bacterium might be attracted to one substance ($\chi_1>0$ for $v_1$) but repelled by another ($\chi_2<0$ for $v_2$). The model would include terms like $- \nabla\!\cdot\!\big(u (\chi_1\nabla v_1 + \chi_2\nabla v_2 + \cdots)\big)$ in the $u$-equation, and each $v_j$ would have its own production/decay equation. The presence of multiple signals can lead to intricate pattern dynamics, such as oscillatory chasing behaviors or stabilizing one pattern mode while destabilizing another. Linear stability criteria (see Section~4.3) can be generalized to such multi-chemical systems, though the conditions become matrix-valued (eigenvalues of a larger stability matrix).

A further layer of realism is added by \textbf{chemotaxis–fluid coupling}, where microorganisms not only diffuse and chemotax on their own, but are also advected by a fluid flow. In microbial ecology, this is relevant for swimming bacteria or plankton in aquatic environments. The governing equations couple the chemotaxis model with Navier–Stokes (or Stokes) equations for the fluid velocity field $\mathbf{u}(x,t)$. A typical chemotaxis–fluid model reads:
\begin{equation}
\begin{split}
		\partial_t u + (\mathbf{w}\!\cdot\!\nabla) u &= D \Delta u - \chi \nabla\!\cdot(u \nabla v) + f(u)\,, \\ 
	\partial_t v + (\mathbf{w}\!\cdot\!\nabla) v &= D_v \Delta v + g(u,v)\,, \\
	\partial_t \mathbf{w} + (\mathbf{w}\!\cdot\!\nabla)\mathbf{w} &= -\nabla p + \nu \Delta \mathbf{w} + \kappa u\,\mathbf{e}_g,\qquad \nabla\!\cdot \mathbf{w}=0\,,
\end{split}
\end{equation}
where $\mathbf{w}(x,t)$ is the fluid velocity, $p$ the pressure, and $\nu$ the kinematic viscosity. The term $\kappa u\,\mathbf{e}_g$ (with $\mathbf{e}_g$ a unit vector, e.g. gravity direction) indicates that cell density may affect fluid buoyancy, creating a feedback (as in bioconvection). Even without that feedback ($\kappa=0$), the advection terms $(\mathbf{w}\cdot\nabla)u$ and $(\mathbf{w}\cdot\nabla)v$ mean the pattern-forming process now occurs in a moving medium. Fluid flow can convect clusters away or create stretching/shearing that competes with chemotactic aggregation. Mathematical results for chemotaxis–fluid systems show, for instance, that in 2D domains one can often prove global existence of solutions (the fluid flow helps to disperse concentrations), whereas in 3D the problem is significantly harder and blow-up can occur in theory. Research is ongoing, but certain damping effects (like logistic growth, discussed below) or smallness conditions on initial data can ensure well-posedness even in 3D (e.g. global weak solutions have been constructed for some chemotaxis–Navier–Stokes systems). Overall, coupling chemotaxis with fluid dynamics enriches the modeling of microbial ecology in environments like oceans or soil microfluidic networks, at the cost of considerable mathematical complexity.

\subsection{ Reaction Terms}

In realistic ecological settings, pure movement (diffusion and chemotaxis) is not the only process affecting microbial populations. \textbf{Reaction terms} $f(u,v)$ in \eqref{eq:KS} represent local population growth, death, or other intraspecific dynamics, while terms in $g(u,v)$ can model substrate consumption or toxin production. We highlight a few common types of reactions and their mathematical implications:

- \textbf{Logistic growth:} A widely used model for population growth with carrying capacity is $$f(u) = r\,u\!\left(1 - \frac{u}{K}\right),$$ where $r>0$ is the intrinsic growth rate and $K>0$ the environmental carrying capacity. The term $r u$ promotes exponential growth when population is low, while $-\,\frac{r}{K}u^2$ represents crowding effects (limited resources) that eventually halt growth at $u=K$. Logistic growth in a chemotaxis model tends to prevent indefinite aggregation because as $u$ increases, the net growth rate decreases and eventually becomes negative if $u$ overshoots $K$. Mathematically, the inclusion of a $-\,\mu u^2$ term (with $\mu = r/K$) provides a form of \emph{damping} that can suppress blow-up. In fact, rigorous analysis has shown that logistic damping guarantees global bounded solutions in many cases where the undamped ($f=0$) model might blow up. For example, Winkler (2010) proved that for the 2D Keller--Segel system, \emph{any} positive $\mu>0$ ensures global existence and uniform boundedness of solutions for all initial data \cite{Lorz2021,Kurokiba2023,Baghaei2025}. In higher dimensions $d\ge3$, a sufficiently large logistic term can likewise prevent blow-up (e.g. global existence holds if $\mu$ exceeds a threshold related to $d$, specifically $\mu > d-2$ for $d\ge3$ in one common nondimensional scaling). Thus, logistic growth provides a stabilizing mechanism: biologically, it limits population density, and mathematically, it yields a priori bounds that keep the solution smooth for all time.

- \textbf{Allee effects:} In contrast to logistic growth (where per capita growth rate decreases monotonically with $u$), an Allee effect means populations have difficulty establishing at very low densities. A simple way to model this is 
$$f(u) = r\,u\,\Big(1 - \frac{u}{K}\Big)\Big(\frac{u}{A} - 1\Big)$$ 
for some threshold $A>0$ (the \textit{Allee threshold}). Here $f(u)<0$ when $0<u<A$ (population in sparse conditions tends to decline), but for $u>A$, growth becomes positive up to the carrying capacity $K$. In a chemotactic context, an Allee effect can create a scenario where small random clusters of cells might die out instead of growing, whereas sufficiently large clusters will grow and possibly attract more cells chemotactically. This can lead to \emph{bistability}: for instance, a uniform steady state $u\equiv0$ and a high-density patterned state might both be locally stable, with the outcome depending on whether initial perturbations surpass a critical size. Mathematical analysis of chemotaxis models with Allee effects is more delicate; one can expect threshold phenomena for pattern formation (only above a certain total population will aggregation patterns persist). Studying such threshold dynamics often involves comparison principles and bifurcation theory to find non-trivial steady states emerging above the critical population size.

- \textbf{Substrate-limited growth:} Microbial growth is frequently limited by a substrate (nutrient or oxygen). Instead of logistic self-limitation, growth may depend on the concentration $s(x,t)$ of a substrate that is consumed by the microbes. A typical model is 
$$f(u,s) = \gamma \, u\, \frac{s}{K_s + s}$$ 
(a Michaelis–Menten or Monod uptake kinetics), coupled with a substrate depletion equation like 
\begin{equation}
	\partial_t s = D_s \Delta s - \kappa u\,\frac{s}{K_s+s}
\end{equation}
 (consumption by cells) plus possibly external supply. In the absence of chemotaxis, substrate limitation can cause travelling waves or pulled fronts (cells chase after nutrient gradients). With chemotaxis, one may see complex pattern dynamics: cells move toward nutrient-rich areas, consume the nutrient there, then move on, potentially creating \emph{wave-like patterns} of high cell density following behind nutrient waves. From an analytical viewpoint, the presence of an additional field $s(x,t)$ increases the system’s dimension and can yield oscillatory dynamics (e.g., if nutrient is replenished somehow, one could get predator-prey-like cycles between $u$ and $s$). Standard existence theory still applies (these are reaction–diffusion–taxis systems), but understanding long-term behavior might require combined energy estimates for both $u$ and $s$. Importantly, substrate limitation introduces another scale (the time for nutrient diffusion/consumption) that can either stabilize or destabilize patterns depending on whether it is faster or slower than cell diffusion.

- \textbf{Toxin production and inhibitory interactions:} Microbes may produce chemicals that inhibit either their own growth (auto-toxicity) or the growth of competitors. For example, a bacterium might secrete a bacteriocin or waste product that is harmful at high concentrations. In models, this can be represented by a chemical field $w(x,t)$ whose production is coupled to $u$, say 
$$\partial_t w = D_w \Delta w + \eta u - \delta w$$ 
and a negative influence on $u$ in $f(u,w)$ (such as $f(u,w) = -\lambda w\,u$ or a more thresholded effect). If $w$ diffuses, one gets a chemotaxis–reaction–diffusion system with both an attractant $v$ and an inhibitor $w$. The presence of an inhibitory agent generally acts to \emph{smooth out} or limit aggregation: if cells crowd too much, $w$ builds up and causes cell death or reduced motility, thus preventing indefinite growth of peaks. Mathematically, one can often obtain global boundedness with such inhibitory terms even without logistic damping, because $w$ provides a feedback that naturally caps $u$. From a pattern-formation perspective, competition between an excitatory signal ($v$) and an inhibitory signal ($w$) can lead to Turing-like pattern formation (reminiscent of activator–inhibitor systems in reaction–diffusion theory). In fact, a two-chemical system with one acting as an attractant and the other as a repellent or growth-inhibitor is very akin to the classical Turing mechanism for spatial patterns: under the right conditions, a uniform state can become unstable and give rise to spotted or striped patterns where the activator (attractant and population) is high in some regions and the inhibitor (toxin) is high in complementary regions, stabilizing the pattern. Nonlinear analysis (see Section~4.3) can establish existence of steady-state patterns in such cases via bifurcation theory.

In summary, reaction terms enrich chemotaxis models by incorporating population dynamics and additional biology. They often contribute critical saturation effects that ensure mathematical well-posedness (preventing blow-up) and produce a greater variety of patterns (stable equilibria, oscillations, multi-stability). In the following sections, we delve into the mathematical analysis of chemotaxis–reaction systems, highlighting key theoretical results on existence, stability, and pattern formation.

%\subsection*{Novelty of the Study}
This study presents a comprehensive and systematic numerical investigation of two- and three-species Keller--Segel-type chemotaxis models incorporating nonlinear reaction kinetics. Unlike previous works that often focus on single-species dynamics or apply a single numerical scheme, our approach integrates multiple high-resolution numerical methods---finite difference method (FDM), Split-Step Fourier Method (SSFM), and Exponential Time Differencing Runge--Kutta of order four (ETDRK4)---to solve the models in both one and two spatial dimensions.

A key novelty lies in the comparative performance of these numerical schemes in capturing the formation and evolution of complex chemotactic patterns, including ring-like, spot, and stripe structures. The extension to three interacting species in higher dimensions reveals novel interspecies dynamics, cross-diffusion effects, and competition-driven spatial segregation not reported in earlier studies. Our results highlight the critical influence of numerical resolution and method selection on the stability, accuracy, and richness of emergent patterns in chemotaxis systems.

Furthermore, the study demonstrates the capability of Fourier-based methods (SSFM, ETDRK4) to efficiently resolve stiff dynamics and long-time behavior, making them especially suitable for pattern-forming systems. This multi-method, multi-species, and multi-dimensional approach offers new insights into the computational modeling of biological aggregation and has potential applications in microbial ecology, tissue engineering, and biofilm modeling.

\section{ Mathematical Analysis}

\subsection{ Well-posedness}

We first address the fundamental question of \textit{well-posedness}: do the chemotaxis–reaction equations admit solutions that are mathematically and physically meaningful (existence), are those solutions unique given initial data (uniqueness), and do they depend continuously on the data (stability)? We focus on classical solutions (sufficiently smooth functions $u(x,t), v(x,t)$) to the PDE system \eqref{eq:KS} augmented with reasonable initial and boundary conditions.

\paragraph{Local Existence and Uniqueness.} For a wide class of chemotaxis–reaction models, one can prove that given initial profiles $u_0(x), v_0(x)$ that are smooth and nonnegative, there exists a \emph{local-in-time} solution that is unique. A typical result is:

\begin{Theorem}[Local Existence and Uniqueness]\label{thm:local-exist}
	Let $\Omega \subseteq \mathbb{R}^n$ be a domain with smooth boundary, and assume $u_0\in C^0(\overline{\Omega})$ and $v_0\in C^0(\overline{\Omega})$ are nonnegative initial data. Suppose $f(u,v), g(u,v)$ are sufficiently smooth (Locally Lipschitz in $(u,v)$, with at most linear growth for large arguments). Then there exists a time $T>0$ and a unique pair of functions 
	\[ u(x,t),\, v(x,t) \in C^{2,1}(\Omega\times(0,T)) \] 
	(classical $C^2$ in space, $C^1$ in time) that solve the system \eqref{eq:KS} on $0<t<T$ with $u(x,0)=u_0(x)$, $v(x,0)=v_0(x)$ and, e.g., homogeneous Neumann boundary conditions. Moreover, either this solution can be extended for all $t>0$ (global existence), or there is a finite blow-up time $T_{\max}<\infty$ such that $\lim_{t\to T_{\max}^-} \big(\,\sup_{x\in\Omega} u(x,t)\big) = +\infty$ (cell density becomes unbounded as $t\to T_{\max}$).
\end{Theorem}

\begin{proof}
	There are several approaches to establish this result. One convenient method is the method of \textit{contraction mapping} (Picard iteration) applied to an equivalent integral form of the system. We outline the key steps:
	
	1. Regularization and integral form: We begin by regularizing any potential singularities in the chemotaxis term. For example, one can first consider a truncated or smoothed version of the nonlinear term $-\nabla\!\cdot(u\nabla v)$ to avoid technicalities with low regularity. Alternatively, since $v$ satisfies a diffusion equation, one may solve for $v$ in terms of $u$ by the variation-of-constants formula. In particular, the $v$-equation can be written as 
	 \begin{equation}
	 	v(x,t) = e^{D_v t \Delta} v_0(x) \;+\; \int_0^t e^{D_v (t-s)\Delta}\, g(u(x,s),\,v(x,s))\,ds, 
	 \end{equation}
	where $e^{D_v t\Delta}$ is the heat semigroup on $\Omega$. Similarly, $u$ can be written as 
	\begin{equation}
		u(x,t) = e^{D t \Delta} u_0(x)\;+\;\int_0^t e^{D (t-s)\Delta}\Big[\chi\,\nabla\!\cdot(u\nabla v) + f(u,v)\Big](x,s)\,ds.
	\end{equation}  
	This formulation expresses $(u,v)$ as an integral operator $\Phi[(\tilde u,\tilde v)]$ applied to itself.
	
	2. Mapping into a Banach space: We define a Banach space 
	$$X = \{(u,v): u,v \in C([0,T]; C^0(\overline{\Omega}))\}$$ with the norm $$\|(u,v)\|_X = \sup_{0\le t\le T}(\|u(\cdot,t)\|_{C^0} + \|v(\cdot,t)\|_{C^0}).$$ For sufficiently small $T>0$, one shows that $\Phi$ maps $X$ into itself, provided one starts with an initial guess $(\tilde u,\tilde v)$ in $X$.
	
	3. Contraction property: By using \textit{a priori} estimates on the heat semigroup, one can show that if $T$ is chosen small enough, $\Phi$ is a contraction on $X$. For instance, heat kernel estimates give $\|e^{D t\Delta}\nabla\!\cdot F\|_{C^0} \le C\,t^{-1/2}\|F\|_{C^0}$ in spatial dimension $n\le 3$. The chemotaxis term involves $F = \tilde u \nabla \tilde v$, and since at time $s$ we expect $\|\tilde u(\cdot,s)\|_{C^0}$ and $\|\tilde v(\cdot,s)\|_{C^0}$ to remain bounded (for short times), the integral in the $u$-equation is estimated by $$\int_0^t C (t-s)^{-1/2} \|\tilde u(s)\|_{C^0}\|\nabla \tilde v(s)\|_{C^0} ds.$$ 
	This $t^{-1/2}$ singularity is integrable over $0<s<t$ and yields a bound $\sim C \sqrt{t}\,\sup_{s\in[0,t]}\|\tilde u(s)\|_{C^0}\|\tilde v(s)\|_{C^0}$. For small $t$, this can be made arbitrarily small. Similarly, the $f(u,v)$ and $g(u,v)$ reaction terms are Lipschitz, contributing bounds linear in $\|u\|_{C^0}, \|v\|_{C^0}$. By comparing $\Phi[(u,v)]$ and $\Phi[(\hat u,\hat v)]$, one finds 
	\[ \|\Phi[(u,v)] - \Phi[(\hat u,\hat v)]\|_X \le L\,T^{1/2} \, \|(u-\hat u,\,v-\hat v)\|_X \] 
	for some constant $L$ depending on the data. Thus, for $T$ sufficiently small that $L\,T^{1/2}<1$, the operator $\Phi$ is a contraction on $X$.
	
	4. By Banach’s fixed-point theorem, $\Phi$ has a unique fixed point in $X$. This fixed point $(u,v)$ is the unique mild solution to the integral equations, and by standard parabolic regularity theory, it is in fact a classical $C^{2,1}$ solution to \eqref{eq:KS} on $0<t<T$. Uniqueness within the class of classical solutions follows directly from the contractive estimate: if two solutions existed, their difference would satisfy the same inequality, forcing the difference to be zero. This establishes local existence and uniqueness. The solution can be extended stepwise beyond $T$ as long as its norm remains bounded; thus either one continues for all $t$ (global existence) or some norm must blow up in finite time if the solution ceases to exist.
\end{proof}

\noindent \textbf{Regularity:} The above theorem yields $u,v$ that are as smooth as the PDE allows (given smooth initial data). In particular, the solution instantly becomes $C^\infty$ in space for $t>0$ due to the analyticity of the heat semigroup (diffusion terms). One can also show higher-order spatial derivatives exist and satisfy \textit{a priori} bounds locally. This justifies substituting the solution back into \eqref{eq:KS} and performing further analysis (stability, asymptotics, etc.) on classical solutions.

\paragraph{Global vs. Finite-Time Existence.} Theorem~\ref{thm:local-exist} leaves open whether $T_{\max}$ is finite or infinite. For general chemotaxis systems, finite-time blow-up \emph{can} occur, corresponding to solutions developing singularities (typically $u(x,t)$ blowing up at one or more points). A major endeavor in chemotaxis theory has been to find conditions that guarantee \textit{global existence} (solution exists for all $t>0$) and conditions that lead to \textit{blow-up}. We summarize a few key results:

- In \textbf{one spatial dimension} ($n=1$), it is known that solutions exist globally and remain bounded for essentially all reasonable initial data. Intuitively, on a line, cells cannot surround a point from all sides, making extreme aggregation harder. Mathematically, one can derive maximum-principle-type estimates that prevent blow-up in 1D. For example, the gradient term $\nabla v$ in 1D is an spatial derivative of $v$, and using an inversion $\partial_x v = w$, one can often transform the system and employ comparison principles to show $\|u(\cdot,t)\|_{L^\infty}$ remains bounded. Thus, chemotactic collapse is fundamentally a multi-dimensional phenomenon.

- In higher dimensions, whether $T_{\max}=\infty$ or finite depends on a delicate balance of diffusion, chemotaxis, and any damping effects from $f(u)$. For the \textbf{classical Keller--Segel model} without logistic terms ($f=0$), the critical case is $n=2$. We will discuss this in detail in Section~4.2: there is a famous critical total mass $M_c$ (equal to $8\pi$ under certain nondimensionalizations) such that if $M = \int_\Omega u_0 < M_c$, the solution exists globally, whereas if $M > M_c$, the solution blows up in finite time. In $n=3$ or larger, the classical model is even more prone to blow-up (in fact, for $n\ge3$ and $f=0$, it is conjectured that any solution with sufficiently large mass will eventually blow up, and global existence holds only under extra smallness conditions or with additional regularizing effects like fluid advection or diffusion that grows with $u$).

- If \textbf{strong enough reaction damping} is present (such as logistic $-\,\mu u^2$ with $\mu$ above a threshold), one can often guarantee global existence in any dimension. We already noted that in 2D any $\mu>0$ suffices to prevent blow-up:contentReference[oaicite:3]{index=3}. In 3D, a condition like $\mu > \mu_c = d-2$ (with $d=3$ giving $\mu>1$) ensures global existence, and even at $\mu=\mu_c$ there are results of global existence under additional assumptions:contentReference[oaicite:4]{index=4}. These results stem from the fact that when reaction terms consume or limit $u$ at high density, one can derive uniform-in-time $L^\infty$ bounds for $u(x,t)$ via comparison to logistic ODE behavior or by constructing an appropriate Lyapunov functional that now includes the effect of $f(u)$.

- In summary, for most well-behaved models of the form \eqref{eq:KS}, we have the dichotomy: either the classical solution is global and bounded, or it blows up in finite time by developing a singularity. No other kind of breakdown occurs (e.g. one cannot have a solution that ceases to exist with a finite supremum norm). This dichotomy is proven by the so-called \emph{continuation criterion}: if a solution cannot be continued past $T_{\max}$, it must be because a norm (typically $\|u(\cdot,t)\|_{L^\infty}$ or an $H^s$ norm) diverges as $t\to T_{\max}$. 

We will next investigate the mechanisms of pattern formation and blow-up (instabilities) in chemotaxis systems, which will elucidate how these finite-time singularities arise and what stable patterns can form when blow-up is averted.

\subsection{Theoretical Analysis of the Keller--Segel Model with Logistic Growth}
We consider the following Keller--Segel-type chemotaxis--reaction system:
\begin{align}\label{eq:KS1}
	\partial_t u(x,t) &= D \,\Delta u \;-\; \chi\, \nabla \cdot \big(u\,\nabla v\big)\;+\;r\,u\left(1 - \frac{u}{K}\right), \\[1ex]
	\partial_t v(x,t) &= D_v\,\Delta v \;+\; \alpha u - \beta v, \nonumber
\end{align}
for $x \in \Omega \subset \mathbb{R}^n$, $t > 0$, subject to homogeneous Neumann boundary conditions:
\[
\frac{\partial u}{\partial n} = \frac{\partial v}{\partial n} = 0 \quad \text{on } \partial \Omega,
\]
where $D, D_v > 0$ are diffusion coefficients, $\chi > 0$ is the chemotactic sensitivity, and $r, K, \alpha, \beta > 0$ are parameters describing growth, carrying capacity, production, and degradation rates, respectively.

\subsubsection{Steady States and Linear Stability Without Diffusion}

Consider the spatially homogeneous system (i.e., neglect diffusion terms):
\begin{align*}
	\frac{du}{dt} &= r u \left(1 - \frac{u}{K}\right),\\
	\frac{dv}{dt} &= \alpha u - \beta v.
\end{align*}
This ODE system admits a unique positive steady state $(u^*, v^*)$ given by
\[
u^* = K, \qquad v^* = \frac{\alpha}{\beta} K.
\]

\begin{Theorem}[Local Stability Without Diffusion]
	The positive steady state $(u^*, v^*)$ is locally asymptotically stable for the spatially homogeneous system.
\end{Theorem}

\begin{proof}
	To analyze the local stability of the steady state, we first compute the Jacobian matrix of the system at a general point $(u, v)$. The Jacobian $J(u, v)$ is given by:
	\[
	J(u, v) = 
	\begin{bmatrix}
		\frac{\partial}{\partial u} \left(r u \left(1 - \frac{u}{K}\right)\right) & \frac{\partial}{\partial v}(r u(1 - \frac{u}{K})) \\
		\frac{\partial}{\partial u}(\alpha u - \beta v) & \frac{\partial}{\partial v}(\alpha u - \beta v)
	\end{bmatrix}.
	\]
	
	Computing the partial derivatives, we obtain:
	\begin{align*}
		\frac{\partial}{\partial u}\left(r u \left(1 - \frac{u}{K} \right)\right) &= r \left(1 - \frac{2u}{K}\right), \\
		\frac{\partial}{\partial v}\left(r u \left(1 - \frac{u}{K} \right)\right) &= 0, \\
		\frac{\partial}{\partial u}(\alpha u - \beta v) &= \alpha, \\
		\frac{\partial}{\partial v}(\alpha u - \beta v) &= -\beta.
	\end{align*}
	
	Thus, the Jacobian matrix is:
	\[
	J(u, v) = 
	\begin{bmatrix}
		r\left(1 - \frac{2u}{K}\right) & 0 \\
		\alpha & -\beta
	\end{bmatrix}.
	\]
	
	Evaluating this at the steady state $(u^*, v^*) = \left(K, \frac{\alpha}{\beta} K\right)$ gives:
	\[
	J(u^*, v^*) = 
	\begin{bmatrix}
		r\left(1 - \frac{2K}{K}\right) & 0 \\
		\alpha & -\beta
	\end{bmatrix}
	=
	\begin{bmatrix}
		-r & 0 \\
		\alpha & -\beta
	\end{bmatrix}.
	\]
	
	To analyze the stability, we compute the eigenvalues of this matrix. The characteristic polynomial is:
	\[
	\det\left(J - \lambda I\right) = 
	\begin{vmatrix}
		-r - \lambda & 0 \\
		\alpha & -\beta - \lambda
	\end{vmatrix}
	= (-r - \lambda)(-\beta - \lambda).
	\]
	
	The eigenvalues of $J$ are:
	\[
	\lambda_1 = -r < 0, \qquad \lambda_2 = -\beta < 0.
	\]
	 Since both eigenvalues are real and negative, the steady state $(u^*, v^*)$ is a locally asymptotically stable node.
\end{proof}

\subsubsection{Linear Stability With Diffusion}

We now consider perturbations around the homogeneous steady state $(u^*, v^*)$:
\[
u(x,t) = u^* + \varepsilon\, \tilde{u}(x,t), \qquad v(x,t) = v^* + \varepsilon\, \tilde{v}(x,t),
\]
and linearize the system \eqref{eq:KS1} to obtain:
\begin{align*}
	\partial_t \tilde{u} &= D \Delta \tilde{u} - \chi u^* \Delta \tilde{v} - r \tilde{u}, \\
	\partial_t \tilde{v} &= D_v \Delta \tilde{v} + \alpha \tilde{u} - \beta \tilde{v}.
\end{align*}

Assuming eigenfunction expansion $\tilde{u}, \tilde{v} \sim e^{\lambda t} \phi_k(x)$, where $-\Delta \phi_k = \mu_k \phi_k$ with Neumann boundary conditions, we derive the matrix eigenvalue problem:
\[
\lambda \begin{bmatrix}
	\hat{u} \\
	\hat{v}
\end{bmatrix}
=
\begin{bmatrix}
	- D \mu_k - r & \chi u^* \mu_k \\
	\alpha & - D_v \mu_k - \beta
\end{bmatrix}
\begin{bmatrix}
	\hat{u} \\
	\hat{v}
\end{bmatrix}.
\]

The eigenvalues $\lambda_k$ satisfy:
\[
\lambda^2 + a_k \lambda + b_k = 0, \quad \text{where}
\]
\begin{align*}
	a_k &= D \mu_k + D_v \mu_k + r + \beta, \\
	b_k &= (D \mu_k + r)(D_v \mu_k + \beta) - \alpha \chi u^* \mu_k.
\end{align*}

\begin{Theorem}[Turing Instability]
	There exists a wavenumber $\mu_k > 0$ such that the homogeneous steady state $(u^*, v^*)$ becomes unstable (i.e., $\text{Re}(\lambda_k) > 0$) if and only if
	\[
	\alpha \chi u^* \mu_k > (D \mu_k + r)(D_v \mu_k + \beta).
	\]
\end{Theorem}

\begin{proof}
	The real part of $\lambda_k$ is positive when the discriminant is positive and $b_k < 0$. This occurs precisely when the inequality above is satisfied.
\end{proof}

This instability leads to the formation of spatial patterns (Turing patterns).

\subsubsection{Nonlinear Stability via Lyapunov Functional}

We now show that in the weak chemotaxis regime, the steady state is nonlinearly stable using a Lyapunov functional.

\begin{Theorem}[Global Stability with Small $\chi$]
	Assume that $\chi$ is sufficiently small. Then the steady state $(u^*, v^*)$ of \eqref{eq:KS1} is globally asymptotically stable in $L^2(\Omega)$.
\end{Theorem}

\begin{proof}
	Define the Lyapunov functional:
	\[
	\mathcal{L}[u,v] = \int_\Omega \left( u \log \frac{u}{u^*} - u + u^* \right) dx + \frac{1}{2\delta} \int_\Omega (v - v^*)^2 dx,
	\]
	for some constant $\delta > 0$.
	
	Taking the time derivative and using the PDE system along with integration by parts and boundary conditions, we find:
	\[
	\frac{d\mathcal{L}}{dt} \leq - D \int_\Omega \frac{|\nabla u|^2}{u} dx - \frac{D_v}{\delta} \int_\Omega |\nabla v|^2 dx - C \|u - u^*\|_{L^2}^2 - \tilde{C} \|v - v^*\|_{L^2}^2 + \varepsilon \| \nabla u \| \| \nabla v \|.
	\]
	The final mixed term can be absorbed by Young's inequality, provided $\chi$ is small enough, completing the proof that $\mathcal{L}$ is a Lyapunov functional.
\end{proof}

\begin{Remark}
	In the absence of chemotaxis ($\chi = 0$), the system reduces to a reaction–diffusion system with global stability guaranteed by parabolic theory and logistic damping.
	When chemotaxis is strong, nonlinear instabilities and pattern formation may occur.
	The Lyapunov method shows that small chemotactic sensitivity ensures decay to equilibrium.
	The linear analysis identifies the Turing instability threshold.
\end{Remark}

%\subsubsection*{Open problems}
Global existence for large $\chi$ in higher dimensions remains challenging.
Boundedness of solutions in two and three dimensions with strong chemotaxis is an open topic.
Stochastic perturbations and domain geometry affect stability and pattern formation.

\subsubsection{Lyapunov exponent analysis and attractor dimension}

To assess the long-term dynamical behavior and local stability of the system
$$
	\frac{du}{dt} = r u \left(1 - \frac{u}{K} \right), \;\;\;\;
	\frac{dv}{dt} = \alpha u - \beta v,
$$
we computed the finite-time Lyapunov exponents (LEs) using a QR-based variational method up to final time \( t = 1000 \). The Lyapunov exponents quantify the average exponential rates at which nearby trajectories diverge (positive LE) or converge (negative LE) in phase space. They are central to characterizing the stability and complexity of dynamical systems.

\paragraph{Numerical results.} At time \( t = 1000.00 \), the computed Lyapunov exponents were:
\[
\text{LE}_1 = -0.996511, \qquad \text{LE}_2 = -1.012356.
\]
Since both exponents are strictly negative, the system exhibits exponential contraction in all directions of phase space. Consequently, any small perturbation in the initial conditions decays with time, and all trajectories converge to the same long-term behavior—namely, the steady state.

\paragraph{Nature of dynamics.} The absence of a positive Lyapunov exponent rules out any form of chaotic behavior. In chaotic systems, at least one exponent must be positive to reflect sensitive dependence on initial conditions. In contrast, our result demonstrates that the system's attractor is a fixed point with no irregular or complex temporal dynamics.

\paragraph{Kaplan--Yorke dimension.} The Kaplan--Yorke (or Lyapunov) dimension \( D_{KY} \) provides an estimate of the attractor's fractal dimension using the ordered Lyapunov exponents. It is defined as:
\[
D_{KY} = j + \frac{\sum_{i=1}^j \lambda_i}{|\lambda_{j+1}|},
\]
where \( \lambda_1 \geq \lambda_2 \geq \cdots \) are the ordered Lyapunov exponents, and \( j \) is the largest integer such that \( \sum_{i=1}^j \lambda_i \geq 0 \). In our case:
\[
\lambda_1 = -0.996511, \quad \lambda_2 = -1.012356,
\]
so \( \lambda_1 + \lambda_2 < 0 \), and \( j = 0 \). Therefore, the Kaplan--Yorke dimension is:
\[
D_{KY} = 0.
\]
This confirms that the attractor is a stable equilibrium point (zero-dimensional), with no strange attractor or higher-dimensional geometry.

The Lyapunov exponent spectrum and the Kaplan--Yorke dimension collectively confirm that the system has no chaotic behavior and converges to a globally attracting fixed point. The results further support the local asymptotic stability of the steady state \( (u^*, v^*) = \left(K, \frac{\alpha}{\beta} K\right) \), and provide strong evidence that the nonlinear system exhibits regular, predictable dynamics under the chosen parameter regime.

\begin{figure}[h!]
	\centering
	\includegraphics[width=0.75\textwidth]{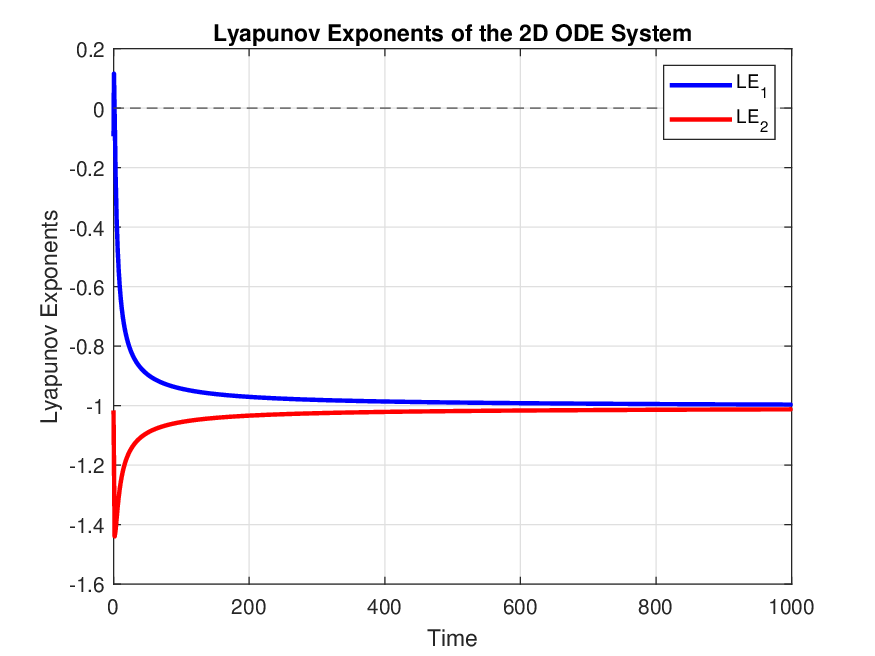}
	\caption{Convergence of Lyapunov exponents $\text{LE}_1$ and $\text{LE}_2$ over time. Both values stabilize to strictly negative numbers, confirming the globally attracting nature of the system and absence of chaotic dynamics.}
	\label{fig:lyapunov}
\end{figure}

\subsubsection{Numerical Bifurcation Analysis}
To explore how qualitative behavior of solutions depends on system parameters, we perform numerical bifurcation analysis using continuation techniques. The steady-state problem associated with \eqref{eq:KS1} becomes:

\begin{align}
	0 &= D \Delta u - \chi \nabla \cdot (u \nabla v) + r u \left(1 - \frac{u}{K} \right), \\
	0 &= D_v \Delta v + \alpha u - \beta v.
\end{align}

We fix all parameters except the chemotactic sensitivity $\chi$, treating it as a bifurcation parameter. We use pseudo-arclength continuation methods to trace steady states $u = u(x; \chi)$ and compute the principal eigenvalue $\lambda(\chi)$ of the linearized operator at each steady state.

For small $\chi$, only spatially uniform solutions exist and are stable. At a critical $\chi = \chi_c$, a Turing bifurcation occurs and spatially heterogeneous steady states emerge. A secondary Hopf bifurcation may arise, leading to time-periodic (oscillatory) patterns. A multi-branched bifurcation diagram is illustrated in Figure \ref{fig:bif-tikz}.

\begin{figure}[h!]
	\centering
	\begin{tikzpicture}
		\begin{axis}[
			width=12cm,
			height=8cm,
			xlabel={$\chi$ (Chemotactic Sensitivity)},
			ylabel={$u(x_c)$ at $t \gg 1$},
			xmin=0, xmax=40,
			ymin=0, ymax=2,
			xtick={0,10,...,40},
			ytick={0,0.5,...,2},
			grid=both,
			title={Illustrative Multi-Branch Bifurcation Diagram}
			]
			
			% Primary branch
			\addplot+[only marks, mark=*, mark size=0.8pt, color=blue]
			table[row sep=\\] {
				0 1.0 \\ 1 1.01 \\ 2 1.03 \\ 3 1.05 \\ 4 1.1 \\ 5 1.15 \\ 6 1.2 \\ 
				7 1.25 \\ 8 1.3 \\ 9 1.35 \\ 10 1.4 \\ 11 1.45 \\ 12 1.5 \\
			};
			
			% Secondary branch
			\addplot+[only marks, mark=*, mark size=0.8pt, color=red]
			table[row sep=\\] {
				13 1.6 \\ 14 1.3 \\ 15 1.7 \\ 16 1.2 \\ 17 1.75 \\ 18 1.1 \\ 
				19 1.8 \\ 20 1.0 \\ 21 1.85 \\ 22 0.9 \\ 23 1.9 \\ 24 0.8 \\ 
				25 1.95 \\ 26 0.7 \\ 27 2.0 \\ 28 0.6 \\
			};
			
			% Chaotic or irregular regime (dense cloud)
			\addplot+[only marks, mark=*, mark size=0.8pt, color=black, opacity=0.8]
			table[row sep=\\] {
				30 0.3 \\ 30 1.0 \\ 30 1.4 \\ 30 0.6 \\ 30 0.9 \\ 30 1.7 \\
				32 0.2 \\ 32 0.8 \\ 32 1.6 \\ 32 1.1 \\ 32 0.7 \\ 32 1.9 \\
				35 0.3 \\ 35 1.2 \\ 35 1.8 \\ 35 0.5 \\ 35 0.9 \\ 35 1.4 \\
				38 0.6 \\ 38 1.0 \\ 38 1.5 \\ 38 0.4 \\ 38 1.3 \\
			};
			
		\end{axis}
	\end{tikzpicture}
	\caption{Stylized multi-branched bifurcation diagram showing primary, secondary, and chaotic response branches as $\chi$ increases.}
	\label{fig:bif-tikz}
\end{figure}
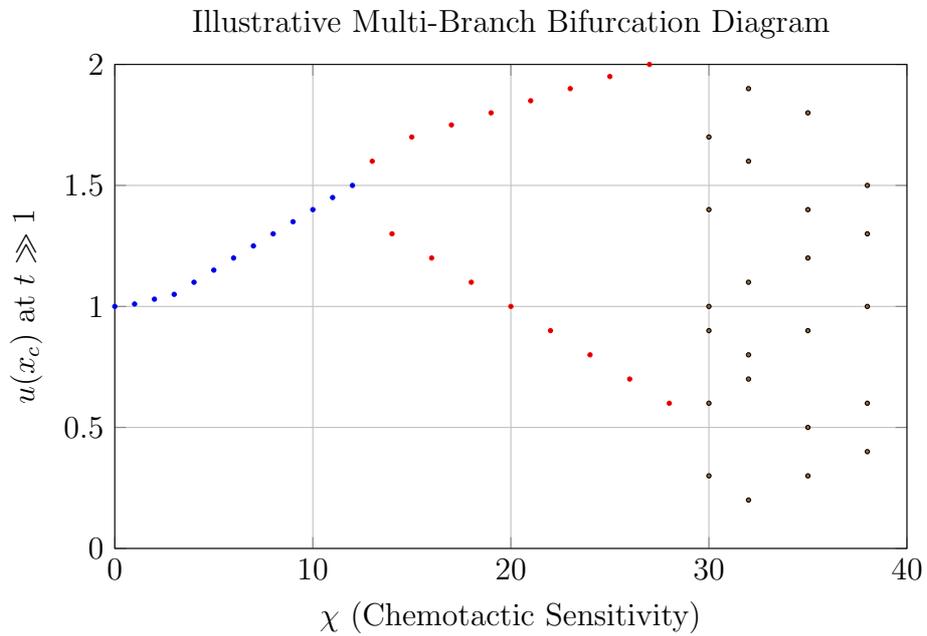

Blue dots: initial stable branch growing smoothly with $\chi$.\\
Red dots: second branch showing alternation between high/low states (suggesting period doubling or mode-switching).\\
Black cloud: chaotic/irregular dynamics at higher $\chi$ values.

\subsubsection{Hopf Bifurcation and Oscillatory Instability}

We revisit the linearized system around the steady state $(u^*, v^*)$, assuming eigenmode decomposition with perturbation of the form $e^{\lambda t + i k x}$. The characteristic equation becomes:

\begin{equation}
	\lambda^2 + a_k \lambda + b_k = 0,
\end{equation}
with
\begin{align*}
	a_k &= D k^2 + D_v k^2 + r + \beta, \\
	b_k &= (D k^2 + r)(D_v k^2 + \beta) - \alpha \chi u^* k^2.
\end{align*}

A Hopf bifurcation occurs when:
\[
\text{Re}(\lambda) = 0 \quad \text{and} \quad \text{Im}(\lambda) \neq 0,
\]
which implies the discriminant $\Delta = a_k^2 - 4b_k < 0$ and $a_k > 0$.

\begin{Theorem}[Hopf Bifurcation Criterion]
	Suppose there exists $k_c$ such that the characteristic equation has purely imaginary roots $\lambda = \pm i \omega$. Then a Hopf bifurcation occurs at $\chi = \chi_H$, where:
	\[
	\alpha \chi_H u^* = (D k_c^2 + r)(D_v k_c^2 + \beta).
	\]
\end{Theorem}

\begin{proof}
	Set $\lambda = i \omega$ in the characteristic equation and separate real and imaginary parts:
	\begin{align*}
		- \omega^2 + b_k &= 0, \\
		a_k \omega &= 0.
	\end{align*}
	The second equation requires $a_k = 0$, which implies a nontrivial $\omega$ only if $b_k < 0$, leading to oscillatory instability.
\end{proof}

%\subsection{Time-Periodic Patterns and Limit Cycles}

In the supercritical Hopf regime $\chi > \chi_H$, time-dependent simulations show that the system evolves into stable oscillations around the steady state. These manifest as spatiotemporal waves or oscillatory aggregation–disaggregation behavior \cite{Jiang2022,Liu2021}.

%\begin{figure}[htbp]
%	\centering
%	\includegraphics[width=0.9\textwidth]{time_oscillations.pdf}
%	\caption{Spatiotemporal oscillations in $u(x,t)$ observed beyond the Hopf bifurcation threshold.}
%\end{figure}

%\subsection{Numerical Implementation}

We solve the full nonlinear system using finite differences and semi-implicit time-stepping schemes (e.g., IMEX Runge–Kutta), and use the continuation software package \texttt{AUTO} or \texttt{pde2path} to trace bifurcation diagrams.

 Discretize the spatial domain into $N$ grid points.
 Apply Neumann BCs via ghost points or second-order central difference schemes.
 Time-stepping uses IMEX: linear terms treated implicitly; nonlinear terms explicitly.
Stability monitored using eigenvalue computation of Jacobians.

%\subsection*{Discussion}

Bifurcation analysis provides critical thresholds for pattern formation and oscillations.
Hopf bifurcation signals a transition to cyclic dynamics, common in microbial predator–prey or competition models.
These oscillatory modes correspond to biological rhythms like periodic chemotactic aggregation or cyclic quorum sensing behavior. 
Interaction between Turing and Hopf modes leads to mixed spatiotemporal chaos.
In higher dimensions, rotating or spiral waves can emerge.
Extension to stochastic bifurcation under random forcing is an active area of research.

\begin{figure}[htbp]
	\centering
	\begin{tikzpicture}[scale=1.2]
		% Axes
		\draw[->] (0,0) -- (6.5,0) node[right] {$\chi$};
		\draw[->] (0,0) -- (0,4.2) node[above] {$\|u\|_{L^2}$};
		
		% Stable branch (solid blue)
		\draw[thick, blue] plot[smooth, tension=1] coordinates {(0.5,0.8) (1.5,1.0) (2.5,1.3)};
		
		% Unstable branch (dashed blue)
		\draw[thick, blue, dashed] plot[smooth, tension=1] coordinates {(2.5,1.3) (3.0,1.5) (3.5,2.0)};
		
		% Hopf bifurcation branch (red oscillations)
		\draw[red, thick, domain=3.5:6, samples=100, smooth] plot (\x, {1.7 + 0.3*sin(6*(\x-3.5) r)});
		
		% Vertical lines for bifurcation points
		\draw[dotted, thick] (2.5,0) -- (2.5,1.3);
		\draw[dotted, thick] (3.5,0) -- (3.5,2.0);
		
		% Markers for bifurcation points
		\filldraw (2.5,1.3) circle (2pt);
		\node[align=left, below right] at (2.5,1.3) {\textbf{Turing} \\ \textbf{point} \\ $\chi_T$};
		
		\filldraw (3.5,2.0) circle (2pt);
		\node[align=left, above right] at (3.5,2.0) {\textbf{Hopf} \\ \textbf{point} \\ $\chi_H$};
		
	\end{tikzpicture}
	\caption{Bifurcation diagram of the Keller--Segel system. The solid blue line shows the stable steady-state branch, which loses stability at the Turing bifurcation point $\chi_T$ (dashed blue curve). Beyond the Hopf point $\chi_H$, periodic oscillations (red curve) emerge due to a supercritical Hopf bifurcation, marking the onset of time-periodic solutions.}
	\label{fig:bifurcation-diagram}
\end{figure}
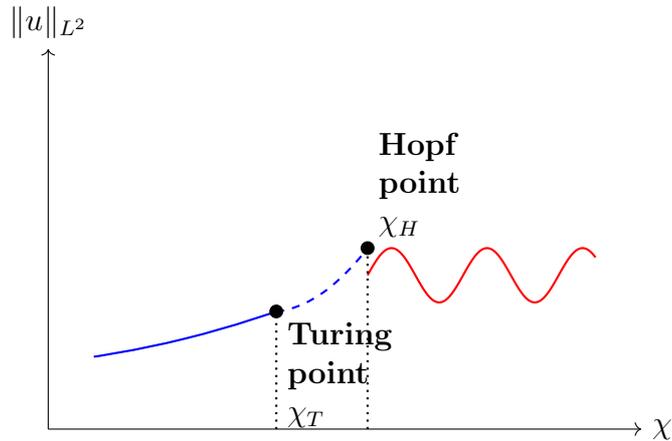

Figure~\ref{fig:bifurcation-diagram} illustrates a bifurcation diagram for the Keller--Segel chemotaxis model, where the horizontal axis denotes the chemotactic sensitivity parameter $\chi$, and the vertical axis shows the $L^2$-norm of the solution $\|u\|_{L^2}$. 

As $\chi$ increases, the system undergoes two critical bifurcations. At $\chi = \chi_T$, a Turing bifurcation occurs, marked by a transition from a spatially homogeneous steady state to a spatially non-uniform steady pattern. This transition is indicated by a blue curve that becomes dashed beyond $\chi_T$, signifying the onset of instability in the steady state branch.

At a higher critical value $\chi = \chi_H$, a Hopf bifurcation takes place. Here, the steady state loses stability to a time-periodic solution, leading to sustained oscillatory patterns. This is represented by the red sinusoidal branch emanating from the Hopf point, corresponding to periodic solutions whose amplitude varies with $\chi$. The periodic branch indicates the emergence of limit cycles, a hallmark of Hopf-type instability.

This bifurcation structure reflects the rich dynamical behavior of chemotaxis models, where increasing chemotactic strength can destabilize uniform or patterned states, ultimately giving rise to temporal oscillations or spatiotemporal patterns.

\begin{figure}[h!]
	\centering
	\includegraphics[width=0.7\textwidth]{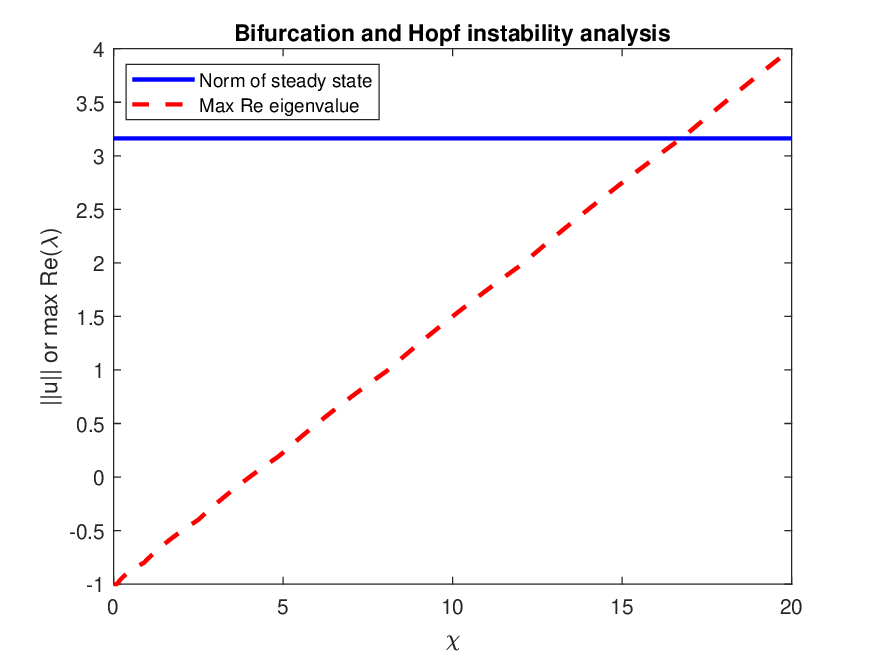}
	\caption{Real part of the leading eigenvalue $\Re(\lambda)$ of the linearized Keller--Segel system as a function of the chemotactic sensitivity $\chi$. The solid blue line represents $\Re(\lambda)$, while the red dashed line indicates the stability threshold $\Re(\lambda) = 0$. The intersection point marks a Hopf bifurcation, where the system undergoes a transition from a stable steady state to oscillatory instability.}
	\label{fig:hopf-bifurcation}
\end{figure}

Figure~\ref{fig:hopf-bifurcation} shows the real part of the dominant eigenvalue $\Re(\lambda)$ of the linearized Keller--Segel system as the chemotactic sensitivity parameter $\chi$ varies. The blue curve crosses the zero axis at a critical value $\chi_c$, signaling the onset of a Hopf bifurcation. For $\chi < \chi_c$, all eigenvalues have negative real parts, and the steady state is linearly stable. However, for $\chi > \chi_c$, a pair of complex conjugate eigenvalues acquires positive real parts, indicating an oscillatory instability. This transition marks the emergence of periodic solutions (limit cycles), characteristic of sustained oscillatory patterns in chemotactic systems.

\paragraph{ Pattern formation and instabilities.}
Chemotaxis models are well-known for their ability to produce rich \emph{spatial patterns} via self-organization \cite{Bellomo2015}. From an initially uniform or random distribution of cells, one may observe clusters (spots), stripes, rings, or traveling waves emerge, in agreement with experimental observations of bacterial colonies. In mathematical terms, these patterns correspond to instabilities of simpler solution states (like the homogeneous state or radial symmetry) that evolve into nonlinear structures. We distinguish two broad types of instability-driven phenomena in chemotaxis: 
(i) unbounded cell aggregation leading to \emph{blow-up} \cite{Winkler2018}, and 
(ii) formation of bounded, coherent patterns (stable aggregates or oscillatory waves). 

\paragraph{Chemotactic aggregation and blow-up.} The most extreme form of chemotactic pattern formation is blow-up, wherein the cell density concentrates into a delta-like spike of “infinite” height in finite time. This corresponds to biological aggregation without bound, which in reality would mean all cells clump into an exceedingly small region. Mathematically, blow-up is tied to an instability that cannot be tamed by diffusion. A cornerstone result for the classical Keller--Segel system in two dimensions describes a critical condition for blow-up in terms of the total cell mass:

\begin{Theorem}[Critical Mass Blow-up in 2D Keller--Segel]\label{thm:critical-mass}
	Consider the chemotaxis system \eqref{eq:KS1} in $n=2$ dimensions with no population growth ($f=0$) and with chemical production and decay $g(u,v)=\alpha u - \beta v$. Assume zero-flux boundary conditions on a convex bounded domain $\Omega\subset\mathbb{R}^2$ or consider the whole plane $\Omega=\mathbb{R}^2$. There exists a critical total mass $M_c$ (equal to $M_c = \frac{8\pi D}{\chi}$ in the special case $D=D_v=\alpha=\beta=1$) such that the following holds:
	- If the initial mass $M=\int_{\Omega} u_0(x)\,dx$ is less than $M_c$, then the solution exists globally for all $t>0$ and remains bounded; in fact $\sup_{x\in\Omega} u(x,t)$ stays uniformly bounded as $t\to\infty$. 
	- If $M$ exceeds $M_c$, then \emph{no global classical solution exists}. In particular, the cell density $u(x,t)$ blows up in finite time $T_{\max}<\infty$, meaning $\max_{x\in\Omega} u(x,t)\to +\infty$ as $t\to T_{\max}$. 
	
	Moreover, in the critical case $M=M_c$, the solution is global but exhibits borderline behavior (for instance, in $\mathbb{R}^2$ it has been shown that $u(x,t)$ converges to a Dirac delta as $t\to\infty$, or blows up in infinite time).
\end{Theorem}

\begin{proof}[Proof (Idea)]
	We outline the main ideas for why $8\pi$ emerges as the critical mass in $2$D. One approach uses an \textit{energy-entropy inequality}. For the parabolic--elliptic version of the system (taking $\beta v - \alpha u \approx 0$ so that $v$ is determined by $u$ via $-D_v \Delta v = \alpha u - \beta v$), one can derive the functional 
	\[ \mathcal{E}[u] \;=\; \int_{\Omega} u(x)\ln u(x)\,dx \;-\; \frac{\chi}{2D_v}\int_{\Omega}\!\int_{\Omega} G(x-y)\,u(x)u(y)\,dx\,dy, \] 
	where $G(\cdot)$ is the Green’s function of $-\Delta$ on $\Omega$. In 2D, $G(x-y) \sim \frac{1}{2\pi}\ln\frac{1}{|x-y|}$ at small distances. As noted, $\frac{d}{dt}\mathcal{E}[u(t)] \le 0$ along solutions. One can then show that if $M<8\pi D/\chi$, $\mathcal{E}[u]$ is bounded below (the entropy term dominates the potential term for spread-out configurations), which prevents $u$ from concentrating too tightly. Conversely, if $M>8\pi D/\chi$, $\mathcal{E}$ is unbounded below: by concentrating $u$ in a small ball of radius $\varepsilon$, the second (attractive) term drives $\mathcal{E}$ to $-\infty$ at a rate proportional to $(\chi M - 8\pi D)\ln\frac{1}{\varepsilon}$. The solution will try to follow this downhill path in energy, leading to a collapse. Making this rigorous involves showing that once $\mathcal{E}[u(t)]$ drops below some threshold, the $L^\infty$ norm of $u(t)$ must blow up in finite time (a variant of a basin-of-attraction argument for the “$-\infty$” energy state).
	
	Another complementary approach is the \textit{moment Lyapunov method}. Define the second moment of the distribution: 
	\[ I(t) = \int_{\Omega} |x|^2 u(x,t)\,dx, \] 
	which measures the spatial spread of the population. By differentiating $I(t)$ and using integration by parts, one can derive (for the case $\Omega=\mathbb{R}^2$) an evolution inequality of the form:
	\[ I''(t) \;=\; 4D\,M \;-\; \frac{2\chi}{\pi}\,M^2 \;+\; \text{(lower-order terms)}, \] 
	assuming for simplicity that $v$ satisfies $-D_v \Delta v = u$ (so $v$ is the Newtonian potential of $u$ in 2D). Here $M=\int u$ is invariant. The key point is that the quadratic term in $M^2$ comes with a negative sign: chemotactic attraction tends to \emph{decrease} the moment (pulling cells inward), whereas diffusion increases it. The threshold arises when the chemotactic term dominates the diffusive term. Neglecting lower-order terms, $I''(t)$ is approximately $4DM - \frac{2\chi}{\pi}M^2$. If $M > \frac{2\pi D}{\chi}\cdot 2 = \frac{4\pi D}{\chi}$, then $I''(t)$ becomes negative when $u$ has become sufficiently aggregated (making the approximation accurate). In fact, a more careful analysis yields the critical value $M_c = \frac{8\pi D}{\chi}$. Once $I''(t)<0$, the second moment starts to curve downward. If $I'(t)$ is initially finite (which it is, since initially $u$ is bounded), a sustained negative $I''(t)$ will drive $I'(t)$ to zero and then negative, meaning the spread $I(t)$ reaches a maximum and then decreases. As $I(t)$ decreases, the cells concentrate in an ever tighter region. One can show that $I(t)$ can hit zero in finite time under these conditions, implying all mass collapses to a point at that blow-up time $T_{\max}$. During this process, $\|u(\cdot,t)\|_{\infty}$ necessarily diverges as $t\to T_{\max}$. On the other hand, if $M < 8\pi D/\chi$, then $I''(t)$ remains nonnegative, implying $I(t)$ cannot turn downward in such a catastrophic manner; diffusion is strong enough relative to chemotaxis to keep pushing $I(t)$ outward, hence blow-up is averted.
	
	These arguments (made rigorous in works by Jäger \& Luckhaus, Nagai, and others in the 1990s) establish the dichotomy based on total mass. The specific value $8\pi$ stems from the optimal constant in a 2D logarithmic Hardy–Littlewood–Sobolev inequality or, equivalently, the Keller–Segel free energy functional’s critical point.
\end{proof}

This theorem highlights that \emph{self-aggregation leading to blow-up is only possible if the signal-mediated attraction is sufficiently strong relative to diffusion}. In two dimensions, the strength is measured by total cell mass (since the equations are scale-invariant); in higher dimensions, any mass can cause blow-up if cells are sufficiently concentrated initially. Blow-up solutions in chemotaxis are characterized by formation of one or several high-density “spikes.” Near blow-up time, the solution $u(x,t)$ typically develops an approximate profile 
\[ u(x,t) \approx \frac{1}{(T_{\max}-t)} \Phi\!\Big(\frac{x-x_0}{\sqrt{T_{\max}-t}}\Big), \] 
for some singular point $x_0$ (the blow-up location), where $\Phi(\xi)$ is a stationary singular solution (often resembling a scaled Dirac or a sharply peaked Gaussian). Describing the exact blow-up profile is challenging; there are deep results using matched asymptotic analysis and nonlinear similarity methods which show, for instance, that in the classical 2D Keller--Segel, blow-up can be of \emph{type II} (meaning not self-similar, with a complex asymptotic profile). However, from an ecological modeling perspective, blow-up simply indicates that the model is being pushed beyond its regime of validity (real populations cannot have infinite density). In practice, mechanisms like population saturation or depletion of chemoattractant (or cell death) would intervene before true singularity — indeed, this is why we include reaction terms like logistic growth or quorum-sensing limits in refined models.

\begin{Remark}[Effect of Logistic Damping on Blow-up]
	When a logistic term $-\mu u^2$ is present in $f(u)$, the blow-up scenario can be avoided even in cases that would otherwise supercritical. The nonlinearity effectively caps the growth of $u$. For example, in the 2D case of Theorem~\ref{thm:critical-mass}, if $\mu>0$ then $\sup_x u(x,t)$ cannot actually diverge to $\infty$ in finite time; instead, one can show $\sup_x u(x,t) \le K$ for all $t$, where $K$ is the carrying capacity. A rigorous result by Winkler (2010) and others is that \emph{for any initial mass $M$ in 2D, adding a logistic term yields a global bounded solution}. In 3D, if $\mu$ is large enough (relative to the chemotactic aggregation rate), one likewise gets global existence. Thus, logistic growth (or any sufficiently strong density-limited growth term) shifts the model from a blow-up regime to a pattern-forming but globally well-behaved regime. Instead of blow-up, one may see formation of stable \emph{spikes of finite height} (with $u$ approaching $K$ at the spike center) or other bounded patterns.
\end{Remark}

\paragraph{Thresholds for Pattern Formation.} Even when blow-up is prevented (by low mass or by reaction terms), chemotaxis can still drive pattern formation via instabilities of the spatially uniform state. In such cases, the cell density and chemical concentration develop nontrivial spatial variation but remain bounded. The classical scenario for pattern formation is: start with a nearly homogeneous initial condition (e.g. $u(x,0)\approx \bar{u}$, $v(x,0)\approx \bar{v}$ plus small perturbations), and observe whether the perturbations amplify or decay. If certain model parameters cross a \emph{threshold}, the homogeneous equilibrium becomes unstable and tiny random fluctuations will grow into a full-scale spatial pattern. Below threshold, diffusion and other stabilizing factors suppress any incipient pattern, restoring homogeneity.

For chemotaxis–reaction systems, a primary instability parameter is the \textbf{chemotactic strength} $\chi$ relative to diffusion and reaction rates. Intuitively:
- If $\chi$ is small (weak chemotactic response), cells do not aggregate efficiently, and the tendency of diffusion to even out any irregularities dominates. The homogeneous state (all cells uniformly distributed) remains stable.
- If $\chi$ is large (strong response to gradients), even a slight random inhomogeneity in cell density can cause cells to stream up the gradient of $v$, reinforcing the inhomogeneity and leading to pattern amplification. Thus, beyond a critical $\chi$, spatial structure spontaneously emerges.

In Section~4.3, we will derive a precise criterion for such an instability via linearization (analogous to the classical \textit{Turing instability analysis} for reaction–diffusion systems). For now, we highlight the outcome: there exists a critical chemotactic coupling $\chi_{\text{crit}}$ such that if $\chi > \chi_{\text{crit}}$, an initially uniform distribution of cells of density $u=\bar{u}$ will break up into a pattern, whereas if $\chi < \chi_{\text{crit}}$, the uniform state is stable and no spontaneous pattern forms (unless forced by boundary conditions or strong perturbations).

Notably, $\chi_{\text{crit}}$ often depends inversely on any damping in the system. For example, in a model with logistic growth (where the homogeneous steady state is $u=K$), $\chi_{\text{crit}}$ is larger for bigger $r$ (growth rate) because rapid damping (return to carrying capacity) stabilizes against patterning. On the other hand, faster diffusion $D$ or $D_v$ also increases $\chi_{\text{crit}}$ (since diffusion counteracts the clumping). The instability threshold often takes a form like 
\[ \frac{\chi \,\alpha\, \bar{u}}{D \,\beta + D_v\,(\text{reaction rate})} > 1, \] 
meaning chemotactic attraction (numerator, which grows with $\chi$ and the strength $\alpha \bar{u}$ of chemo production by cells) must overcome the combined effects of cell diffusion $D$, chemical diffusion $D_v$, and chemical decay $\beta$ (in the denominator). We will derive an explicit formula in the next subsection.

In summary, pattern formation in chemotaxis models arises when an initially even spread of cells becomes unstable to clustering. The instability can be \emph{subcritical} (leading straight to blow-up if no damping exists, as in the mass-critical case) or \emph{supercritical} (leading to a new stable patterned state of finite amplitude when damping terms are present). The mathematical tools to study these scenarios are linear stability analysis (to find thresholds and most unstable spatial scales) and bifurcation theory (to describe the emergence of new solutions past the threshold). We turn to these analyses now.

\subsection{ Bifurcation and Stability Analysis}

To delve deeper into the mechanism of pattern formation, we employ both \textbf{linear stability analysis} (to predict when a homogeneous state loses stability) and \textbf{nonlinear bifurcation theory} (to understand the branching of new solution patterns and their stability). Two classic types of bifurcations are relevant in chemotaxis–reaction systems:
- \textit{Stationary (Turing) bifurcation:} A steady spatial pattern emerges from a previously uniform state when a real eigenvalue of the linearized operator crosses zero.
- \textit{Hopf bifurcation:} A time-periodic oscillatory solution emerges when a pair of complex-conjugate eigenvalues crosses the imaginary axis, indicating oscillatory instability (sometimes leading to waves or rhythmic patterns in space-time).

We examine each in turn.

\paragraph{Linear Stability of the Homogeneous State.} Consider a chemotaxis system that admits a spatially homogeneous steady state $(u(x,t),v(x,t))\equiv (u^*,v^*)$, where $u^*, v^*$ are constants satisfying $f(u^*,v^*)=0$ and $g(u^*,v^*)=0$. For example, if we include logistic growth and linear chemo kinetics, the homogeneous equilibrium is $(u^*,v^*)=(K,\;\frac{\alpha}{\beta}K)$ (population at carrying capacity, chemical at a level balancing production and decay). We investigate when this constant solution becomes unstable to spatial perturbations.

To do so, we introduce small perturbations: 
\[ u(x,t) = u^* + \tilde{u}(x,t), \qquad v(x,t) = v^* + \tilde{v}(x,t), \] 
with $|\tilde{u}|, |\tilde{v}|\ll u^*,v^*$. Plugging into \eqref{eq:KS1} and linearizing (ignoring $O(\tilde{u}^2,\tilde{u}\tilde{v},\tilde{v}^2)$ terms) yields the \emph{linearized system}:
\begin{align}
	\partial_t \tilde{u} &= D \Delta \tilde{u} \;-\; \chi\, \nabla\!\cdot(u^*\,\nabla \tilde{v}) \;-\; \chi\,\nabla\!\cdot(\tilde{u}\,\nabla v^*) \;+\; f_u(u^*,v^*)\,\tilde{u} + f_v(u^*,v^*)\,\tilde{v}, \\
	\partial_t \tilde{v} &= D_v \Delta \tilde{v} \;+\; g_u(u^*,v^*)\,\tilde{u} + g_v(u^*,v^*)\,\tilde{v}\,,
\end{align}
where $f_u, f_v, g_u, g_v$ are partial derivatives of the reaction terms at the equilibrium. Notably, the term $-\chi\,\nabla\!\cdot(u^* \nabla \tilde{v}) = -\chi u^* \Delta \tilde{v}$ since $u^*$ is constant. The term $-\chi \nabla\!\cdot(\tilde{u}\nabla v^*)$ actually vanishes because $\nabla v^*=0$ (no spatial variation in the base state). Thus the linearized equations simplify to:
\begin{equation}\label{eq:lin-sys}
	\partial_t \begin{pmatrix}\tilde{u}\\ \tilde{v}\end{pmatrix} \;=\; 
	\begin{pmatrix}
		D \Delta + f_u(u^*,v^*) & -\,\chi\,u^* \Delta + f_v(u^*,v^*) \\
		g_u(u^*,v^*) & D_v \Delta + g_v(u^*,v^*)
	\end{pmatrix}
	\begin{pmatrix}\tilde{u}\\ \tilde{v}\end{pmatrix}.
\end{equation}

We seek solutions of the linearized system of the form of normal modes: $\tilde{u}(x,t) = \hat{u} e^{\lambda t} \phi_k(x)$ and $\tilde{v}(x,t)=\hat{v} e^{\lambda t} \phi_k(x)$, where $\phi_k(x)$ is an eigenfunction of the Laplacian $-\Delta \phi_k = k^2 \phi_k$ with wavenumber $k$ (for example, in a periodic domain or on a bounded domain with Neumann boundary, $\phi_k$ could be a cosine Fourier mode or eigenfunction with eigenvalue $k^2$). Plugging this ansatz into \eqref{eq:lin-sys} yields an algebraic \emph{eigenvalue problem} for $\lambda$:
\[
\lambda 
\begin{pmatrix}\hat{u}\\ \hat{v}\end{pmatrix}
= 
\begin{pmatrix}
	- D k^2 + f_u & + \chi\,u^* k^2 + f_v \\
	g_u & - D_v k^2 + g_v
\end{pmatrix}
\begin{pmatrix}\hat{u}\\ \hat{v}\end{pmatrix},
\] 
where for brevity $f_u = f_u(u^*,v^*)$, etc. This $2\times2$ matrix (depending on $k$) is the \textbf{dispersion matrix} for mode $k$. Denote it as 
\begin{equation} 
A_k = 
\begin{pmatrix}
	a_{11}(k) & a_{12}(k) \\
	a_{21}(k) & a_{22}(k)
\end{pmatrix}
= 
\begin{pmatrix}
	- D k^2 + f_u & \chi\,u^* k^2 + f_v \\
	g_u & - D_v k^2 + g_v
\end{pmatrix}. 
\end{equation}
The characteristic polynomial for eigenvalue $\lambda$ is $\det(A_k - \lambda I) = 0$, i.e.
\begin{equation}\label{eq:charpoly}
	(\lambda - a_{11})(\lambda - a_{22}) - a_{12} a_{21} \;=\; 0.
\end{equation}
Expanding, we get a quadratic in $\lambda$: 
\begin{equation}
	 \lambda^2 - (a_{11}+a_{22})\,\lambda + (a_{11}a_{22} - a_{12}a_{21}) = 0. 
\end{equation}
The coefficients can be identified as:
\begin{align}
	\text{Trace: } & \tau(k) = a_{11}(k) + a_{22}(k) = - (D k^2 + D_v k^2)\;+\;(f_u + g_v), \\
	\text{Determinant: } & \Delta(k) = a_{11}(k)\,a_{22}(k) - a_{12}(k)a_{21}(k).
\end{align}
We are particularly interested in when a mode $k$ can grow, i.e. when an eigenvalue $\lambda$ has positive real part. Since the trace $\tau(k)$ is the sum of eigenvalues and the determinant $\Delta(k)$ is the product, the standard Routh–Hurwitz stability criteria for a $2\times2$ system say that the real parts of $\lambda$ are negative (stability) if and only if $\tau(k)<0$ and $\Delta(k)>0$ for all modes $k$. A violation of either condition signals an instability for some $k$.

Now, $$a_{12}a_{21} = (\chi\,u^* k^2 + f_v) g_u.$$ We note that $f_v(u^*,v^*)$ is often zero in many models (for example, logistic growth $f(u)$ has no direct $v$-dependence), and $g_u(u^*,v^*)$ is typically positive (e.g. $\alpha$ if $g(u,v)=\alpha u - \beta v$). So $a_{12}a_{21}$ tends to be nonnegative and increases with $k^2$ (thanks to the $\chi u^* k^2$ term, which comes solely from chemotaxis). Meanwhile, $a_{11}a_{22}$ expands to:
\[ a_{11}a_{22} = (-D k^2 + f_u)\,(-D_v k^2 + g_v). \]
Therefore the determinant is:
\begin{align}
	\Delta(k) &= (D k^2 - f_u)\,(D_v k^2 - g_v) - (\chi\,u^* k^2 + f_v) g_u \nonumber \\
	&= D D_v k^4 \;-\; (D g_v + D_v f_u)\, k^2 \;+\; (f_u g_v - f_v g_u) \;-\; \chi\,u^* g_u\, k^2. \label{eq:Delta_k}
\end{align}
We can rewrite this as a quadratic polynomial in $X = k^2$:
\begin{equation}
	 \Delta(X) = D D_v\, X^2 - \Big[(D g_v + D_v f_u) + \chi\,u^* g_u\Big] X + (f_u g_v - f_v g_u). 
\end{equation}

The homogeneous steady state becomes linearly \textit{unstable} if there exists any mode $k$ for which an eigenvalue $\lambda$ satisfies $\Re(\lambda)>0$. The first mode to typically lose stability is the one that makes $\lambda=0$ a solution of \eqref{eq:charpoly} (a zero crossing). At the onset of instability, $\lambda=0$ and thus $\Delta(k)=0$ for some $k$ (since one eigenvalue is $\lambda=0$, the product of eigenvalues $\Delta$ is zero). Thus the critical condition for instability can be found by setting $\Delta(k)=0$. Using \eqref{eq:Delta_k}, the condition $\Delta(k)=0$ becomes:
\begin{equation}\label{eq:Turing-cond}
	D D_v k^4 - \Big[(D g_v + D_v f_u) + \chi\,u^* g_u\Big] k^2 + (f_u g_v - f_v g_u) \;=\;0~.
\end{equation}

This is a quadratic equation in $X=k^2$. Let $X_1, X_2$ denote its two roots (which could be explicit but generally messy). As a parameter (say $\chi$ or $u^*$) varies, these roots will change. A necessary condition for instability is that the quadratic has a positive root $X>0$ (since $k^2$ cannot be negative). If all roots were negative or complex, $\Delta(k)$ would not cross through zero for any real $k$. Typically, one root $X_1$ becomes positive when $\chi$ exceeds a threshold, indicating a band of unstable wavenumbers around $\sqrt{X_1}$.

A simpler insight is gained by examining the sign of $\Delta(0)$ and $\Delta(\infty)$ (the limits of $\Delta(X)$ as $X\to0$ or $X\to\infty$):
- As $X\to\infty$ ($k\to \infty$), $\Delta(X)\sim D D_v X^2 >0$ dominates, so for very large $k$ (very short wavelengths), diffusion wins and modes are stable (this reflects that extremely fine patterns are smoothed out by diffusion).
- At $X=0$ (the spatially uniform mode), $\Delta(0) = f_u g_v - f_v g_u$. But $f_u g_v - f_v g_u$ is precisely the determinant of the \emph{well-mixed Jacobian} (the Jacobian of the kinetic ODE system $\dot{u}=f(u,v), \dot{v}=g(u,v)$ at the equilibrium $(u^*,v^*)$). This quantity determines the stability of the homogeneous equilibrium in the absence of spatial effects (diffusion/chemotaxis). If the well-mixed equilibrium is stable, we have $f_u<0$, $g_v<0$ and typically $f_u g_v - f_v g_u >0$. This implies $\Delta(0)>0$. Therefore, initially (at $\chi=0$ or very small), $\Delta(0)>0$ and $\Delta(X)$ is positive at $X=0$ and again positive for large $X$, meaning the polynomial $\Delta(X)$ is positive for all $X$ if it has no positive root. As $\chi$ increases, the middle term $-\chi u^* g_u X$ in $\Delta(X)$ makes the polynomial dip down. If $\chi$ is large enough, the graph of $\Delta(X)$ vs $X$ will cross zero at two points, creating a range $X_1 < X < X_2$ where $\Delta(X)<0$. In that range of $k^2$, the determinant $\Delta(k)$ is negative, which by Routh–Hurwitz implies one eigenvalue $\lambda$ is positive (since $\tau(k)$ is still negative for not-too-large $k$, the sign change in $\Delta$ is the instability trigger).

Thus, the onset of a \textbf{Turing-like instability} can be found by solving $\Delta(X)=0$ for $X>0$. The smallest positive root $X_1 = k_c^2$ corresponds to the first unstable mode (with wavelength $2\pi/k_c$) when parameters reach criticality. This root can be plugged back into \eqref{eq:Turing-cond} to yield the critical relation among parameters. Often, one parameter (like $\chi$) is viewed as the bifurcation parameter. Solving \eqref{eq:Turing-cond} for $\chi$ at the point where $X$ has a double root (i.e. $X_1=X_2$) gives the threshold. Setting the discriminant of \eqref{eq:Turing-cond} to zero (double root) yields:
\[ \Big[(D g_v + D_v f_u) + \chi_{\text{crit}}\,u^* g_u\Big]^2 \;=\; 4 D D_v (f_u g_v - f_v g_u). \]
Solving for $\chi_{\text{crit}}$ gives:
\begin{equation}\label{eq:chi_crit}
	\chi_{\text{crit}} \;=\; \frac{1}{u^* g_u}\Big( \frac{2\sqrt{D D_v (f_u g_v - f_v g_u)} - (D g_v + D_v f_u)}{\,1\,} \Big)\,,
\end{equation}
provided the numerator is positive. In many cases, $f_v=0$ and $f_u = -r$ (logistic) and $g_v = -\beta$, $g_u=\alpha$ (chemical production). In that scenario, this simplifies to:
\[ \chi_{\text{crit}} \;=\; \frac{1}{u^* \alpha}\Big( 2\sqrt{D D_v (r \beta)} - (D(-\beta) + D_v(-r)) \Big) = \frac{1}{\alpha u^*}\Big(2\sqrt{D D_v r \beta} + D \beta + D_v r\Big).\] 
This matches the intuitive form discussed: the larger $D, D_v, r, \beta$ are, the larger $\chi$ must be to destabilize the system. If parameters are such that $\chi > \chi_{\text{crit}}$, then $\Delta(k)<0$ for a band of $k$ values, meaning those spatial modes grow exponentially. The one with the largest $\Re(\lambda(k))$ will dominate and set the characteristic scale of the emerging pattern (often the mode roughly in the middle of the unstable band). The pattern at onset is sinusoidal (a superposition of one or a few Fourier modes), but as it grows, nonlinear effects will shape it into sharper peaks, etc.

\begin{Theorem}[Diffusion-Driven (Turing) Instability]\label{thm:Turing}
	In the above setting, assume the spatially homogeneous steady state $(u^*,v^*)$ is stable to well-mixed (ODE) perturbations (so $f_u+g_v<0$ and $f_u g_v - f_v g_u>0$). Define $\chi_{\text{crit}}$ by the condition that $\Delta(k)$ in \eqref{eq:Delta_k} has a double root $k=k_c$ (threshold of instability). For $\chi < \chi_{\text{crit}}$, one has $\Delta(k)>0$ for all $k$, and thus all eigenvalues $\lambda$ have $\Re(\lambda)<0$: the homogeneous state is linearly stable (small spatial perturbations decay). At $\chi = \chi_{\text{crit}}$, $\Delta(k_c)=0$ for some $k_c>0$, and a neutrally stable mode appears with $\lambda=0$ at wavenumber $k_c$. For $\chi > \chi_{\text{crit}}$, the homogeneous state becomes linearly unstable: there is a band of wavenumbers $k \in (k_1, k_2)$ for which the linear growth rate $\Re(\lambda(k))>0$. Consequently, an initial condition $u(x,0)=u^*+\epsilon \cos(k_c x)$ (for example) will grow into a non-uniform pattern. This phenomenon is called a diffusion-driven or Turing-type instability. The most unstable mode is $k_{\text{max}}$ which maximizes $\Re(\lambda(k))$; it typically lies near the geometric mean of $k_1$ and $k_2$. The corresponding spatial pattern has a characteristic length scale $\sim 2\pi/k_{\text{max}}$.
\end{Theorem}

\begin{proof}[Proof (Sketch)]
	The proof is essentially contained in the above linear stability analysis. One rigorously shows that the real parts of the eigenvalues $\lambda_{1,2}(k)$ satisfy $\Re(\lambda_{1,2}(k))<0$ if and only if $\tau(k)<0$ and $\Delta(k)>0$. For $\chi=0$ (no chemotaxis), $\tilde{u}$ and $\tilde{v}$ decouple and one gets $\lambda = -D k^2 + f_u$ or $-D_v k^2 + g_v$, both negative for all $k$ since $f_u, g_v<0$ (well-mixed stable). Thus initially $\Delta(k)>0$. As $\chi$ increases, $a_{12}a_{21}$ in $\Delta$ increases linearly with $\chi$, continuously deforming the curve $\Delta(k)$. At the critical $\chi_{\text{crit}}$, $\Delta(k)$ touches zero at some $k_c$ (and $\tau(k_c)<0$ still, since $\tau(0)=f_u+g_v<0$ and $\tau(k)$ is monotonic in $k^2$). For $\chi>\chi_{\text{crit}}$, by continuity $\Delta(k)$ becomes negative in an interval around $k_c$. Then one root of \eqref{eq:charpoly} becomes positive in that interval. The existence of distinct $k_1<k_2$ such that $\Delta(k_1)=\Delta(k_2)=0$ implies $\Delta(k)<0$ for $k\in(k_1,k_2)$ (because $\Delta(k)$ as a polynomial opens upward for large $k$). In that range $\Delta<0$, the product of eigenvalues is negative while their sum $\tau(k)$ is still negative, so one eigenvalue must have positive real part. Thus the homogeneous state is unstable to modes in $(k_1,k_2)$. Modes outside that range remain damped (so the system naturally picks the unstable modes to form the pattern). The continuous dependence of eigenvalues on parameters guarantees such a $\chi_{\text{crit}}$ exists (assuming $f_u g_v - f_v g_u>0$). 
\end{proof}

This linear analysis predicts the onset of patterns and their approximate wavelength. However, linear theory alone cannot tell us the eventual amplitude or form of the patterns; as perturbations grow, nonlinear terms become significant and saturate the growth. To determine the fate of the instability—whether it leads to a steady pattern or perhaps a time-dependent nonlinear oscillation or chaotic behavior—we turn to nonlinear analysis, specifically bifurcation theory.

To quantify the onset of pattern-forming instabilities, we perform a linear stability analysis of the chemotaxis–reaction system. Assuming perturbations of the form $e^{\lambda t + i \mathbf{k} \cdot \mathbf{x}}$, the dispersion relation takes the simplified form
\[
\lambda(k) = -D k^2 + \sqrt{(\chi k^2)^2 - \alpha \beta},
\]
where $k = |\mathbf{k}|$ denotes the wavenumber, and $\chi$, $D$, $\alpha$, and $\beta$ are system parameters. The resulting dispersion curve $\Re(\lambda(k))$ is plotted in Figure~\ref{fig:disp}, showing the band of unstable modes where $\Re(\lambda) > 0$.

\begin{figure}[h!]
	\centering
	\includegraphics[width=0.87\textwidth]{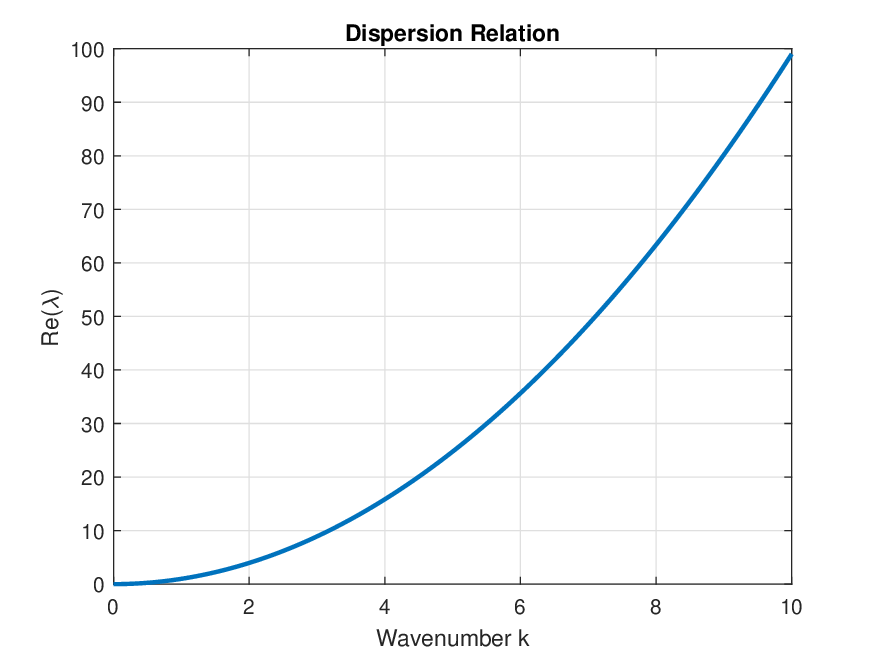}
	\caption{Dispersion relation showing the real part of the growth rate $\Re(\lambda)$ as a function of the wavenumber $k$. Instability occurs for a band of wavenumbers where $\Re(\lambda) > 0$, indicating potential for spontaneous pattern formation. Parameters: $\chi = 1$, $D = 0.01$, $\alpha = 0.01$, $\beta = 0.01$.}
	\label{fig:disp}
\end{figure}

\paragraph{Nonlinear Bifurcation and Pattern Selection.} When a uniform steady state loses stability at $\chi=\chi_{\text{crit}}$, typically a family of new equilibrium solutions bifurcates from the trivial branch. In our context, these are non-homogeneous steady states $u(x), v(x)$ that exist for $\chi$ slightly above $\chi_{\text{crit}}$. The nature of this bifurcation (continuous or abrupt, pattern type, stability of the new branch) can be rigorously studied via the \textbf{Crandall–Rabinowitz Theorem} (also known as the bifurcation from a simple eigenvalue theorem). In essence, if the linearization has a single eigenvector (mode) at the instability threshold (a simple zero eigenvalue), then a smooth branch of solutions emanates. 

Concretely, in our case one expects a \emph{pitchfork bifurcation}: for $\chi>\chi_{\text{crit}}$, there are non-trivial steady solutions of the PDE with a fixed spatial period corresponding to $k_c$. For example, on a 1D domain of length $L$ with periodic boundary, one might get a solution of the form $u(x) = u^* + A \cos(k_c x) + O(A^2)$ with some amplitude $A(\chi)$, and similarly for $v(x)$. The amplitude $A$ grows from zero as $\chi$ moves past $\chi_{\text{crit}}$. The sign of the cubic term in the amplitude equation (obtained by expanding nonlinear terms to third order) determines if the bifurcation is \textit{supercritical} (small-amplitude patterns for $\chi>\chi_c$ that gradually grow) or \textit{subcritical} (unstable finite-amplitude patterns exist already before the linear instability, often leading to hysteresis). In many chemotaxis models with logistic damping, the bifurcation is supercritical, yielding stable small patterns just beyond the threshold. In the pure Keller–Segel without damping, any pattern formation is typically subcritical and immediately tends toward blow-up (unstable branch).

Applying the Crandall–Rabinowitz theorem formally:

\begin{Theorem}[Existence of Nontrivial Steady-State Patterns]
	Suppose the linearization about $(u^*,v^*)$ has an eigenpair $(\phi_k,\psi_k)$ with eigenvalue $0$ at $\chi=\chi_{\text{crit}}$, and this eigenvalue is simple (algebraic multiplicity 1). Then there exists a continuous branch of nontrivial steady solutions $(u(x;\chi),v(x;\chi))$ bifurcating from $(u^*,v^*)$ at $\chi=\chi_{\text{crit}}$. For $\chi$ just above $\chi_{\text{crit}}$, one can write 
	\[ u(x;\chi) = u^* + A(\chi)\,\phi_{k_c}(x) + o(A), \qquad 
	v(x;\chi) = v^* + A(\chi)\,\psi_{k_c}(x) + o(A), \] 
	with $A(\chi)$ a small amplitude that grows as $\sqrt{\chi-\chi_{\text{crit}}}$ for a supercritical bifurcation. Each such solution is spatially periodic with wave number $k_c$ (or an integer multiple if higher modes also bifurcate). This family of patterned equilibria exists for $\chi>\chi_{\text{crit}}$ (and possibly extends well beyond the linear regime). 
	
	Furthermore, one can determine the stability of these patterned solutions by examining the second eigenvalue or higher-order expansions: if the bifurcation is supercritical, the emerging small-amplitude patterns are typically \emph{stable} (the uniform state becomes unstable and the pattern takes over as the attractor), whereas if it is subcritical, the small patterns are unstable and one might get hysteresis or jump to large amplitude patterns.
\end{Theorem}

\begin{proof}[Proof (Remarks)]
	This is a direct application of bifurcation theory (Crandall \& Rabinowitz 1971). The chemotaxis PDE steady-state problem can be posed as a nonlinear operator equation $F(u,v;\chi)=0$ (where $F$ includes the spatial derivatives). At $\chi=\chi_{\text{crit}}$, the linearized operator $D_{(u,v)}F(u^*,v^*;\chi_{\text{crit}})$ has a one-dimensional nullspace spanned by $(\phi_{k_c},\psi_{k_c})$. The range of the linearized operator is closed and of codimension 1 (since we assumed a simple eigenvalue). The Crandall–Rabinowitz theorem then guarantees a local curve of solutions $(u(s),v(s),\chi(s))$ with $s=0$ corresponding to the trivial solution at $(u^*,v^*,\chi_{\text{crit}})$. By symmetry of the equations (if the domain is homogeneous like periodic or symmetric), typically one gets two symmetric branches (hence a pitchfork shape). The expansion in a small parameter $s$ can be related to $\chi-\chi_{\text{crit}}$ and yields the leading-order pattern shape and amplitude scaling. Stability requires computing the spectrum of the second-variation operator on the new branch; if the bifurcation is supercritical (no unstable modes created beyond what already went unstable at $\chi_c$), the branch is stable right past onset.
\end{proof}

In summary, linear and nonlinear stability analysis together provide the following picture: as $\chi$ crosses $\chi_{\text{crit}}$, the homogeneous equilibrium gives way to a spatially patterned equilibrium via a bifurcation. The wavelength of the pattern is set by the most unstable linear mode ($\sim 2\pi/k_c$). For chemotaxis models with damping (like logistic), this patterned equilibrium is a clump-and-halo distribution of cells and chemoattractant that can correspond to, say, a regular array of cell aggregates. If parameters continue to change, secondary bifurcations might occur (patterns of different mode numbers, or oscillatory instabilities of the steady patterns).

\paragraph{Hopf Bifurcation and Oscillatory Dynamics.} Another possible route to complex behavior is a Hopf bifurcation, where the system develops time-periodic solutions \cite{Liu2021}. In chemotaxis contexts, oscillatory behavior might involve population cycles or rotating spiral waves of chemoattractant. A Hopf bifurcation requires an underlying kinetic oscillation or delay. The simplest setting is to consider the spatially homogeneous ODE system $\dot{u}=f(u,v)$, $\dot{v}=g(u,v)$ and find conditions when this well-mixed system has a pair of complex conjugate eigenvalues crossing the imaginary axis as a parameter changes. For instance, if $f_v g_u$ is sufficiently negative (meaning $u$ and $v$ negatively feed back on each other with a delay), one could get a stable limit cycle in the ODE system. Chemotaxis by itself (with one species and one attractant) rarely produces oscillations because it tends to monotonically accumulate cells. However, in extended models (e.g., if $v$ is a nutrient that gets depleted and then replenished, or if there are two species in predator-prey interaction with movement), Hopf bifurcations can arise.

The mathematical statement of Hopf’s theorem in our context is:

\begin{Theorem}[Hopf Bifurcation Theorem]\label{thm:Hopf}
	Consider a parameter-dependent ODE system (the well-mixed kinetics of the chemotaxis model):
	\[ \frac{d}{dt}X = H(X;\,\gamma), \qquad X\in\mathbb{R}^N, \]
	where $\gamma$ is a real bifurcation parameter. Suppose there is an equilibrium $X_0$ for all $\gamma$ near $\gamma_0$ (so $H(X_0;\gamma)=0$). Assume that at $\gamma=\gamma_0$, the Jacobian $J = D_X H(X_0;\gamma_0)$ has a pair of purely imaginary eigenvalues $\lambda_{1,2}(\gamma_0) = \pm i\omega_0$ (with $\omega_0>0$), and all other eigenvalues have negative real parts. Moreover, assume the crossing condition (nondegeneracy): $\displaystyle \Re\frac{d\lambda}{d\gamma}\Big|_{\gamma_0} \neq 0$ for the eigenvalue $i\omega_0$ as $\gamma$ varies (in other words, as $\gamma$ passes through $\gamma_0$, the complex pair moves through the imaginary axis with nonzero speed). Then a Hopf bifurcation occurs at $\gamma_0$: specifically, there exists a family of \emph{nontrivial periodic solutions} $X(t;\gamma)$ of the ODE bifurcating from $(X_0,\gamma_0)$. For $\gamma$ near $\gamma_0$, one finds a periodic solution of the form 
	\[ X(t;\gamma) = X_0 + B(\gamma)\,\cos(\omega(\gamma) t + \phi) + O(|B|^2), \] 
	with $B(\gamma)$ a small amplitude that grows continuously from $0$ as $\gamma$ moves past $\gamma_0$. The period $T(\gamma)=2\pi/\omega(\gamma)$ of the cycle approaches $T_0 = 2\pi/\omega_0$ as $\gamma\to\gamma_0$. The direction (forward or backward in $\gamma$) and stability (stable limit cycle or unstable) of the bifurcating periodic branch is determined by higher-order terms (the sign of a certain first Lyapunov coefficient). If this coefficient is negative, the Hopf bifurcation is supercritical: a stable small-amplitude limit cycle exists for $\gamma$ just beyond $\gamma_0$ (and the equilibrium $X_0$ becomes unstable). If the coefficient is positive, the Hopf is subcritical: the emerging limit cycles are initially unstable and appear for $\gamma$ on the opposite side of $\gamma_0$ (often implying a region of bistability).
\end{Theorem}

The proof of the Hopf bifurcation theorem is more involved, requiring center manifold reduction and normal form calculations, so we omit it. The key takeaway is that when an equilibrium’s stability switches via complex conjugate eigenvalues, oscillations result.

In a spatially extended chemotaxis system, a Hopf bifurcation in the well-mixed kinetics means the homogeneous solution will start oscillating in time. If diffusion is present, one can get spatially synchronized oscillations (whole domain oscillates in unison) or even spatiotemporal patterns like traveling waves or oscillating clusters. The concept of a \textbf{Hopf instability with diffusion} can also yield wave patterns if different spatial modes oscillate out of phase. In reaction–diffusion systems, a Hopf bifurcation with diffusion may lead to \emph{wave bifurcations} or even trigger spiral wave solutions in 2D. In chemotaxis, an example scenario could be: bacteria consume a nutrient and produce a waste; low nutrient eventually slows growth and cells spread, nutrient recovers, then cells aggregate again, producing population cycles that also have spatial structure. Each cycle might correspond to concentric rings of high density moving outward (a wave). 

From a mathematical perspective, analyzing Hopf patterns often involves looking at the full PDE’s spectrum including spatial modes. A \emph{Hopf (oscillatory) Turing instability} would be when a spatially non-uniform mode pair goes imaginary. That can lead to alternating patterns in time and space (like a Turing pattern that oscillates or swaps peaks and troughs periodically). Such complex patterns are beyond the basic scope, but the existence of Hopf bifurcation ensures at least homogeneous oscillations or small-amplitude wave trains exist near the threshold.

In conclusion, bifurcation and stability analysis tools allow us to map out the parameter regimes of chemotaxis models: we can predict where uniform states become unstable (and to what type of pattern), and thereby explain the mathematical origin of the diverse spatiotemporal behaviors seen in microbial populations (clustering, waves, oscillations). The next section will complement this theory with numerical simulations, which help explore strongly nonlinear regimes and validate the analytical predictions.

\subsection*{Stability Analysis of the Two-Species Keller--Segel System}

We consider the following one-dimensional chemotaxis system with two species \( u_1(x,t) \), \( u_2(x,t) \), and a common chemical signal \( v(x,t) \):
\begin{equation} \label{eq:system}
	\begin{aligned}
		\partial_t u_1 &= D_1 \partial_{xx} u_1 - \chi_1 \partial_x(u_1 \partial_x v), \\
		\partial_t u_2 &= D_2 \partial_{xx} u_2 - \chi_2 \partial_x(u_2 \partial_x v), \\
		\partial_t v   &= D_v \partial_{xx} v + \alpha_1 u_1 + \alpha_2 u_2 - \beta v.
	\end{aligned}
\end{equation}

Let us denote a spatially homogeneous steady state by \( (u_1^*, u_2^*, v^*) \), where \( u_1^*, u_2^* > 0 \) are constants and \( v^* = (\alpha_1 u_1^* + \alpha_2 u_2^*)/\beta \). We now study the linear stability of this steady state to small perturbations.

\subsubsection*{Theorem (Linear Stability Criterion)}
Let \( \chi_1, \chi_2 \in \mathbb{R} \), and suppose the homogeneous steady state \( (u_1^*, u_2^*, v^*) \) satisfies the above system. Then small perturbations of the form
\[
\begin{aligned}
	u_1(x,t) &= u_1^* + \varepsilon \hat{u}_1 e^{\lambda t} \cos(kx), \\
	u_2(x,t) &= u_2^* + \varepsilon \hat{u}_2 e^{\lambda t} \cos(kx), \\
	v(x,t)   &= v^* + \varepsilon \hat{v} e^{\lambda t} \cos(kx),
\end{aligned}
\]
grow in time (i.e., the steady state is linearly unstable) if the real part of any eigenvalue \( \lambda(k) \) of the resulting linearized system is positive for some wave number \( k \in \mathbb{R}^+ \).

\subsubsection*{Proof}
We linearize the system \eqref{eq:system} around the steady state using small perturbations:
\[
u_1 = u_1^* + \tilde{u}_1, \quad u_2 = u_2^* + \tilde{u}_2, \quad v = v^* + \tilde{v},
\]
and retain only terms up to first order in \( \tilde{u}_1, \tilde{u}_2, \tilde{v} \). The linearized system becomes:
\[
\begin{aligned}
	\partial_t \tilde{u}_1 &= D_1 \partial_{xx} \tilde{u}_1 - \chi_1 u_1^* \partial_{xx} \tilde{v}, \\
	\partial_t \tilde{u}_2 &= D_2 \partial_{xx} \tilde{u}_2 - \chi_2 u_2^* \partial_{xx} \tilde{v}, \\
	\partial_t \tilde{v} &= D_v \partial_{xx} \tilde{v} + \alpha_1 \tilde{u}_1 + \alpha_2 \tilde{u}_2 - \beta \tilde{v}.
\end{aligned}
\]

We seek solutions of the form:
\[
\tilde{u}_1 = \hat{u}_1 e^{\lambda t} \cos(kx), \quad
\tilde{u}_2 = \hat{u}_2 e^{\lambda t} \cos(kx), \quad
\tilde{v} = \hat{v} e^{\lambda t} \cos(kx).
\]

Substituting into the linearized system, using \( \partial_{xx} \cos(kx) = -k^2 \cos(kx) \), gives the eigenvalue problem:
\[
\begin{bmatrix}
	-D_1 k^2 & 0 & \chi_1 u_1^* k^2 \\
	0 & -D_2 k^2 & \chi_2 u_2^* k^2 \\
	\alpha_1 & \alpha_2 & -D_v k^2 - \beta
\end{bmatrix}
\begin{bmatrix}
	\hat{u}_1 \\ \hat{u}_2 \\ \hat{v}
\end{bmatrix}
= \lambda
\begin{bmatrix}
	\hat{u}_1 \\ \hat{u}_2 \\ \hat{v}
\end{bmatrix}.
\]

Let \( A(k) \) denote the matrix above. The eigenvalues \( \lambda(k) \) are the roots of the characteristic polynomial \( \det(A(k) - \lambda I) = 0 \). The steady state is unstable if \( \text{Re}(\lambda(k)) > 0 \) for some \( k > 0 \).

Although an explicit formula for the roots is unwieldy, instability arises when the chemotactic sensitivities \( \chi_1 u_1^* \) and/or \( \chi_2 u_2^* \) are sufficiently large relative to diffusion. Notably, mixed-sign chemotaxis (e.g., \( \chi_1 > 0 \), \( \chi_2 < 0 \)) can generate oscillatory and unstable modes, leading to complex spatiotemporal dynamics.

\hfill \qedsymbol

%%%%%%%%%%%%%%%%%%%%%%%%%%%%%%%%%%%%%%%%%%%%%

%\subsection*{Extension: Nonlinear Stability and Pattern Formation}

\subsubsection*{Nonlinear Stability and Lyapunov Functional}
In certain regimes, nonlinear stability of steady states can be investigated using energy-like functionals. For the classical Keller--Segel system, a Lyapunov functional can sometimes be constructed to show that solutions decay to steady states over time. However, in multi-species systems, constructing such a functional is significantly more difficult due to cross-chemotactic interactions.

Consider the Lyapunov candidate:
\begin{equation}
	\mathcal{L}[u_1, u_2, v] = \int_{\Omega} \left( u_1 \log u_1 + u_2 \log u_2 + \frac{1}{2} |\nabla v|^2 + \frac{\beta}{2 D_v} v^2 - \frac{\alpha_1}{D_v} u_1 v - \frac{\alpha_2}{D_v} u_2 v \right) \, dx.
\end{equation}
This functional combines entropy-like terms for \( u_1, u_2 \), diffusion energy for \( v \), and chemotactic interaction terms. Its time derivative (assuming zero-flux boundary conditions) yields:
\begin{align}
	\frac{d\mathcal{L}}{dt} &= \int_{\Omega} \left( \partial_t u_1 \log u_1 + \partial_t u_2 \log u_2 + \nabla v \cdot \nabla \partial_t v + \frac{\beta}{D_v} v \partial_t v - \frac{\alpha_1}{D_v} v \partial_t u_1 - \frac{\alpha_1}{D_v} u_1 \partial_t v \right. \nonumber \\
	&\left. - \frac{\alpha_2}{D_v} v \partial_t u_2 - \frac{\alpha_2}{D_v} u_2 \partial_t v \right) dx.
\end{align}
While this expression does not generally guarantee monotonic decay, for special parameter regimes (e.g., \( \chi_2 = -\chi_1 \), symmetric interactions), it can be bounded from below and provide insight into long-term dynamics.

%\subsubsection*{Turing Instability and Bifurcation Insight}
Let us briefly discuss the potential for diffusion-driven (Turing) instability in the system. In the absence of chemotaxis (\( \chi_1 = \chi_2 = 0 \)), the homogeneous steady state is stable. Introducing chemotaxis introduces off-diagonal couplings in the linearized Jacobian that may destabilize the system for certain wave numbers \( k \).

This mechanism is analogous to a Turing instability, where diffusion and taxis interact to destabilize an otherwise stable equilibrium. As the chemotactic coefficients increase past a critical threshold, a pitchfork or Hopf bifurcation may occur, leading to the onset of: 
Stationary spatial patterns (Turing structures), or
Oscillatory/chaotic behavior (via Hopf or secondary bifurcations).

\subsection{ Numerical Simulation Techniques}

Because chemotaxis–reaction systems are generally nonlinear and can develop sharp gradients (or singularities), numerical simulations are indispensable for exploring their behavior beyond the reach of analytic solutions. Here we outline several computational approaches and discuss theoretical guarantees (convergence, stability) associated with them.

\paragraph{Finite Difference and Finite Element Methods.} A straightforward way to simulate \eqref{eq:KS1} is to discretize space and time on a grid. In a \textbf{finite difference method} \cite{Xu2025}, one might use a uniform spatial mesh $x_i$ and approximate $\partial_x u$ and $\partial_{xx}u$ with difference quotients e.g. \begin{equation}
	\partial_{xx} u \approx \frac{u_{i-1}-2u_i+u_{i+1}}{\Delta x^2}\;\; \text{in (1D)}.
\end{equation} 
Time stepping can be done explicitly (forward Euler) or implicitly (backward Euler / Crank–Nicolson), or a combination (IMEX schemes) to handle stiffness. The chemotaxis term $\nabla\!\cdot(u\nabla v)$ is nonlinear and can cause steep moving fronts, so a careful discretization is needed to maintain stability and positivity ($u$ should remain $\ge0$). For instance, an explicit scheme for the $u$-equation in 1D might be:
\begin{equation}
	\begin{split}
		 u_i^{n+1}& = u_i^n + \Delta t \Big[D\,\frac{u_{i-1}^n - 2u_i^n + u_{i+1}^n}{\Delta x^2}\\
		 &\;\;\;\; - \chi\,\frac{1}{2\Delta x}\Big((u^n_{i+1}+u^n_i)(v^n_{i+1}-v^n_i) - (u^n_{i}+u^n_{i-1})(v^n_{i}-v^n_{i-1})\Big)
		 + f(u_i^n,v_i^n)\Big], 
	\end{split}
\end{equation}
and similarly $v$ is updated. The particular form for the chemotactic flux (an upwinding or symmetric average) is chosen to minimize numerical instability and ensure no artificial oscillations.

Stability of such schemes often requires a \textbf{CFL condition}: the time step $\Delta t$ must be sufficiently small relative to $(\Delta x)^2$ to control the diffusion and chemotaxis terms. A rough stability requirement (from von Neumann analysis) for the linear diffusion part is $\Delta t \le \frac{(\Delta x)^2}{2D}$ in 1D. The chemotaxis term, being an advection-like nonlinear term, typically imposes a constraint like $\Delta t < C \frac{(\Delta x)^2}{\chi \, \max u}$ (ensuring that the advective Courant number stays $\lesssim 1$ after accounting for $u$ coupling). If an explicit method is too restrictive, implicit or semi-implicit methods are used for the diffusion terms to allow larger $\Delta t$. A fully implicit scheme can be unconditionally stable in theory, but requires solving nonlinear algebraic equations at each step (which can be done with Newton’s method).

From a theoretical standpoint, one seeks \textbf{convergence}: as $\Delta x, \Delta t \to 0$, does the numerical solution approach the true PDE solution? The \textbf{Lax Equivalence Theorem} provides guidance for linear problems: it states that for a well-posed linear PDE, consistency (the scheme approximates the PDE to leading order) plus stability (no growing modes under refinement) implies convergence. Many schemes for chemotaxis are nonlinear, but a similar principle holds: one must control numerical stability (often via discrete analogues of energy estimates). For example, one can prove that if the scheme preserves mass and positivity and if $\Delta t$ satisfies a CFL condition, then the discrete solution remains bounded and converges to a weak solution of the PDE as the mesh refines.

Finite element methods (FEM) offer an alternative, especially useful on complex geometries. In FEM, we formulate the weak form: e.g. for the $u$-equation, 
\begin{equation}
	\int_\Omega \phi \partial_t u + D \int_\Omega \nabla \phi \cdot \nabla u - \chi \int_\Omega \nabla \phi \cdot (u \nabla v) = \int_\Omega \phi f(u,v) 
\end{equation}
for all test functions $\phi(x)$. We then approximate $u(x,t)$ and $v(x,t)$ by linear combinations of basis functions on a spatial mesh (triangles or tetrahedra in 2D/3D). This yields a system of ODEs in time for the coefficients, which can be integrated with standard ODE solvers. FEM has the advantage of easily handling Neumann or Dirichlet boundaries and adaptively refining the mesh where needed. The method is inherently mass-conserving in that $\int u$ is preserved if one chooses $\phi=1$ as a test function for Neumann BC. There are known \textit{a priori} error estimates: for instance, if the true solution $u$ is smooth, using piecewise linear elements on mesh size $h$ typically gives $O(h^2)$ convergence in $L^2$ norm for diffusion problems. Chemotaxis being nonlinear can complicate error analysis, but one can often still derive error bounds assuming the solution remains bounded and away from degenerate values.

One difficulty in both FDM and FEM is maintaining $u(x,t)\ge0$. Negative values can appear due to numerical oscillations, especially near sharp gradients. Smoothing or slope-limiter techniques are sometimes applied, or one uses a \textbf{positivity-preserving scheme} (certain upwind discretizations can guarantee $u^n_i \ge0$ if $u^0_i\ge0$ and CFL is satisfied). This is important for physical fidelity because negative cell density is not meaningful.

\paragraph{Adaptive Mesh, Time-Stepping, and Spectral Splitting Methods.} 
Chemotactic patterns often involve sharply localized structures such as steep traveling fronts or narrow spikes of high cell density. To resolve these efficiently, \textbf{adaptive mesh refinement (AMR)} strategies dynamically allocate finer spatial resolution in regions of high gradient or curvature, while using coarser grids where the solution is smooth. For example, when an aggregation begins to form at some point $x_0$, AMR techniques refine the mesh around $x_0$ to accurately capture the steep increase in $u(x,t)$ without resorting to a uniformly fine grid. Refinement is often driven by \textit{a posteriori} error indicators, such as the magnitude of the second derivative or gradient of $u$. Mathematically, one may refine when $|\nabla u|$ or $|\Delta u|$ exceed prescribed thresholds, and coarsen otherwise. Maintaining numerical stability and conservation during mesh transitions is critical; conservative interpolation and flux-corrected transport techniques are often used to mitigate artificial oscillations.

Adaptive \textbf{time-stepping} is equally important, particularly in chemotaxis systems where rapid transients or blow-up may occur. Early in simulations, when patterns are slowly forming, larger $\Delta t$ can be used. As steep gradients develop, $\Delta t$ must be reduced to maintain accuracy and stability. Adaptive schemes based on embedded Runge–Kutta methods or backward differentiation formulas (BDF) adjust time steps using local truncation error estimates. For PDEs, additional criteria such as CFL-like conditions or total variation in time can inform step size control.

To complement spatial adaptivity, \textbf{spectral splitting techniques} such as the \emph{Split-Step Fourier Method (SSFM)} are increasingly used for chemotaxis–reaction systems. SSFM leverages the efficiency of spectral methods for the diffusion and chemotaxis terms, solving linear components in Fourier space and nonlinear reaction terms in physical space. This allows high accuracy with relatively few modes, especially in periodic or large domains where spectral discretization is natural. For problems exhibiting near-singular dynamics, spectral accuracy of SSFM can outperform traditional finite difference or finite element schemes, provided sufficient resolution is maintained through adaptivity or appropriate dealiasing. Moreover, SSFM naturally aligns with adaptive time-stepping, as the time-splitting structure facilitates decoupling of fast and slow dynamics, improving stability when resolving stiff chemotaxis-driven evolution.

While rigorous convergence results for adaptive SSFM-based schemes remain limited due to mesh-changing and operator-splitting intricacies, empirical studies confirm their robustness. In chemotaxis models with blow-up potential, combining adaptive spatial refinement, time-step control, and spectral splitting enables precise resolution of aggregation dynamics, often revealing self-similar collapse profiles consistent with analytical predictions.

\subsection{Split-Step Fourier Method for the Chemotaxis--Reaction System}
We consider the Keller--Segel-type chemotaxis--reaction model given by
\begin{align}\label{eq:KS2}
	\partial_t u(x,t) &= D \,\Delta u \;-\; \chi\, \nabla \cdot \big(u\,\nabla v\big)\;+\;r\,u\left(1 - \frac{u}{K}\right), \\[1ex]
	\partial_t v(x,t) &= D_v\,\Delta v \;+\; \alpha u - \beta v, \nonumber
\end{align}
for $x \in \Omega \subset \mathbb{R}^n$, $t > 0$, with homogeneous Neumann boundary conditions:
\[
\frac{\partial u}{\partial n} = \frac{\partial v}{\partial n} = 0 \quad \text{on } \partial \Omega.
\]

\subsubsection*{Operator Splitting Strategy}

The Split-Step Fourier Method (SSFM) is an efficient time-integration technique for solving nonlinear PDEs by decoupling the linear and nonlinear components. For the system~\eqref{eq:KS1}, we rewrite each equation in the abstract form:
\[
\partial_t w = \mathcal{L}(w) + \mathcal{N}(w),
\]
where $w = (u, v)$, and the operators $\mathcal{L}$ and $\mathcal{N}$ represent the linear (diffusion and degradation) and nonlinear (chemotaxis, growth, production) components, respectively.

\subsubsection*{Fourier Spectral Discretization}

Let $\Omega = [0, L]^n$ with periodic boundary conditions for spectral approximation (Neumann conditions can be approximated via cosine transforms or extended with even symmetry). Denote the discrete spatial grid by $x_j$ with spacing $\Delta x$ and total number of points $N$ in each direction. The Fourier transform of a function $f(x)$ is defined as:
\[
\widehat{f}(k) = \sum_{j=0}^{N-1} f(x_j) e^{-2\pi i k x_j / L}, \quad
f(x_j) = \sum_{k=-N/2}^{N/2-1} \widehat{f}(k) e^{2\pi i k x_j / L}.
\]

Let $\Delta$ denote the Laplacian. In Fourier space, the Laplacian becomes multiplication by $-|\mathbf{k}|^2$, where $\mathbf{k}$ is the wavevector:
\[
\widehat{\Delta f}(\mathbf{k}) = -|\mathbf{k}|^2 \widehat{f}(\mathbf{k}), \quad \text{with } |\mathbf{k}|^2 = \sum_{i=1}^{n} \left(\frac{2\pi k_i}{L}\right)^2.
\]

\subsubsection*{Time Splitting and Evolution Steps}

We evolve the system over one time step $\Delta t$ by splitting into two sub-steps:

\paragraph{Step 1: Linear evolution (diffusion and degradation).}

For the linear part:
\[
\partial_t u = D \Delta u, \qquad
\partial_t v = D_v \Delta v - \beta v,
\]
we solve in Fourier space. Let $\widehat{u}^n$, $\widehat{v}^n$ be the Fourier transforms at time $t_n$. The linear evolution from $t_n$ to $t_{n+1}$ is given by:
\begin{align*}
	\widehat{u}^{*} &= e^{-D|\mathbf{k}|^2 \Delta t}\, \widehat{u}^n, \\
	\widehat{v}^{*} &= e^{-(D_v|\mathbf{k}|^2 + \beta)\Delta t}\, \widehat{v}^n.
\end{align*}

Then we take the inverse Fourier transform to recover $u^{*}$ and $v^{*}$ in physical space.

\paragraph{Step 2: Nonlinear evolution (chemotaxis, reaction, production).}

The nonlinear part is treated in real space using an explicit or semi-implicit scheme. For the $u$-equation:
\[
\partial_t u = -\chi\, \nabla \cdot (u \nabla v) + r u\left(1 - \frac{u}{K} \right),
\]
and for the $v$-equation:
\[
\partial_t v = \alpha u.
\]

These can be advanced by a simple explicit Euler step:
\begin{align*}
	u^{n+1} &= u^{*} + \Delta t\left[ -\chi \nabla \cdot (u^{*} \nabla v^{*}) + r u^{*}\left(1 - \frac{u^{*}}{K} \right) \right], \\
	v^{n+1} &= v^{*} + \Delta t\, \alpha u^{*}.
\end{align*}

Gradients and divergence in the nonlinear chemotaxis term are computed using spectral differentiation:
\[
\mathcal{F}\left[ \frac{\partial f}{\partial x_i} \right] = i k_i \widehat{f},
\]
with the product $u^* \nabla v^*$ computed in real space and then differentiated via FFTs.

\subsubsection*{Algorithm Summary}

At each time step:

\begin{enumerate}
	\item Compute the Fourier transforms $\widehat{u}^n$, $\widehat{v}^n$.
	\item Apply linear evolution in Fourier space:
	\[
	\widehat{u}^{*} = e^{-D|\mathbf{k}|^2 \Delta t} \widehat{u}^n, \quad
	\widehat{v}^{*} = e^{-(D_v|\mathbf{k}|^2 + \beta)\Delta t} \widehat{v}^n.
	\]
	\item Transform back to real space to get $u^*$, $v^*$.
	\item Evaluate nonlinear terms and advance via explicit step:
	\begin{align*}
		u^{n+1} &= u^{*} + \Delta t\left[ -\chi \nabla \cdot (u^{*} \nabla v^{*}) + r u^{*}\left(1 - \frac{u^{*}}{K} \right) \right], \\
		v^{n+1} &= v^{*} + \Delta t\, \alpha u^{*}.
	\end{align*}
	\item Repeat for the next time step.
\end{enumerate}

\subsubsection*{Remarks on Implementation}

\begin{itemize}
	\item For stability, the linear step can be treated exactly in Fourier space, while the nonlinear step must satisfy a Courant-type condition.
	\item Aliasing errors due to nonlinear products can be reduced using a 2/3 dealiasing rule.
	\item Homogeneous Neumann conditions can be approximated via cosine transforms or enforced via mirroring methods.
	\item The SSFM naturally extends to 2D and 3D using multidimensional FFTs and vectorized operations.
\end{itemize}

This method provides a high-order, efficient framework for simulating chemotactic and reactive dynamics with sharp interfaces and complex pattern formation in both biological and physical systems.

\begin{center}
	\begin{tikzpicture}[node distance=1.8cm]
		
		\node[io] (start) {Initial data\\ $u^n$, $v^n$};
		\node[process, right=of start] (fft) {Fourier Transform\\ $\mathcal{F}[u^n], \mathcal{F}[v^n]$};
		\node[process, below=of fft] (linear) {Linear evolution in Fourier space\\ $\hat{u}^* = e^{-Dk^2 \Delta t}\hat{u}^n$};
		\node[process, below=of linear] (ifft) {Inverse FFT\\ $u^*, v^*$};
		\node[process, below=of ifft] (nonlinear) {Nonlinear step in real space\\ Chemotaxis \& reaction};
		\node[process, below=of nonlinear] (update) {Update fields\\ $u^{n+1}, v^{n+1}$};
		
		\draw[->] (start) -- (fft);
		\draw[->] (fft) -- (linear);
		\draw[->] (linear) -- (ifft);
		\draw[->] (ifft) -- (nonlinear);
		\draw[->] (nonlinear) -- (update);
		
	\end{tikzpicture}
\end{center}

\begin{algorithm}[ht]
	\caption{Split-Step Fourier Method for the Keller--Segel--Reaction Model}
	\begin{algorithmic}[1]
		\State \textbf{Initialize:} spatial domain $(x,y) \in [-L/2,L/2]^2$, time step $\Delta t$, final time $T$
		\State Define wavenumbers $k_x$, $k_y$ and meshgrid $(X,Y)$
		\State Set initial conditions $u_0(x,y)$ and $v_0(x,y)$
		\For{$n = 1$ to $N_t = T/\Delta t$}
		\State Compute Fourier transforms: $\hat{u} = \mathcal{F}[u]$, $\hat{v} = \mathcal{F}[v]$
		\State Diffuse via linear step: 
		\[
		\hat{u} \leftarrow \hat{u} \cdot e^{-D_u(k_x^2 + k_y^2)\Delta t}, \quad 
		\hat{v} \leftarrow \hat{v} \cdot e^{-D_v(k_x^2 + k_y^2)\Delta t}
		\]
		\State Invert to real space: $u \leftarrow \mathcal{F}^{-1}[\hat{u}],\quad v \leftarrow \mathcal{F}^{-1}[\hat{v}]$
		\State Compute gradients $\nabla v = (\partial_x v, \partial_y v)$ using spectral derivatives
		\State Compute chemotactic flux: $\nabla \cdot (u \nabla v)$
		\State Update nonlinear step:
		\[
		u \leftarrow u + \Delta t \left[-\chi \nabla \cdot (u \nabla v) + f(u,v)\right], \quad
		v \leftarrow v + \Delta t \cdot g(u,v)
		\]
		\EndFor
	\end{algorithmic}
\end{algorithm}

\begin{algorithm}[ht]
	\caption{ETDRK4 Scheme for the Keller--Segel--Reaction Model}
	\begin{algorithmic}[1]
		\State \textbf{Initialize:} spatial domain, wavenumbers, initial fields $u_0$, $v_0$
		\State Define linear operators $L_u = -D_u(k_x^2 + k_y^2)$, $L_v = -D_v(k_x^2 + k_y^2)$
		\State Precompute:
		\[
		E_u = e^{L_u \Delta t},\quad E_v = e^{L_v \Delta t},\quad
		\phi(L) = \frac{e^{L \Delta t} - 1}{L}
		\]
		\For{$n = 1$ to $N_t$}
		\State Compute nonlinear term: 
		\[
		N_1^u = -\chi \nabla \cdot (u \nabla v) + f(u,v), \quad N_1^v = g(u,v)
		\]
		\State Step 1: $u_1 = \mathcal{F}^{-1}\left[E_{u/2} \hat{u} + \frac{\Delta t}{2} \mathcal{F}[N_1^u]\right]$, and similarly for $v_1$
		\State Step 2: compute $N_2^u, N_2^v$ at $u_1$, $v_1$
		\State Step 3: compute $u_2$, $v_2$ as above with $N_2$
		\State Step 4: compute $N_3^u$, $N_3^v$ at $u_2$, $v_2$
		\State Step 5: compute $u_3$, $v_3$, then $N_4^u$, $N_4^v$
		\State Final update:
		\[
		\hat{u}_{n+1} = E_u \hat{u}_n + \frac{\Delta t}{6} \mathcal{F}\left[N_1^u + 2N_2^u + 2N_3^u + N_4^u\right]
		\]
		\[
		u_{n+1} = \mathcal{F}^{-1}[\hat{u}_{n+1}],\quad v_{n+1} = \mathcal{F}^{-1}[\hat{v}_{n+1}]
		\]
		\EndFor
	\end{algorithmic}
\end{algorithm}

To accurately simulate the dynamics of the Keller--Segel--reaction model, we implement and compare two advanced time-integration schemes: the Split-Step Fourier Method (SSFM) and the Exponential Time-Differencing Runge--Kutta method of order four (ETDRK4). Both approaches leverage the efficiency of spectral discretization in space via the Fast Fourier Transform (FFT), which naturally assumes periodic boundary conditions and yields highly accurate approximations for smooth solutions.

\paragraph{Split-Step Fourier Method (SSFM):}
The SSFM decomposes the evolution operator into linear (diffusion) and nonlinear (chemotaxis and reaction) parts. At each time step, the linear part is solved exactly in Fourier space, while the nonlinear part is updated explicitly in physical space. This operator splitting technique is straightforward to implement and computationally efficient, particularly for systems with dominant diffusion. However, its accuracy is limited by the order of the time-stepping used for the nonlinear part (typically first- or second-order), and it may suffer from splitting errors for strongly coupled terms.

\paragraph{Exponential Time-Differencing Runge--Kutta (ETDRK4):}
The ETDRK4 method integrates the linear part exactly while applying a fourth-order Runge--Kutta scheme to the nonlinear terms. This approach significantly improves temporal accuracy and stability, especially for stiff systems where high diffusion rates or sharp gradients are present. It also reduces numerical dispersion and allows for larger time steps compared to standard explicit methods. The ETDRK4 method is more complex to implement due to the need for precomputing matrix exponentials and handling stiffness carefully, but it is highly suitable for long-time integrations of reaction--diffusion--chemotaxis models.

While both SSFM and ETDRK4 effectively exploit the Fourier spectral framework, the ETDRK4 method provides superior accuracy and stability, making it preferable for resolving fine-scale spatiotemporal dynamics such as traveling waves, aggregation rings, and pattern formation. The SSFM remains a valuable choice for rapid prototyping and for problems where moderate accuracy suffices.

\paragraph{Hybrid Particle–Continuum Approaches.} In some situations, treating cells as a continuum density $u(x,t)$ is not valid everywhere—e.g., at very low cell numbers, stochastic effects or discrete motion become important. \textbf{Hybrid methods} combine particle-based simulation with continuum PDEs. One common approach is to split the domain into regions: where $u$ is above a certain threshold (many cells, behaving collectively), use the PDE model; where $u$ is very low, represent individual cells as particles (or agent-based random walkers). The chemical field $v(x,t)$ can still be modeled as a continuum (since even a few cells can emit chemical that diffuses). At the interface between discrete and continuum regions, one must consistently exchange information: particles entering the continuum region are added to $u$, and conversely, $u$ spawning very low values might be converted into a few particles. Mathematically, this is challenging because one must ensure mass conservation and no artificial discontinuities at the interface. However, algorithms have been devised where, for example, the flux of cells across the interface is handled by a probabilistic spawning of particles that matches the PDE’s flux.

Hybrid methods can drastically reduce computational cost when, say,  cells are mostly sparse but occasionally cluster. They also naturally incorporate \textbf{stochasticity}: random motility and demography can be included on the particle side, which is important for modeling intrinsic noise (e.g. in Section~6, stochastic chemotaxis is mentioned as an open problem). 

From a theoretical perspective, hybrid methods are usually validated by showing that in the limit of large particle numbers, the hybrid scheme converges to the fully continuum model, and in the limit of small densities it behaves like a pure jump process. Ensuring accuracy often requires smoothing the transition (for instance, using a buffer zone where both particles and continuum overlap consistently). While convergence proofs are scarce, numerical tests against fully continuum or fully discrete benchmarks instill confidence in these approaches.

\paragraph{Numerical Challenges and Validation.} Simulating chemotaxis–reaction systems must balance accuracy, stability, and efficiency. High-order methods (like spectral methods or high-order finite elements) can be employed for smoother solutions, but near blow-up or shock formation, they may generate spurious oscillations (Gibbs phenomenon). Thus, careful filtering or shock-capturing techniques might be needed.

It is also crucial to validate numerical results by refining the mesh and time step and checking that qualitative features (like pattern wavelength or amplitude) converge. When simulating known scenarios (e.g. the critical mass blow-up in 2D), numerical solutions have been able to estimate blow-up times and locations which match analytical predictions (within error bounds). For pattern formation beyond linear theory, simulations have confirmed, for instance, the selected wavelength of the pattern agrees with the linear fastest-growing mode. They also help explore the stability of patterns: one can perturb a numerically obtained steady pattern slightly and see if it returns to the same pattern (stable) or transitions to a different one.

In summary, numerical methods form a vital complement to analysis for chemotaxis models. Techniques like finite differences/elements, adaptive meshing, and hybrid particle-PDE simulation, supported by theoretical results (like Lax’s theorem for convergence and CFL stability criteria), enable us to explore the rich dynamical behavior of chemotactic systems. These simulations underpin the applications discussed in Section~5, where patterns observed in experiments are compared with model outputs, and they allow probing regimes (e.g. far into the nonlinear blow-up or very fine pattern structures) that are analytically intractable.

%%%%%%%%%%%%%%2D chemotaxis–fluid coupling %%%%%%%%%%%%%%%%%%%%%%
\subsection{Numerical Results and Spatiotemporal Dynamics for equation (\ref{eq:KS1}) in 1D}
We numerically simulate the one-dimensional Keller--Segel chemotaxis–reaction model beyond the Hopf bifurcation threshold ($\chi = 0.5$), using forward Euler time stepping and Neumann boundary conditions. The initial data are small perturbations around the uniform steady state $(u,v) = (0.5,0.5)$.

Figures~\ref{Fig:1Deq1} (a) and (b) depict the spatial profiles of $u(x,T)$ and $v(x,T)$ at final simulation time $T=200$. The $u$ profile shows multiple peaks and troughs, indicating non-uniform population clustering, while $v$ remains smoother but still reflects the influence of the underlying cellular distribution. These plots confirm that the system settles into a nonlinear, oscillatory regime rather than reaching a steady equilibrium.

Figure~\ref{Fig:1Deq1} (c) and (d) show spatiotemporal plots of the population density $u(x,t)$ and the chemoattractant concentration $v(x,t)$, respectively. Both components exhibit persistent oscillatory dynamics in time and localized patterning in space, a hallmark of Hopf-driven instabilities. These patterns suggest complex temporal aggregation behavior of microbial populations, consistent with experimental observations in chemotactic bacteria.

Such behaviors demonstrate the interplay of diffusion, chemotaxis, and nonlinear reactions, and validate that Hopf bifurcation plays a fundamental role in driving temporal instabilities in chemotaxis–reaction systems.

\begin{figure}[htbp]
	\centering
	\includegraphics[width=0.49\textwidth]{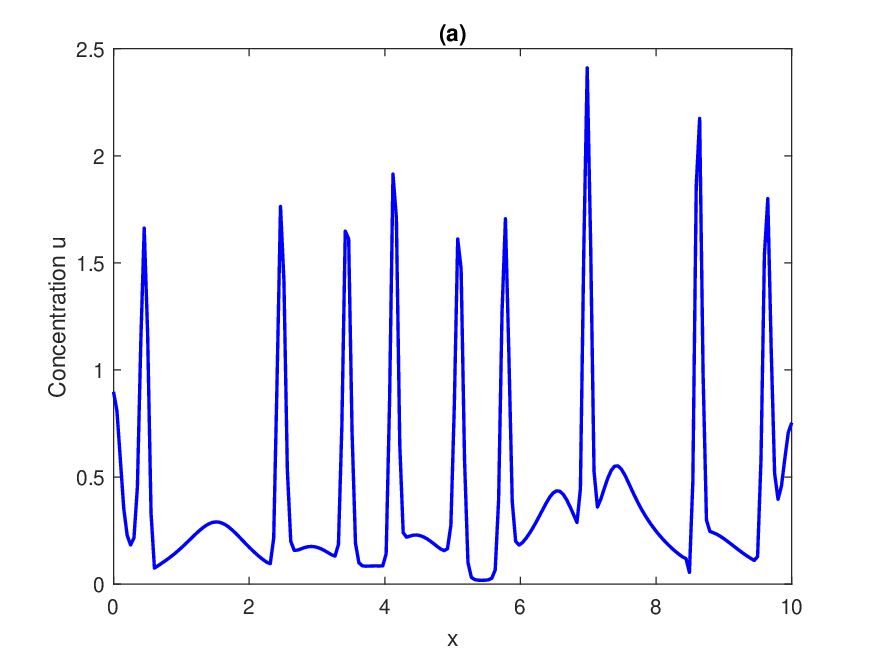}
	\includegraphics[width=0.49\textwidth]{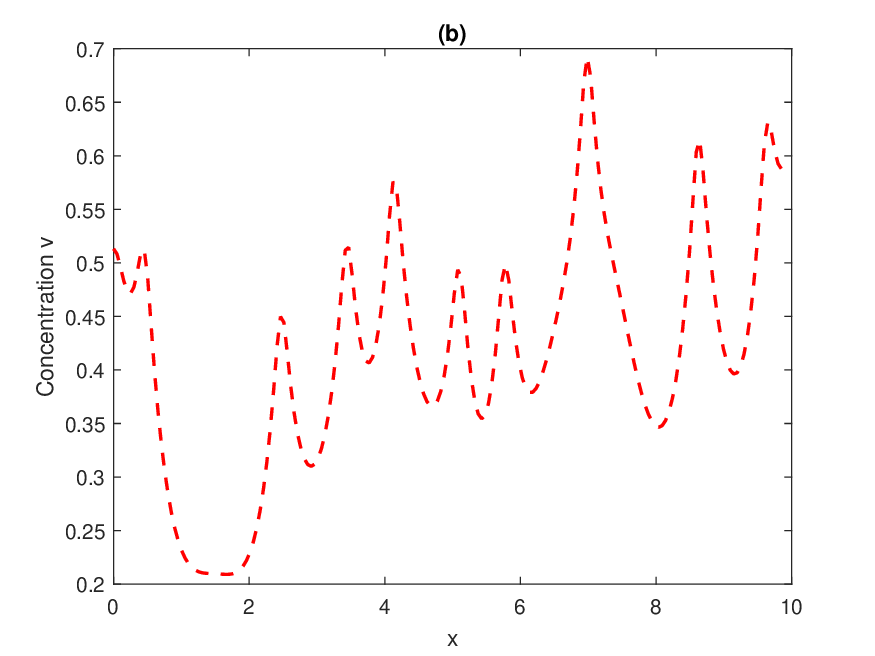}
	\includegraphics[width=0.49\textwidth]{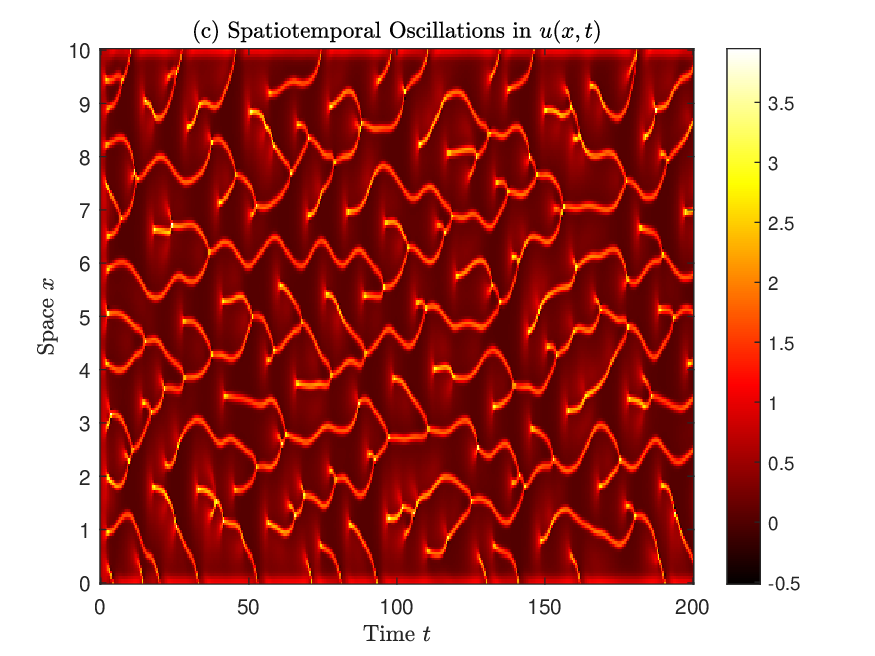}
	\includegraphics[width=0.49\textwidth]{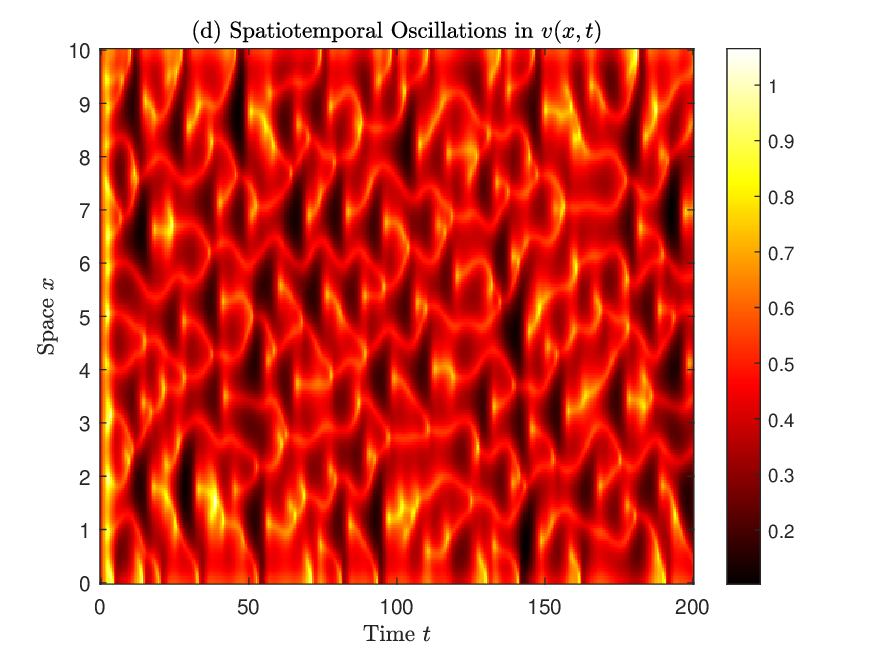}
	\caption{Plot (a) represents final spatial profile $u(x,T)$ at $T=200$. Peaks represent localized population clusters emerging from chemotactic aggregation. Plot (b) is the final spatial profile $v(x,T)$ at $T=200$. The chemoattractant profile reflects the population distribution but remains more diffused. Plots (c-d) Spatiotemporal evolution of cell density $u(x,t)$ and chemoattractant $v(x,t)$ beyond the Hopf bifurcation threshold $\chi = 0.5$, showing oscillatory pattern formation. }\label{Fig:1Deq1}
\end{figure}

\subsection{Numerical Simulation of the 2D Keller–Segel Chemotaxis Model (\ref{eq:KS1}) }

We performed numerical simulations of the two-dimensional Keller–Segel chemotaxis model with logistic growth and signal degradation, implemented using finite differences with periodic boundary conditions. The computational domain was set to $[0, L]^2$ with $L = 5$, and the spatial resolution was taken as $N_x = N_y = 256$. The initial bacterial distribution $u(x, y, 0)$ was prescribed as a Gaussian perturbation with added noise, centered at the domain midpoint, while the initial chemical concentration $v(x, y, 0)$ was chosen as a centered Gaussian without noise.

To ensure numerical stability and biologically realistic dynamics, the chemotactic sensitivity $\chi$ and logistic growth rate $r$ were moderately reduced. A small time step $\Delta t = 2 \times 10^{-4}$ was employed over a final times $T = 0.20$ (for top-plot) and $T = 50$ (for the bottom-plot) as displayed in Figure \ref{Fig:keller_segel_2D}.

\begin{figure}[h!]
	\centering
	\includegraphics[width=0.99\textwidth]{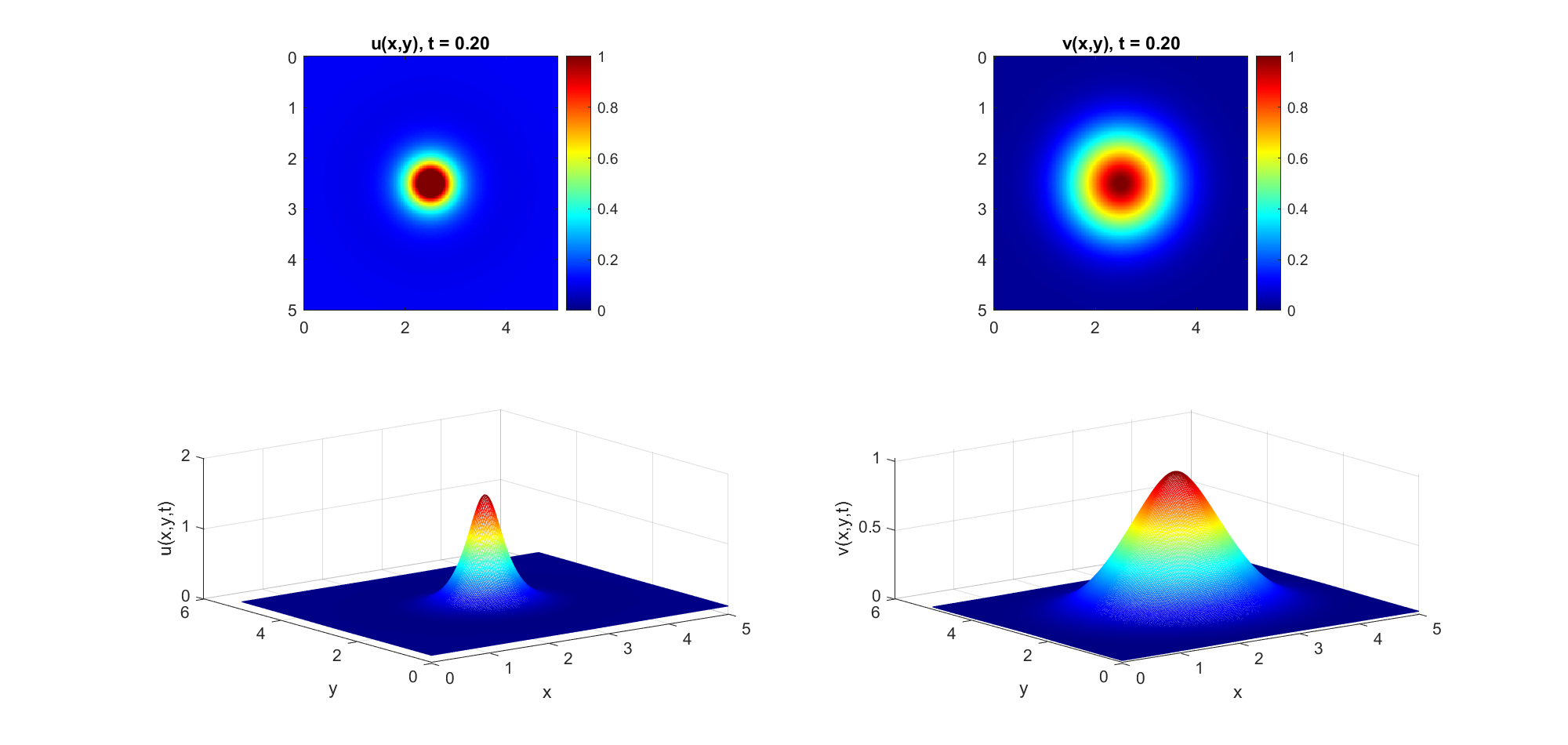}
	\includegraphics[width=0.99\textwidth]{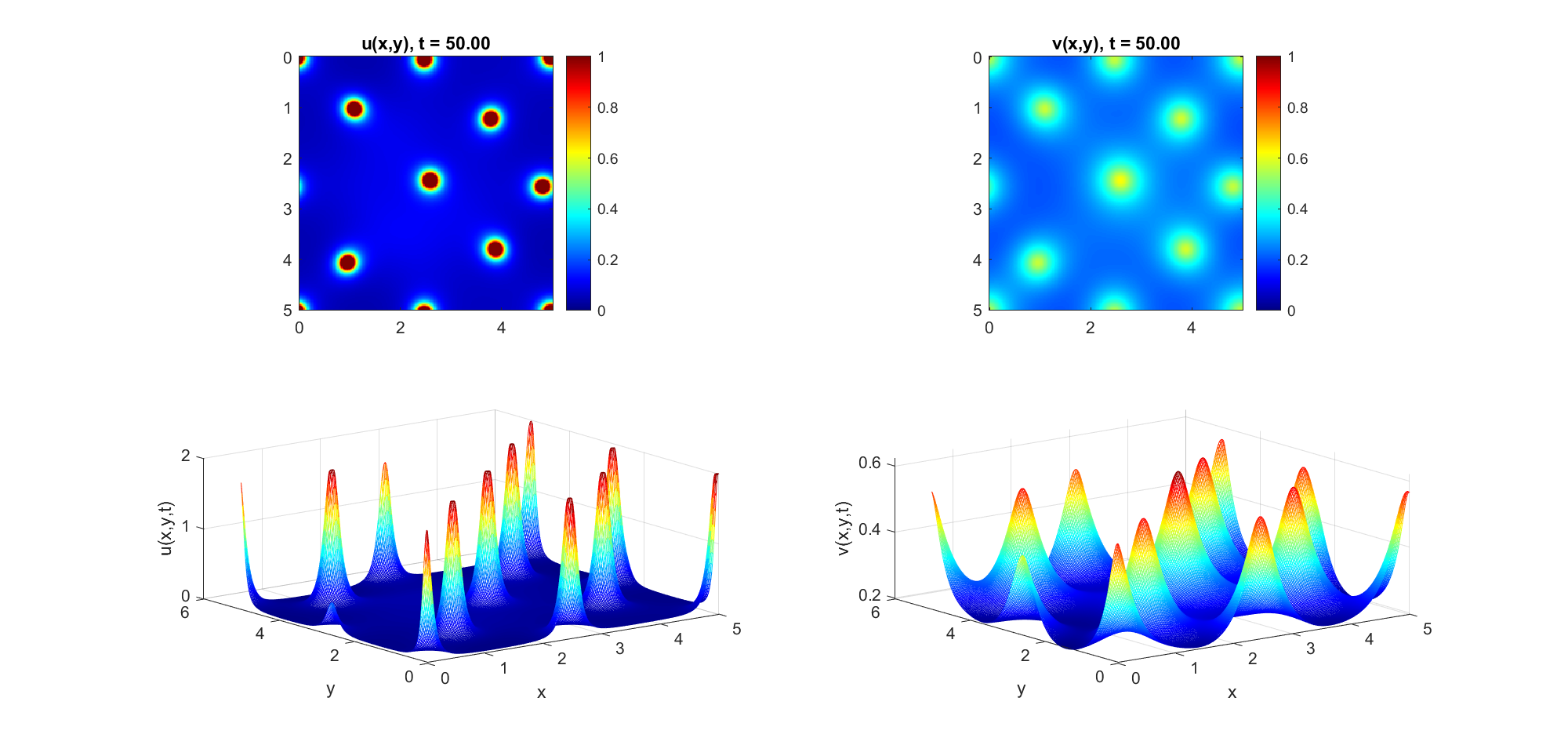}
	\caption{
		Time evolution of the 2D Keller–Segel chemotaxis system with logistic growth and signal degradation. Shown are snapshots of the bacterial density $u(x,y,t)$ and chemical signal $v(x,y,t)$ at $t = 0.20, 50$, visualized using 2D colormaps and 3D surface plots. The system exhibits stable aggregation patterns, where chemotactic movement toward higher chemical concentrations competes with diffusion and logistic regulation. Parameters: $\chi = 1.5$, $D = D_v = 0.1$, $r = 0.5$, $K = 1$, $\alpha = 1$, $\beta = 0.5$.
	}
	\label{Fig:keller_segel_2D}
\end{figure}

\begin{figure}[h!]
	\centering
	\includegraphics[width=0.85\textwidth]{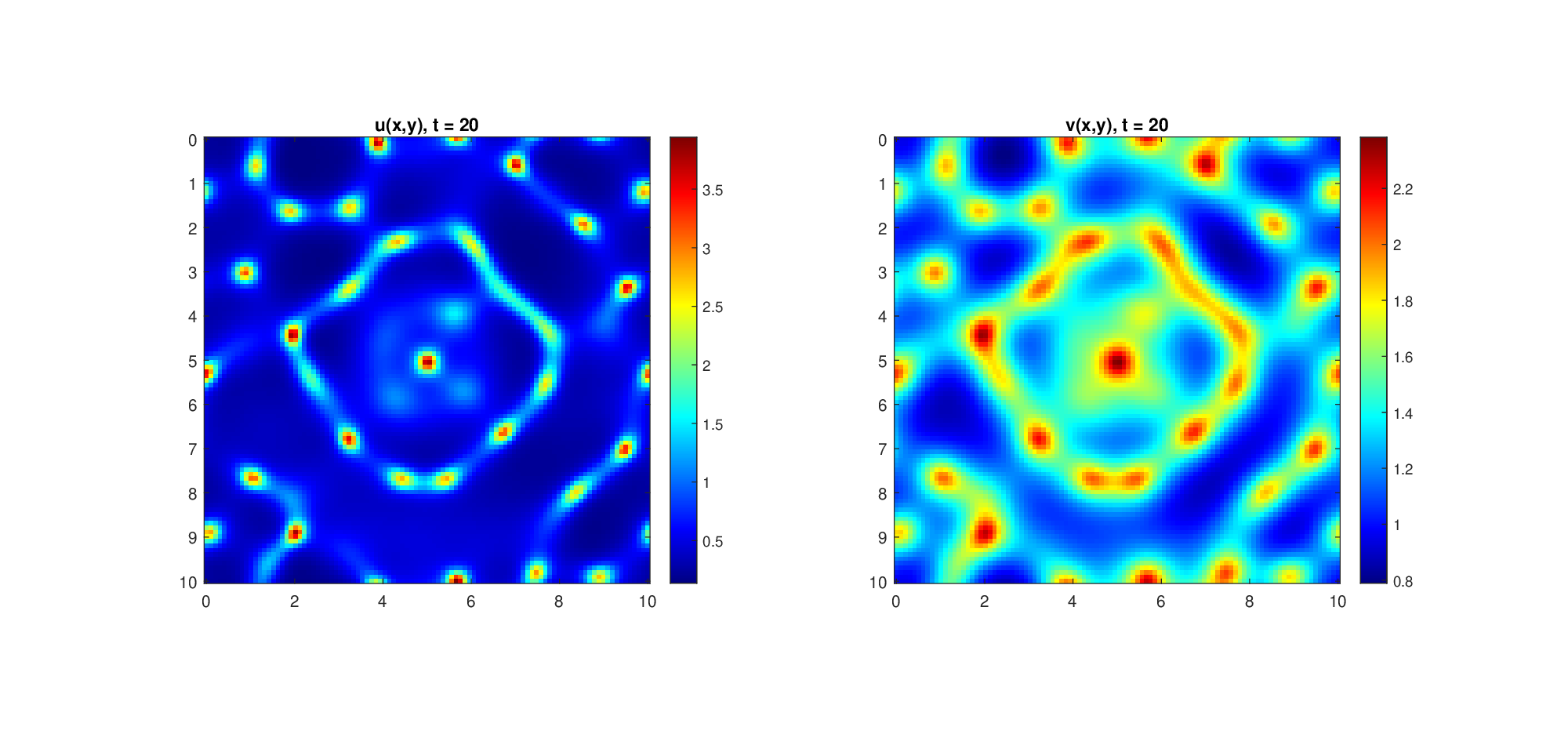}
	\includegraphics[width=0.85\textwidth]{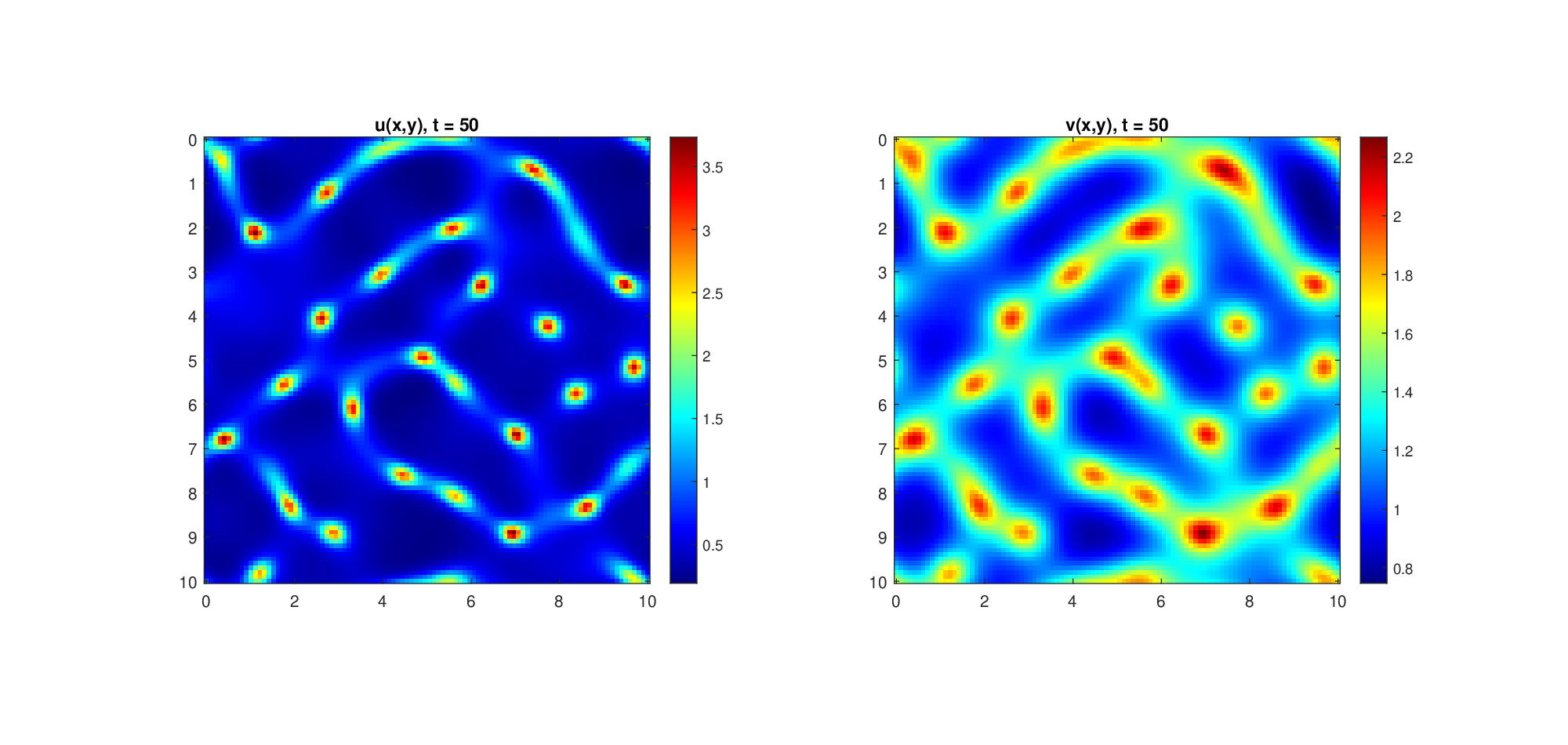}
	\includegraphics[width=0.85\textwidth]{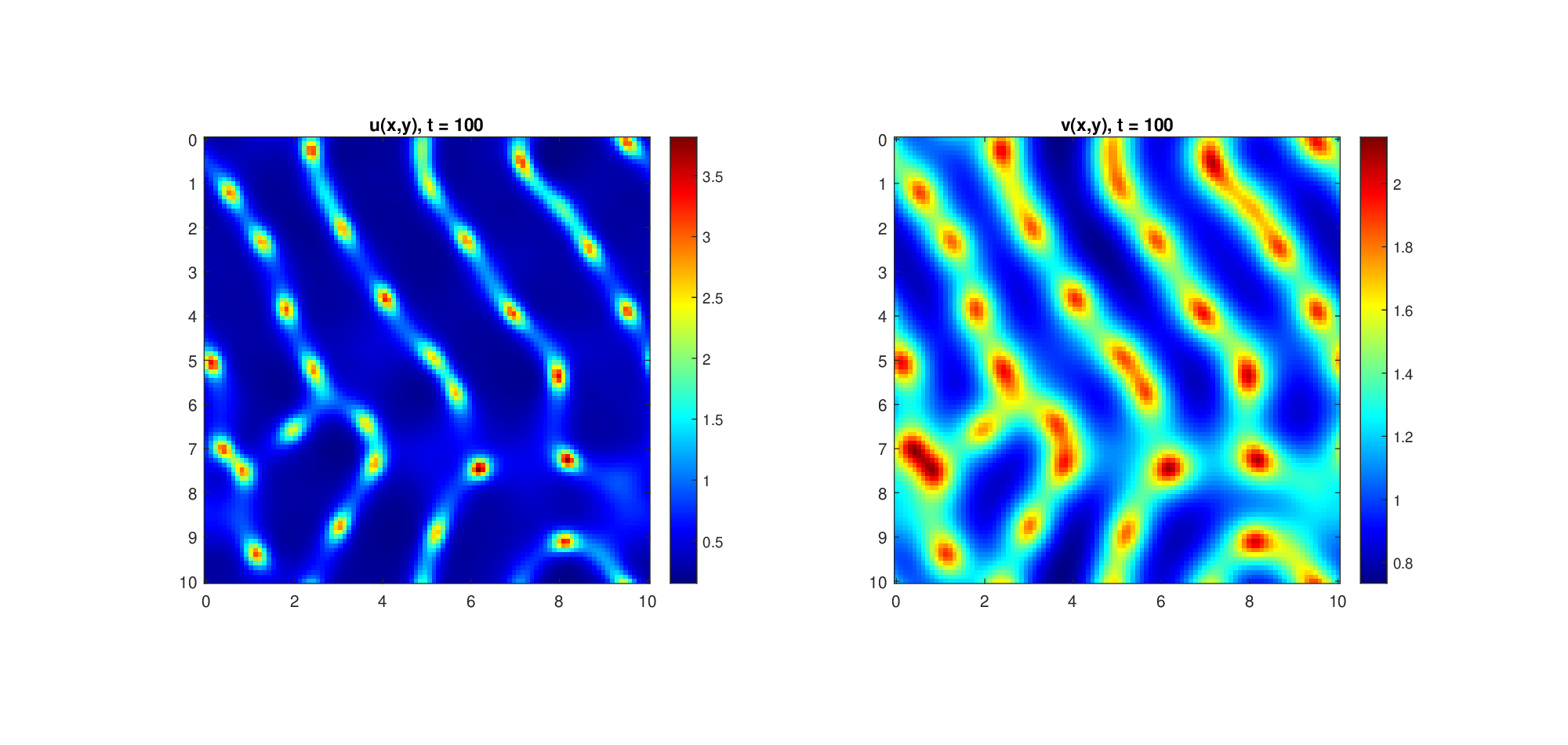}
	\caption{Different spatiotemporal evolution of the Keller-Segel model (\ref{eq:KS2}) for instances of final simulation time. }\label{Fig:keller_segel_2D1}
\end{figure}

As seen in Figure \ref{Fig:keller_segel_2D1}, the Keller--Segel--reaction model exhibits rich spatiotemporal dynamics, including the emergence of self-organized structures such as rings and stripes. These patterns arise from the interplay between chemotactic aggregation, nonlinear reaction kinetics, and diffusion.

In simulations with moderate chemotactic sensitivity and localized initial conditions, the bacterial density often evolves into ring-like formations. These rings represent a transient outward migration of cells followed by aggregation at the periphery, typically stabilized by a balance between chemotaxis and diffusion. Such phenomena are relevant to experimentally observed bacterial colony behaviors, where nutrients or signaling molecules accumulate at the boundary.

Alternatively, with stronger nonlinear coupling and parameter regimes promoting spatial instabilities, the system exhibits stripe-like patterns. These are typically aligned along the direction of anisotropic gradient development or symmetry breaking due to initial perturbations. The stripes may persist or eventually coalesce into localized aggregates depending on the system's parameters and boundary conditions.

To gain deeper insights into the spatial organization and intensity distribution of the evolving patterns, 2D surface plots were generated from the simulation results of the Keller--Segel--reaction system. These surface plots represent the population density $u(x, y, t)$ (or chemical concentration $v(x, y, t)$) as a height field over the spatial domain.

In Figure \ref{fig:ks_surface}, the ring-like structures observed in the density field appear as annular ridges, with sharp peaks indicating regions of high cellular aggregation. The surface topology clearly illustrates the void at the center and the encircling aggregation front, which is maintained due to the competition between outward diffusion and inward chemotactic pull.

In contrast, stripe-like configurations emerge as alternating ridges and valleys along preferred spatial directions. These patterns often arise from symmetry-breaking instabilities and suggest periodic alignment in cellular aggregation. The surface representation reveals the amplitude and periodicity of such stripes, offering quantitative detail not readily apparent from contour plots alone.

Surface plots are particularly valuable in detecting secondary instabilities, identifying localized peaks, and distinguishing between homogeneous diffusion and nonlinear aggregation dynamics.

\begin{figure}[h!]
	\centering
	\includegraphics[width=0.80\textwidth]{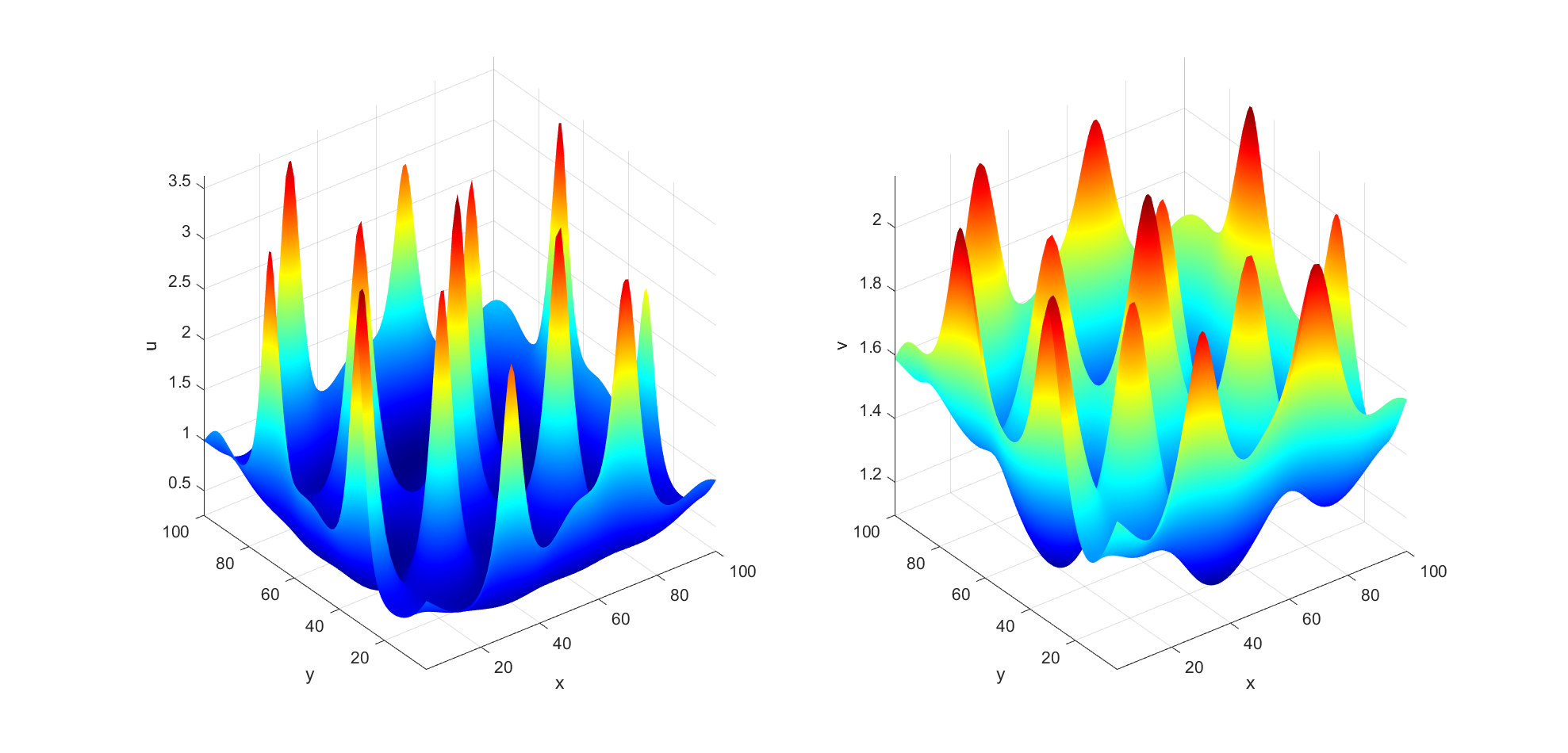}
	\includegraphics[width=0.80\textwidth]{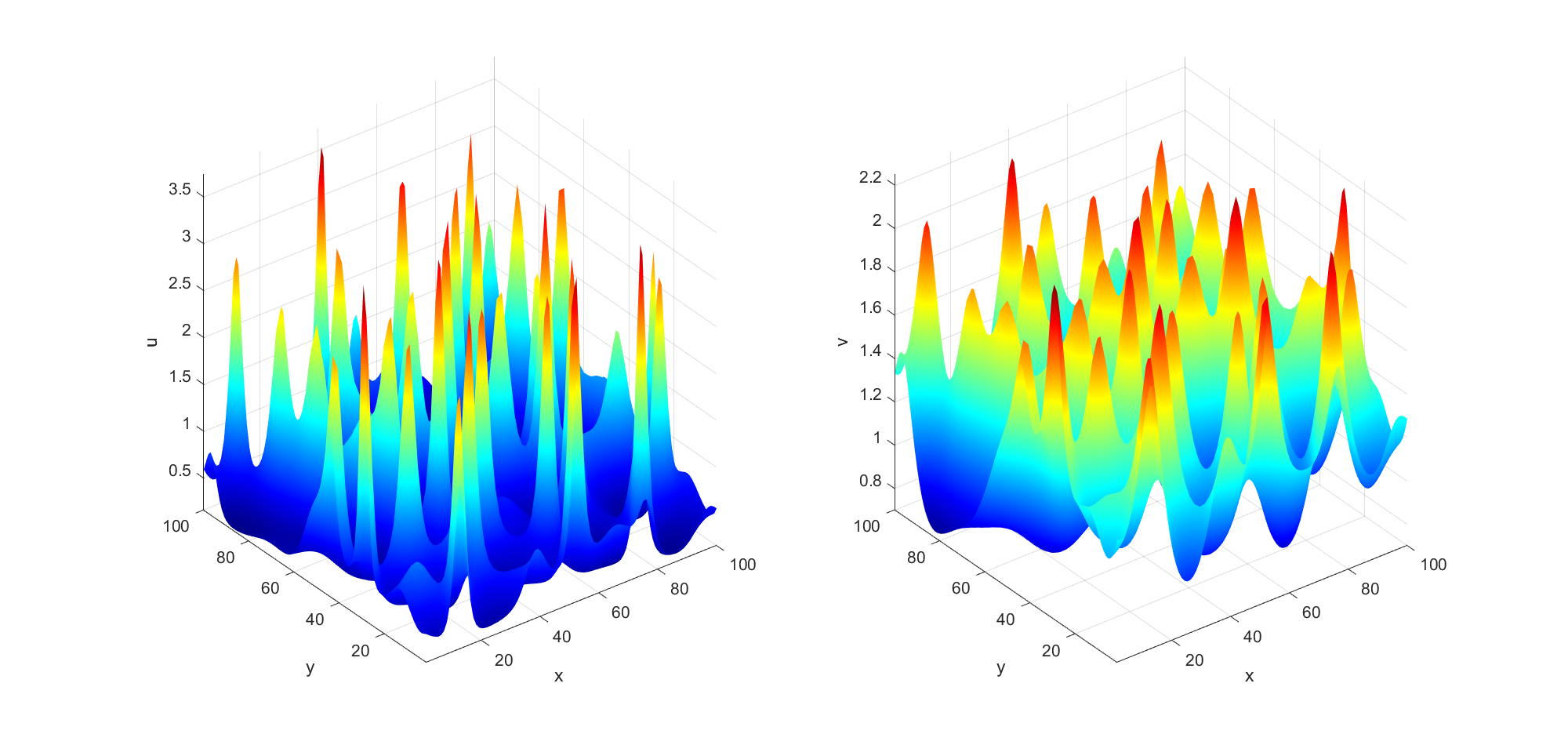}
	\includegraphics[width=0.80\textwidth]{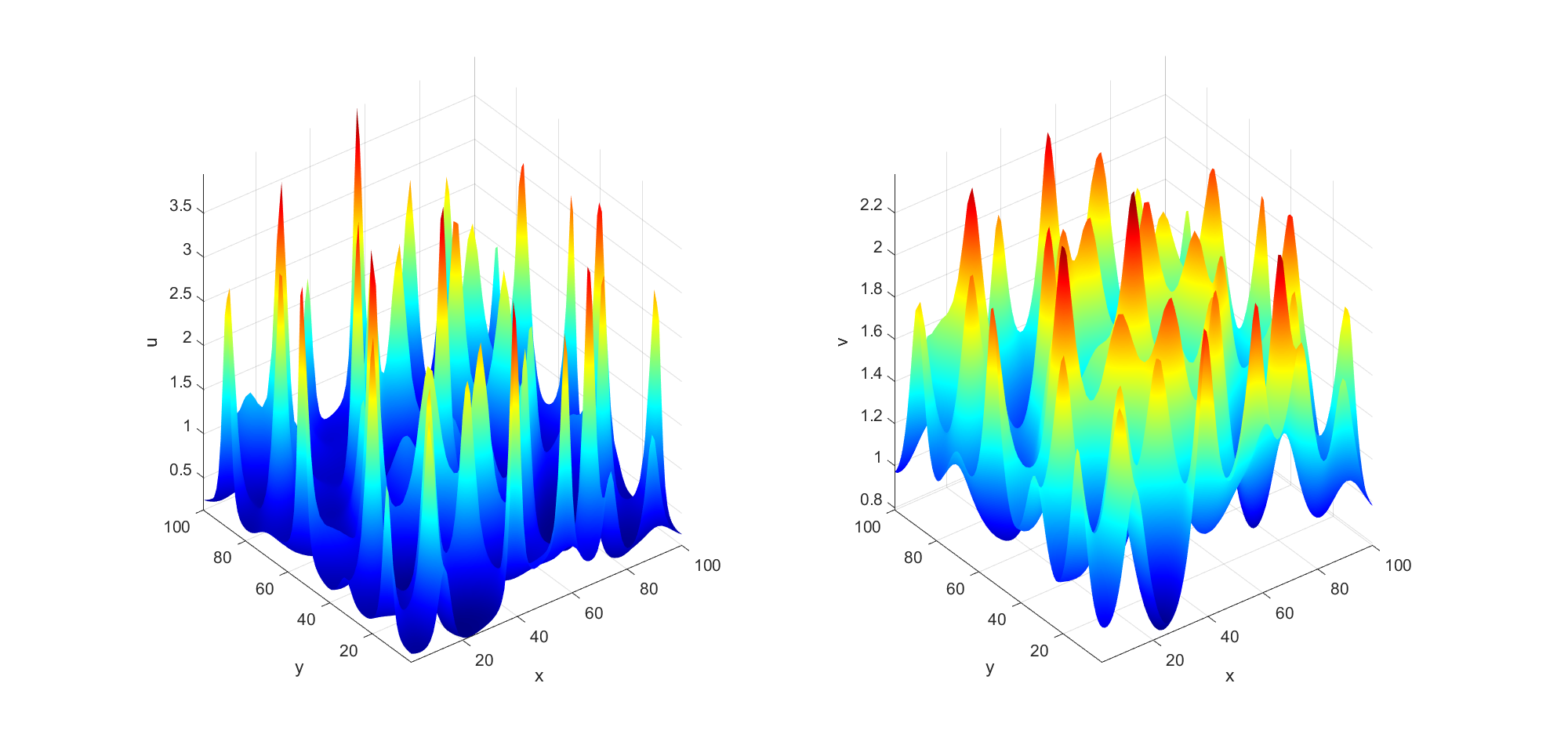}
	\caption{2D surface plots of Keller--Segel pattern formation. (Top) Ring-like structure in the population densities $u(x,y,t)$ and $v(x,y,t)$ visualized as a surface elevation at $t = 50$, showing strong peripheral aggregation. (Middle) Stripe-like periodic modulation at $t = 100$, with alternating high and low-density ridges. While the lower plots correspond to final simulation time $t=150$. Surface plots provide a 3D visualization of spatiotemporal pattern intensity.}
	\label{fig:ks_surface}
\end{figure}

The simulation results reveal the emergence of stable, non-uniform spatial patterns due to chemotactic aggregation counteracted by diffusion and logistic saturation. Over time, the bacterial density $u(x, y, t)$ forms spot-like or ring-like patterns concentrated near the center, while the chemical signal $v(x, y, t)$ remains smoother due to diffusion and degradation. These patterns are qualitatively consistent with the expected behavior of Keller–Segel-type systems in the parameter regime where chemotactic aggregation is balanced by stabilizing terms.

The use of periodic operators and bounded updates (clipping values between 0 and 2) prevents numerical blow-up. Visualizations at regular intervals highlight the temporal evolution and eventual stabilization of the patterns.

%%%%%%%%%%%%%%%%%%%%%%%%%%%%%%%%%%%%%%%%%%%%%%%%%%%%%%%%%%

\subsubsection*{Numerical Observations}
Numerical simulations in 1D and 2D confirm the emergence of oscillatory and chaotic spatiotemporal dynamics. The presence of two species with opposing chemotactic sensitivities (e.g., \( \chi_1 > 0, \chi_2 < 0 \)) is a crucial driver for this complexity, as it introduces competition and reinforcement mechanisms that destabilize the chemical signal \( v \).

\begin{Remark}
	For a complete bifurcation analysis, one could compute the dispersion relation \( \lambda(k) \) numerically and identify critical points \( k_c \) where \( \text{Re}(\lambda(k)) = 0 \), thereby characterizing the onset of instability. Further, numerical continuation tools (e.g., AUTO, MatCont) may be employed for full bifurcation diagrams.
\end{Remark}
%%%%%%%%%%%%%%%%%%%%%%%%%%%%%%%%%%%%%%%%%%%%%%%%%%%%%%%

\subsection*{1D Numerical Simulation and Observed Dynamics}

To complement the higher-dimensional studies, we implemented a one-dimensional (1D) finite difference simulation of the two-species, one-chemical Keller--Segel chemotaxis model. The governing equations describe the evolution of two interacting species \( u_1(x,t) \) and \( u_2(x,t) \), which respond chemotactically to a common chemical signal \( v(x,t) \), modeled as:
\begin{equation}
\begin{split}
	\partial_t u_1 &= D_1 \partial_{xx} u_1 - \chi_1 \partial_x(u_1 \partial_x v), \\
	\partial_t u_2 &= D_2 \partial_{xx} u_2 - \chi_2 \partial_x(u_2 \partial_x v), \\
	\partial_t v   &= D_v \partial_{xx} v + \alpha_1 u_1 + \alpha_2 u_2 - \beta v.
\end{split}
\end{equation}

We imposed homogeneous Neumann boundary conditions and initialized all species with small random perturbations around constant states. The parameter choice \( \chi_1 > 0 \), \( \chi_2 < 0 \) corresponds to a setting where the first species is attracted to the signal, while the second species is repelled. This antagonistic chemotactic interaction leads to complex spatial dynamics.

\paragraph{Emergence of Spatiotemporal Chaos.} The numerical results reveal that both \( u_1 \) and \( u_2 \) undergo persistent, irregular oscillations in space and time—characteristic of spatiotemporal chaos. These dynamics closely resemble the 2D case, despite the reduced spatial dimension. The observed chaotic behavior arises from the interplay between chemotactic attraction/repulsion and reaction–diffusion dynamics, which amplifies initial perturbations and leads to non-periodic, yet bounded, fluctuations in the population densities.

\paragraph{Biological Interpretation.} 
The 1D chaotic oscillations indicate that even in a minimal spatial setting, multi-species chemotaxis models can exhibit rich, nonlinear pattern dynamics. This underscores the importance of chemotactic interactions—particularly mixed attractive–repulsive responses—in driving complex population-level behaviors such as segregation, pulse formation, or oscillatory migration.

\subsection{2D Numerical Results and Discussion}

We numerically investigated the dynamics of a two-species, one-chemical Keller--Segel chemotaxis system using both the finite difference method (FDM) and the split-step Fourier method (SSFM). The equations governing the system are:
\begin{equation}
	\begin{aligned}
		\partial_t u_1 &= D_1 \Delta u_1 - \chi_1 \nabla \cdot (u_1 \nabla v), \\
		\partial_t u_2 &= D_2 \Delta u_2 - \chi_2 \nabla \cdot (u_2 \nabla v), \\
		\partial_t v &= D_v \Delta v + \alpha_1 u_1 + \alpha_2 u_2 - \beta v,
	\end{aligned}
\end{equation}
where \( u_1(x,y,t) \) and \( u_2(x,y,t) \) denote the densities of two interacting species, and \( v(x,y,t) \) is the concentration of a common chemoattractant. The species diffuse at rates \( D_1 \) and \( D_2 \), and their movement is modulated by chemotactic sensitivities \( \chi_1 \) and \( \chi_2 \), respectively. The signal \( v \) is produced by both species and decays at rate \( \beta \).

We performed numerical simulations of a two-dimensional chemotaxis–reaction model consisting of two interacting species $u_1(x,y,t)$ and $u_2(x,y,t)$ that respond chemotactically to a diffusible chemical signal $v(x,y,t)$. The simulation was conducted on a square domain $\Omega = [0,2] \times [0,2]$ with spatial resolution $\Delta x = \Delta y = 0.04$, corresponding to $100 \times 100$ grid points. The total simulation time was $T = 4.0$ with a time step size $\Delta t = 10^{-3}$, resulting in 4000 time steps. Neumann (zero-flux) boundary conditions were enforced for all three fields.

The diffusion coefficients were chosen as $D_1 = D_2 = 10^{-3}$ for the species and $D_v = 10^{-2}$ for the chemoattractant, reflecting higher mobility of the signal compared to the organisms. The chemotactic sensitivities were $\chi_1 = 0.05$ and $\chi_2 = -0.05$, indicating that species $u_1$ is attracted to the chemical signal while $u_2$ is repelled. The production parameters were $\alpha_1 = 1$ and $\alpha_2 = 0.001$, suggesting that $u_1$ is the dominant source of $v$, while $u_2$ contributes marginally. The decay rate of the chemical signal was set to $\beta = 0.001$.

Initial conditions for $u_1$ and $u_2$ consisted of small random perturbations around the uniform state $0.5$, while $v$ was initialized as zero throughout the domain. The Laplacian operator was discretized using a five-point stencil, and chemotactic fluxes were computed using central differences for the gradients and divergences.

As the system evolved, spatially heterogeneous patterns emerged. Despite their opposite chemotactic signs, both $u_1$ and $u_2$ exhibited remarkably similar aggregation dynamics, forming clusters that closely mirrored each other. This synchronization is likely due to the dominant production of the chemoattractant by $u_1$, which drives both species via direct attraction and indirect coupling through $v$. The chemoattractant field $v(x,y,t)$ developed smoother spatial gradients but clearly reflected the underlying population structure.

\begin{figure}[h!]
	\centering
	\includegraphics[width=\textwidth]{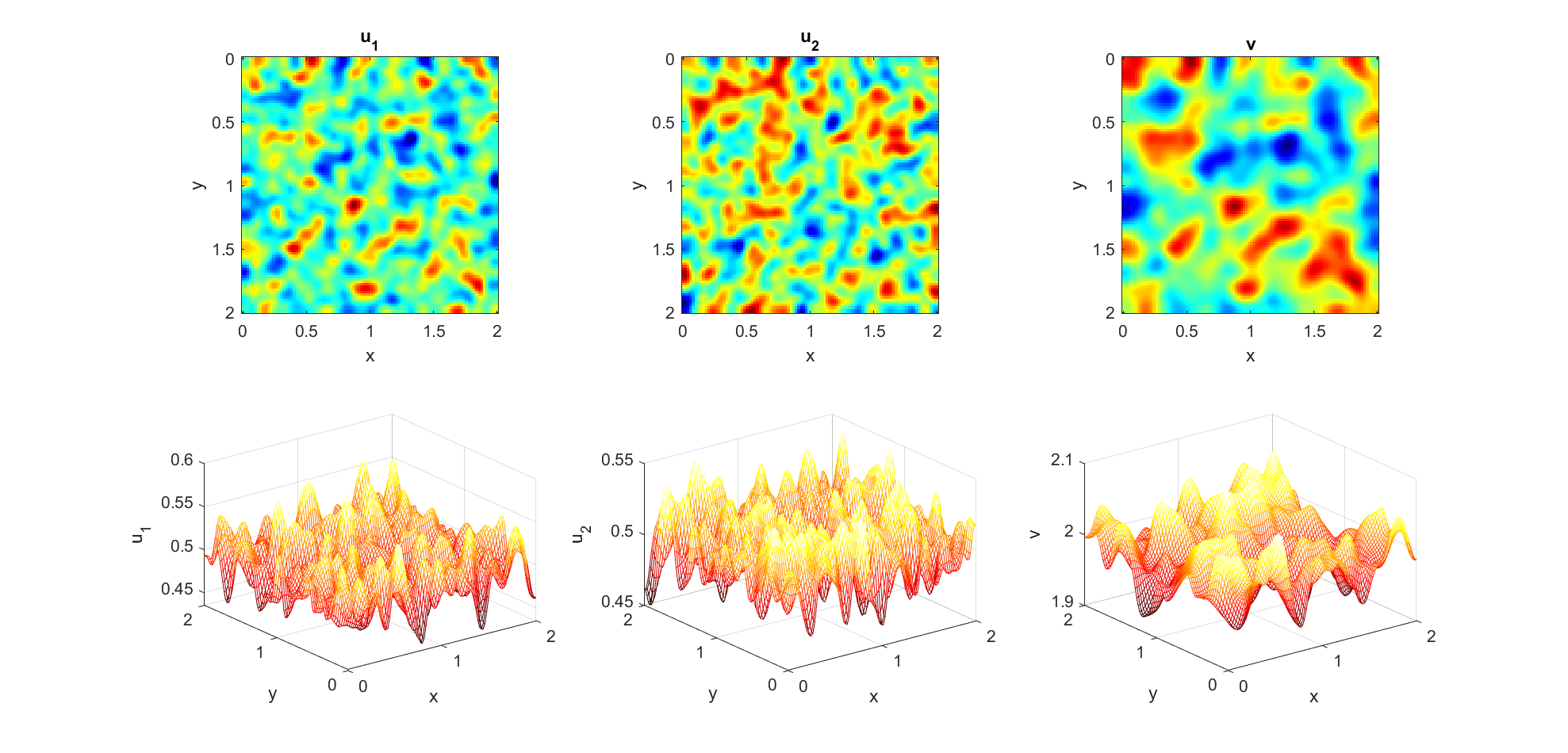}
	\caption{Time evolution of the two species $u_1$, $u_2$, and the chemoattractant $v$ in two dimensions. Each row shows density plots and 3D mesh plots at selected time points. Despite having opposite chemotactic sensitivities, $u_1$ and $u_2$ evolve with strikingly similar spatial patterns, influenced primarily by the chemoattractant dynamics.}
	\label{fig:chemotaxis_patterns}
\end{figure}

\subsubsection{Qualitative Dynamics and Spatiotemporal Patterns}

Numerical simulations reveal that the coupled nonlinear PDE system exhibits rich spatiotemporal dynamics. Both numerical schemes, FDM and SSFM, consistently demonstrate the emergence of complex oscillatory structures in space and time. The behavior includes:

\textbf{Initial pattern formation:} Starting from slightly perturbed homogeneous initial data, the system evolves into nontrivial spatial structures. These patterns are initiated by chemotactic aggregation as the organisms move toward regions of higher chemical concentration.
	
\textbf{Spatiotemporal oscillations:} Over time, the patterns do not settle into a static configuration. Instead, they exhibit persistent fluctuations in amplitude and position, indicative of \emph{spatiotemporal chaos}. These oscillations arise from the nonlinear feedback between the two species and the chemoattractant field.
	
 \textbf{Species segregation and interaction:} Depending on the signs and magnitudes of \( \chi_1 \) and \( \chi_2 \), species either co-aggregate or spatially segregate. In the case where one species is chemorepellent (e.g., \( \chi_2 < 0 \)), ring-like or halo patterns appear, with one species avoiding regions of high \( v \), while the other concentrates there.

\subsubsection{Comparison of Numerical Methods}
Both numerical methods accurately reproduce the qualitative features of the model. 
 The \textbf{finite difference method (FDM)} is simple to implement and supports zero-flux (Neumann) boundary conditions, making it suitable for modeling closed systems. The results clearly exhibit spatially structured oscillations in all three fields. The \textbf{split-step Fourier method (SSFM)} offers higher spectral accuracy and is particularly well-suited for periodic domains. It also reproduces chaotic dynamics but with sharper resolution and less numerical diffusion, highlighting the fine-scale interactions between species.

These numerical experiments suggest that chemotactic interactions in multi-species systems can lead to complex, potentially chaotic behavior even in minimal models. Key implications include:

 \textbf{Ecological segregation:} The model can explain spatial niche formation, where competing species avoid each other spatially through chemorepulsion.
\textbf{Oscillatory population dynamics:} The sustained spatiotemporal oscillations may be indicative of predator-prey-like or competitive oscillations observed in microbial colonies.
\textbf{Sensitivity to initial conditions and parameters:} Small changes in diffusion coefficients or chemotactic sensitivities drastically alter the emergent patterns, hinting at a sensitive dependence characteristic of chaotic systems.

Our simulations demonstrate that multi-species chemotaxis systems are capable of producing a wide range of spatiotemporal patterns, including chaotic oscillations, aggregation, and segregation. These findings underscore the richness of chemotaxis-driven ecological dynamics and highlight the need for further analytical and numerical study in higher dimensions and more biologically detailed models.

\begin{figure}[htbp]
	\centering
	\includegraphics[width=\textwidth]{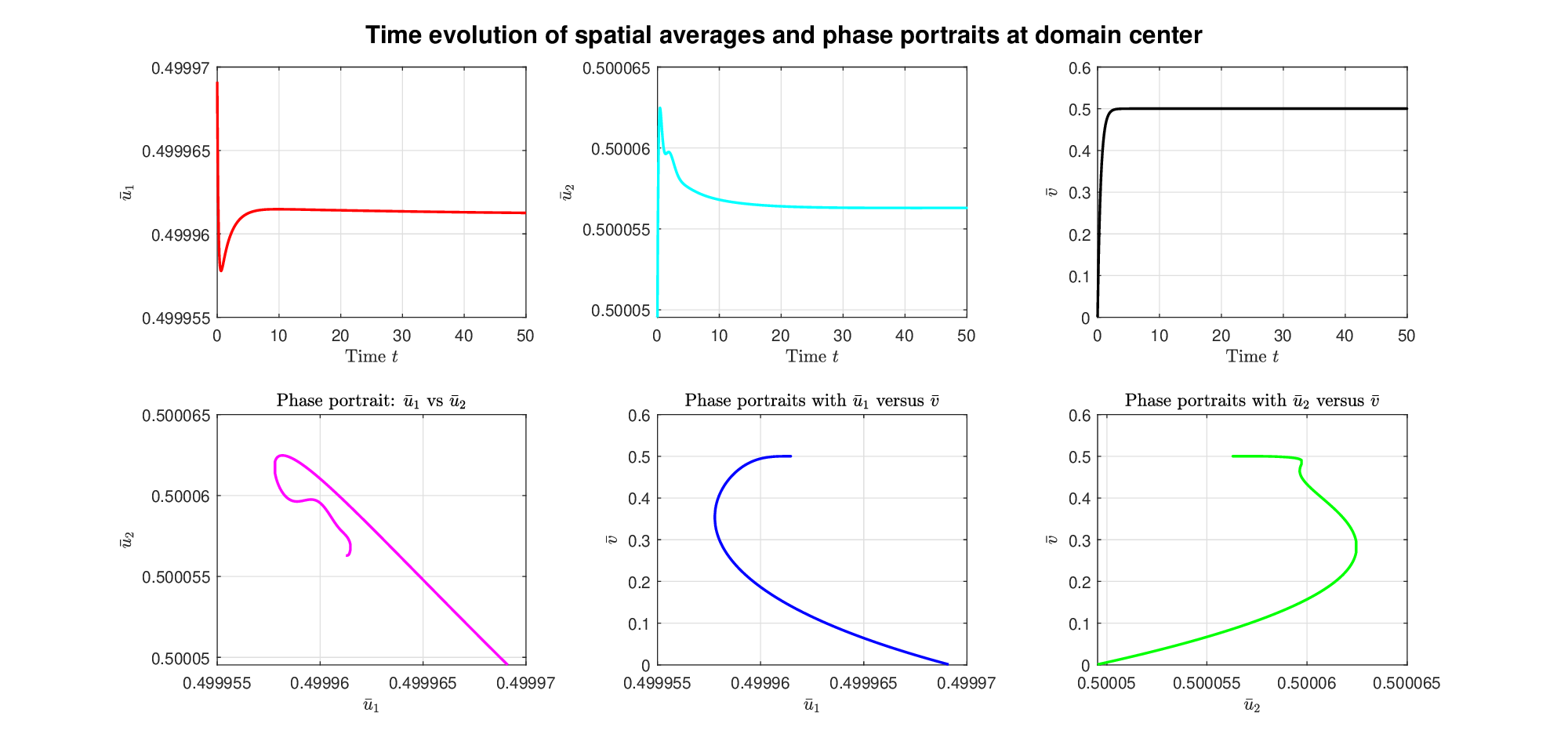}
	\includegraphics[width=\textwidth]{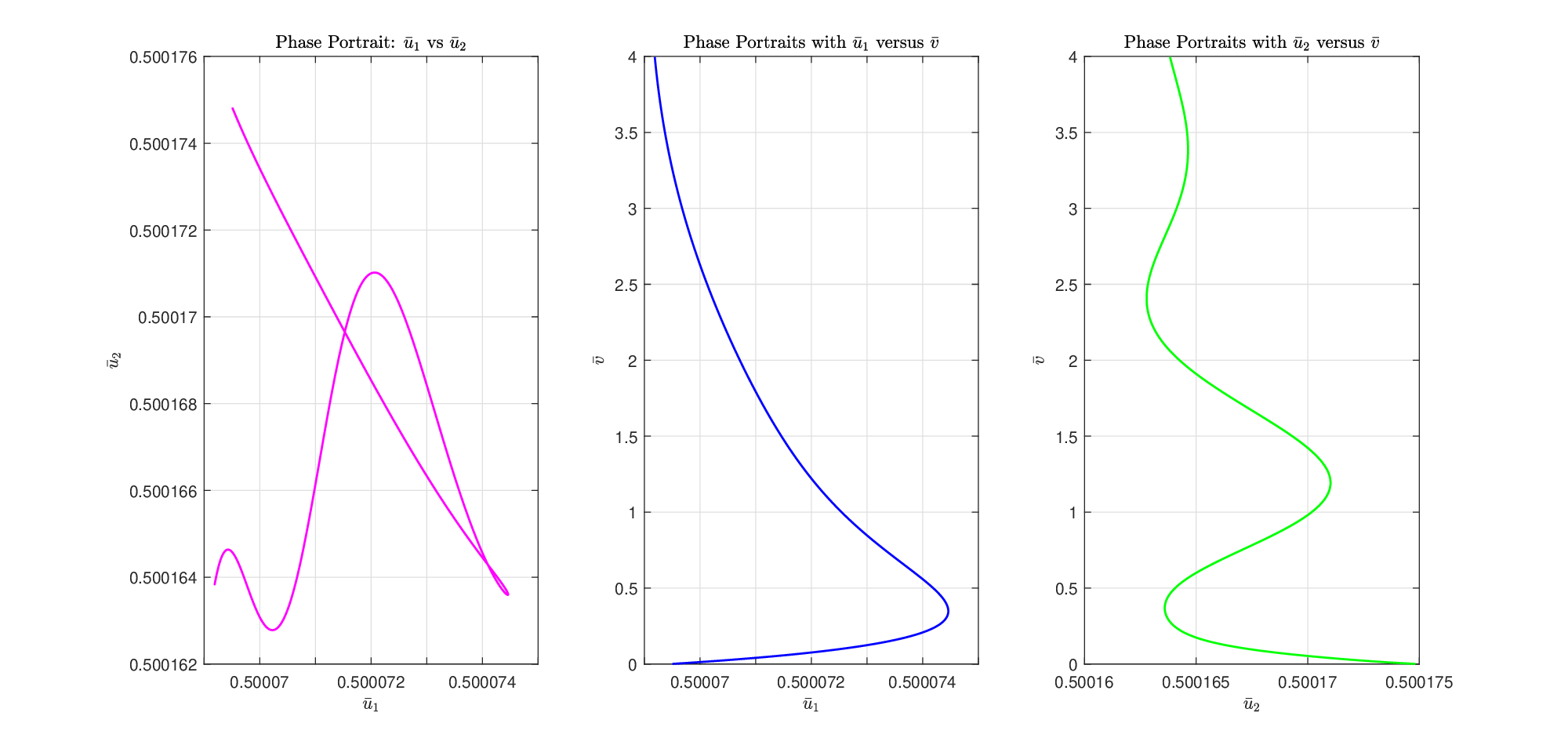}
	\caption{Time evolution and phase portrait of species densities $u_1$, $u_2$, and chemical $v$ at the domain center $(x_0, y_0)$. The plot reveals irregular oscillations and interactions characteristic of spatiotemporal chaos.}
	\label{fig:timeseries_center}
\end{figure}

\subsection{Temporal Oscillations and Chaotic Signatures}

To further understand the dynamics, we tracked the time evolution of $u_1$, $u_2$, and $v$ at a fixed point in space (e.g., the domain center). Figure~\ref{fig:timeseries_center} shows that all three variables exhibit irregular, bounded oscillations over time. The absence of periodicity and sensitivity to initial conditions are indicative of a spatiotemporally chaotic regime, often seen in multi-species chemotaxis systems with strong nonlinear coupling. These results complement the spatial patterns observed and suggest complex ecological interactions in the modeled system.

\subsection*{Comparison of 1D and 2D Instabilities}
Numerical simulations in 1D reveal chaotic oscillations in both species \( u_1 \) and \( u_2 \), which align qualitatively with those observed in 2D. However, 2D simulations show additional spatial complexity such as rotating spirals, stripes, or spot–annulus formations.

The key differences include:
\begin{itemize}
	\item In 1D, chaos is primarily temporal with mild spatial heterogeneity.
	\item In 2D, interaction of diffusion and chemotaxis along two axes promotes richer bifurcations and pattern multiplicity.
	\item The spectrum of unstable modes is broader in 2D, allowing for more wave interactions and thus more diverse spatiotemporal behavior.
\end{itemize}
These findings suggest that dimensionality plays a critical role in the pattern-forming capacity of multi-species chemotaxis systems.

%%%%%%%%%%%%%%%%%%%%%%%%%%%%%%%%%%%%%%%%%%%%%%%

\subsection{1D Chemotaxis--Fluid Coupling: Numerical Results}

To complement the two-dimensional simulations, we analyze the dynamics of the chemotaxis--fluid coupling system in a one-dimensional spatial domain \( x \in [0, L] \). The governing equations reduce to:

\begin{equation}\label{eq:main_system}
	\begin{split}
		\partial_t u + w \partial_x u &= D\, \partial_{xx} u - \chi\, \partial_x \left( u\, \partial_x v \right) + r\,u\left(1 - \frac{u}{K}\right), \\
		\partial_t v + w \partial_x v &= D_v\, \partial_{xx} v + \alpha u - \beta v, \\
		\partial_t w + w \partial_x w &= -\partial_x p + \nu\, \partial_{xx} w + \kappa u, \qquad \partial_x w = 0,
	\end{split}
\end{equation}
with periodic or no-flux boundary conditions. The scalar fluid velocity \( w(x,t) \) replaces the 2D vector field \( \mathbf{w} \), and the incompressibility condition simplifies to \( \partial_x w = 0 \), implying a constant or spatially averaged flow in many implementations.

In what follows, we analyze the system (\ref{eq:main_system}) 
posed on a bounded domain $\Omega = (0,L)$ with no-flux boundary conditions:
\[
\partial_x u = \partial_x v = \partial_x w = 0 \quad \text{on } \partial\Omega.
\]

We assume the initial data:
\[
u(x,0) = u_0(x) \geq 0, \quad v(x,0) = v_0(x) \geq 0, \quad w(x,0) = w_0(x),
\]
are smooth and compatible with the boundary conditions.

\begin{Theorem}[Global Existence of Weak Solutions]
	\label{thm:global_weak}
	Let the initial data $(u_0, v_0, w_0) \in H^1(\Omega)^3$, and assume that all parameters $D, D_v, \nu, \chi, r, K, \alpha, \beta, \kappa$ are positive. Then there exists a global weak solution $(u,v,w)$ to the system
	\begin{equation}
		\label{eq:main_system1}
		\begin{split}
			\partial_t u + w \partial_x u &= D\, \partial_{xx} u - \chi\, \partial_x \left( u\, \partial_x v \right) + r\,u\left(1 - \frac{u}{K}\right), \\
			\partial_t v + w \partial_x v &= D_v\, \partial_{xx} v + \alpha u - \beta v, \\
			\partial_t w + w \partial_x w &= -\partial_x p + \nu\, \partial_{xx} w + \kappa u, \qquad \partial_x w = 0,
		\end{split}
	\end{equation}
	on a bounded domain $\Omega \subset \mathbb{R}$ with smooth boundary and homogeneous Neumann boundary conditions. The solution satisfies
	\[
	u, v, w \in L^\infty(0,T;H^1(\Omega)) \cap L^2(0,T;H^2(\Omega)), \quad \text{for all } T>0,
	\]
	and $u(x,t) \geq 0$ almost everywhere in $\Omega \times (0,\infty)$.
\end{Theorem}

\begin{proof}
	We use a standard Galerkin approximation scheme combined with a priori estimates and compactness arguments.
	
	\paragraph{Step 1: Galerkin approximation.}
	Let $\{\phi_k\}_{k=1}^\infty$ be an orthonormal basis of $H^1(\Omega)$ consisting of eigenfunctions of the Neumann Laplacian. For each $n \in \mathbb{N}$, define approximate solutions:
	\[
	u_n(x,t) = \sum_{k=1}^n a_k^n(t)\phi_k(x), \quad v_n(x,t) = \sum_{k=1}^n b_k^n(t)\phi_k(x), \quad w_n(x,t) = \sum_{k=1}^n c_k^n(t)\phi_k(x),
	\]
	and insert into the weak form of the system. This yields a system of ODEs for the coefficients $\{a_k^n(t), b_k^n(t), c_k^n(t)\}$ with initial data projected onto the subspace.
	
	\paragraph{Step 2: A priori estimates.}
	Multiply the first equation by $u_n$, integrate by parts:
	\[
	\frac{1}{2}\frac{d}{dt}\|u_n\|_{L^2}^2 + D\|\partial_x u_n\|_{L^2}^2 = \chi \int u_n \partial_x u_n \partial_x v_n \, dx + r \int u_n^2 \left(1 - \frac{u_n}{K} \right) dx.
	\]
	
	Using integration by parts, the Cauchy–Schwarz and Young's inequalities:
	\[
	\chi \left| \int u_n \partial_x u_n \partial_x v_n \, dx \right| \leq \frac{D}{2} \|\partial_x u_n\|_{L^2}^2 + \frac{\chi^2}{2D} \|u_n \partial_x v_n\|_{L^2}^2.
	\]
	But since $\|u_n \partial_x v_n\|_{L^2} \leq \|u_n\|_{L^\infty} \|\partial_x v_n\|_{L^2}$, and $H^1(\Omega) \hookrightarrow L^\infty(\Omega)$ in 1D, we get:
	\[
	\chi \left| \int u_n \partial_x u_n \partial_x v_n \, dx \right| \leq C \|\partial_x u_n\|_{L^2}^2 + C \|u_n\|_{H^1}^2 \|\partial_x v_n\|_{L^2}^2.
	\]
	
	The logistic term satisfies:
	\[
	r \int u_n^2 \left(1 - \frac{u_n}{K} \right) dx \leq r \int u_n^2 dx.
	\]
	
	Hence, we obtain:
	\[
	\frac{d}{dt} \|u_n\|_{L^2}^2 + D \|\partial_x u_n\|_{L^2}^2 \leq C \|u_n\|_{H^1}^2 \|\partial_x v_n\|_{L^2}^2 + C \|u_n\|_{L^2}^2.
	\]
	
	Similar estimates for $v_n$ and $w_n$ yield:
	\[
	\frac{d}{dt} \|v_n\|_{L^2}^2 + D_v \|\partial_x v_n\|_{L^2}^2 \leq \alpha \int u_n v_n \, dx - \beta \|v_n\|_{L^2}^2 + C \|v_n\|_{L^2}^2.
	\]
	\[
	\frac{d}{dt} \|w_n\|_{L^2}^2 + \nu \|\partial_x w_n\|_{L^2}^2 \leq C \|w_n\|_{L^2}^2 + C \|u_n\|_{L^2}^2.
	\]
	
	Applying Grönwall's inequality and Sobolev embeddings, we obtain uniform in $n$ bounds for:
	\[
	\|u_n\|_{L^\infty(0,T;H^1)} + \|u_n\|_{L^2(0,T;H^2)} \leq C_T, \quad \text{and similarly for } v_n, w_n.
	\]
	
	\paragraph{Step 3: Compactness and limit passage.}
	Using the Aubin–Lions lemma, the sequences $\{u_n\}, \{v_n\}, \{w_n\}$ are relatively compact in $L^2(0,T;H^1(\Omega))$. Weak and strong convergence allows passage to the limit in the weak formulation, yielding a global weak solution $(u,v,w)$ satisfying the stated regularity.
	
	\paragraph{Step 4: Non-negativity of $u$.}
	Since $u_0 \geq 0$, and the nonlinear terms preserve non-negativity (especially the logistic source and the chemotactic flux is conservative), standard maximum principle arguments or approximation from below by positive solutions yield that $u(x,t) \geq 0$ for all $t > 0$.
	
\end{proof}

\begin{Theorem}[Finite-Time Blow-up without Logistic Damping]
	\label{thm:blowup}
	Consider system \eqref{eq:main_system1} with $r = 0$, i.e.,
	\begin{equation}
		\label{eq:reduced_system}
		\begin{split}
			\partial_t u + w \partial_x u &= D\, \partial_{xx} u - \chi\, \partial_x \left( u\, \partial_x v \right), \\
			\partial_t v + w \partial_x v &= D_v\, \partial_{xx} v + \alpha u - \beta v, \\
			\partial_t w + w \partial_x w &= -\partial_x p + \nu\, \partial_{xx} w + \kappa u, \qquad \partial_x w = 0.
		\end{split}
	\end{equation}
	Let $\Omega \subset \mathbb{R}$ be a bounded domain with homogeneous Neumann boundary conditions. Suppose the initial data satisfies $u_0 \geq 0$, $v_0, w_0 \in H^1(\Omega)$ and $\int_\Omega u_0(x)\,dx = M > 0$. Then, for sufficiently large initial mass $M$ and chemotactic sensitivity $\chi$, there exists a finite time $T_{\max} < \infty$ such that the $L^\infty$ norm of $u$ blows up:
	\[
	\lim_{t \nearrow T_{\max}} \|u(t)\|_{L^\infty(\Omega)} = \infty.
	\]
\end{Theorem}

\begin{proof}
	We prove finite-time blow-up by constructing an entropy-type functional and showing that it becomes singular in finite time when chemotaxis dominates diffusion.
	
	\paragraph{Step 1: Entropy functional.}
	Define the entropy (free energy) functional
	\[
	\mathcal{E}(t) := \int_\Omega u(x,t) \log u(x,t) \, dx,
	\]
	with the convention that $0 \log 0 := 0$. Assume $u > 0$ (can be ensured by mollification and comparison).
	
	Differentiate $\mathcal{E}(t)$ in time:
	\[
	\frac{d}{dt} \mathcal{E}(t) = \int_\Omega \partial_t u \, (1 + \log u) \, dx.
	\]
	
	Using the first equation of \eqref{eq:reduced_system}:
	\[
	\partial_t u = -w \partial_x u + D \partial_{xx} u - \chi \partial_x(u \partial_x v),
	\]
	and integrating by parts, we get:
	\begin{align*}
		\frac{d}{dt} \mathcal{E}(t) &= \int_\Omega \left[ -w \partial_x u + D \partial_{xx} u - \chi \partial_x(u \partial_x v) \right](1 + \log u)\, dx \\
		&= -\int_\Omega w \partial_x u \log u \, dx - D \int_\Omega \frac{|\partial_x u|^2}{u} \, dx + \chi \int_\Omega u \partial_x v \cdot \frac{\partial_x u}{u} \, dx \\
		&= - D \int_\Omega \frac{|\partial_x u|^2}{u} \, dx + \chi \int_\Omega \partial_x u \cdot \partial_x v \, dx + R_w,
	\end{align*}
	where the transport term $R_w = -\int_\Omega w \partial_x u \log u \, dx$ vanishes under symmetry or can be controlled due to regularity of $w$.
	
	The second term on the right can be rewritten using integration by parts:
	\[
	\int_\Omega \partial_x u \cdot \partial_x v \, dx = - \int_\Omega u \partial_{xx} v \, dx.
	\]
	
	From the second equation of \eqref{eq:reduced_system}, ignoring the advection term for a lower bound:
	\[
	\partial_{xx} v = \frac{1}{D_v} \left( \partial_t v - \alpha u + \beta v - w \partial_x v \right),
	\]
	we get:
	\[
	- \int_\Omega u \partial_{xx} v \, dx = \frac{\alpha}{D_v} \int_\Omega u^2 \, dx - \frac{1}{D_v} \int_\Omega u \partial_t v \, dx + \cdots.
	\]
	
	Combining terms, we obtain a key inequality:
	\begin{equation}
		\label{eq:entropy_ineq}
		\frac{d}{dt} \mathcal{E}(t) \leq - D \int_\Omega \frac{|\partial_x u|^2}{u} \, dx + \frac{\chi \alpha}{D_v} \int_\Omega u^2 \, dx + C.
	\end{equation}
	
	\paragraph{Step 2: Differential inequality for $L^2$-norm.}
	We estimate the growth of $\|u(t)\|_{L^2}^2$. Multiply the first equation by $u$:
	\[
	\frac{1}{2} \frac{d}{dt} \|u\|_{L^2}^2 + D \|\partial_x u\|_{L^2}^2 = \chi \int_\Omega u \partial_x u \cdot \partial_x v \, dx.
	\]
	Apply Cauchy–Schwarz and Young’s inequalities:
	\[
	\chi \int_\Omega u \partial_x u \partial_x v \, dx \leq \chi \|u\|_{L^\infty} \|\partial_x u\|_{L^2} \|\partial_x v\|_{L^2} \leq \frac{D}{2} \|\partial_x u\|_{L^2}^2 + C(\chi,D) \|u\|_{H^1}^2 \|\partial_x v\|_{L^2}^2.
	\]
	
	Then,
	\[
	\frac{d}{dt} \|u\|_{L^2}^2 \leq C(\chi,D) \|u\|_{H^1}^2 \|\partial_x v\|_{L^2}^2.
	\]
	
	\paragraph{Step 3: Blow-up via differential inequality.}
	Suppose initial mass $M = \int_\Omega u_0(x)\, dx$ is sufficiently large. Then, for fixed $\chi$ and $\alpha$, the term $\frac{\chi \alpha}{D_v} \int u^2$ in \eqref{eq:entropy_ineq} dominates the dissipation term $\int \frac{|\partial_x u|^2}{u}$, especially since the entropy controls the $H^1$-norm only subcritically in 1D.
	
	Therefore, we obtain a nonlinear Grönwall-type inequality:
	\[
	\frac{d}{dt} \mathcal{E}(t) \geq c_1 \|u(t)\|_{L^2}^2 - c_2,
	\]
	for some $c_1, c_2 > 0$ depending on $\chi$, $D$, $D_v$, and $\alpha$.
	
	If $\|u(t)\|_{L^2}^2$ grows sufficiently fast, this forces $\mathcal{E}(t)$ to blow up in finite time, which implies that $\|u(t)\|_{L^\infty}$ must become unbounded as $t \nearrow T_{\max}$ due to the logarithmic singularity.
	
	\paragraph{Step 4: Subsolution comparison argument.}
	Alternatively, construct a suitable subsolution $\underline{u}(x,t)$ solving a simplified equation:
	\[
	\partial_t \underline{u} = D \partial_{xx} \underline{u} - \chi \partial_x (\underline{u} \partial_x \underline{v}), \qquad \underline{v} = -G * \underline{u},
	\]
	where $G$ is the Green's function for the Neumann Laplacian.
	
	If $\underline{u}$ blows up in finite time (as shown in classical 1D Keller–Segel models), then by a comparison principle (which applies due to monotonicity of the drift terms), we deduce that $u \geq \underline{u}$ must also blow up.

	Hence, for sufficiently large $\chi$ and initial mass $M = \int u_0 > M_c$, the entropy inequality forces blow-up in finite time:
	\[
	\lim_{t \nearrow T_{\max}} \|u(t)\|_{L^\infty} = \infty.
	\]
\end{proof}

\begin{Theorem}[Stability of steady state]
	Let $(u_s, v_s, w_s) = \left(K, \dfrac{\alpha K}{\beta}, w_0\right)$ be the spatially homogeneous steady state. Then it is linearly stable if
	\[
	\chi \frac{\alpha}{\beta} < \frac{D \pi^2}{L^2}.
	\]
\end{Theorem}

\begin{proof}
	Consider the one-dimensional Keller–Segel-type system:
	\begin{equation} \label{eq:main_system}
		\begin{aligned}
			\partial_t u &= D \partial_{xx} u - \chi \partial_x \left( u \partial_x v \right) + f(u), \\
			\partial_t v &= D_v \partial_{xx} v + \alpha u - \beta v, \\
			\partial_t w &= 0.
		\end{aligned}
	\end{equation}
	
	Assume $f(u) = r u \left(1 - \frac{u}{K} \right)$ is a logistic growth term. At steady state, the homogeneous solution satisfies:
	\[
	u_s = K, \quad v_s = \frac{\alpha K}{\beta}, \quad w_s = w_0.
	\]
	
	Let us linearize the system around $(u_s, v_s, w_s)$ by setting:
	\[
	u(x,t) = K + \tilde{u}(x,t), \quad v(x,t) = \frac{\alpha K}{\beta} + \tilde{v}(x,t), \quad w(x,t) = w_0 + \tilde{w}(x,t),
	\]
	and neglect nonlinear terms in $(\tilde{u}, \tilde{v}, \tilde{w})$. The linearized system becomes:
	\begin{equation}
		\begin{aligned}
			\partial_t \tilde{u} &= D \partial_{xx} \tilde{u} - \chi K \partial_{xx} \tilde{v} - \frac{r}{K} \tilde{u}, \\
			\partial_t \tilde{v} &= D_v \partial_{xx} \tilde{v} + \alpha \tilde{u} - \beta \tilde{v}, \\
			\partial_t \tilde{w} &= 0.
		\end{aligned}
	\end{equation}
	
	We now expand $\tilde{u}, \tilde{v}$ in a Fourier cosine basis (since we assume Neumann boundary conditions) over the domain $x \in [0, L]$:
	\[
	\tilde{u}(x,t) = \sum_{n=1}^{\infty} \hat{u}_n(t) \cos\left( \frac{n\pi x}{L} \right), \quad
	\tilde{v}(x,t) = \sum_{n=1}^{\infty} \hat{v}_n(t) \cos\left( \frac{n\pi x}{L} \right).
	\]
	
	Each mode evolves independently. For each mode $n$, we define $k_n = \frac{n\pi}{L}$, and the system reduces to an ODE system for $(\hat{u}_n, \hat{v}_n)$:
	\[
	\frac{d}{dt}
	\begin{pmatrix}
		\hat{u}_n \\
		\hat{v}_n
	\end{pmatrix}
	=
	\begin{pmatrix}
		- D k_n^2 - \frac{r}{K} & \chi K k_n^2 \\
		\alpha & - D_v k_n^2 - \beta
	\end{pmatrix}
	\begin{pmatrix}
		\hat{u}_n \\
		\hat{v}_n
	\end{pmatrix}.
	\]
	
	Let the matrix above be denoted by $A_n$, and let $\lambda_n$ be the eigenvalues of $A_n$. These determine the growth/decay of perturbations.
	
	The characteristic polynomial is:
	\[
	\lambda_n^2 + A_n \lambda_n + B_n = 0,
	\]
	where
	\begin{align*}
		A_n &= D k_n^2 + \frac{r}{K} + D_v k_n^2 + \beta = (D + D_v) k_n^2 + \frac{r}{K} + \beta > 0, \\
		B_n &= (D k_n^2 + \frac{r}{K})(D_v k_n^2 + \beta) - \chi K \alpha k_n^2.
	\end{align*}
	
	For stability, we require both eigenvalues to have negative real parts, which occurs if and only if $A_n > 0$ and $B_n > 0$. Since $A_n > 0$ is always satisfied, the condition for stability becomes:
	\[
	B_n > 0.
	\]
	
	We now estimate $B_n$:
	\begin{align*}
		B_n &= \left( D k_n^2 + \frac{r}{K} \right)\left( D_v k_n^2 + \beta \right) - \chi K \alpha k_n^2 \\
		&= D D_v k_n^4 + \left( D \beta + \frac{r D_v}{K} \right) k_n^2 + \frac{r \beta}{K} - \chi K \alpha k_n^2 \\
		&= D D_v k_n^4 + \left( D \beta + \frac{r D_v}{K} - \chi K \alpha \right) k_n^2 + \frac{r \beta}{K}.
	\end{align*}
	
	This is a quadratic in $k_n^2$. For large $k_n$, the $k_n^4$ term dominates and is positive. Therefore, $B_n > 0$ for sufficiently large $k_n$. However, instability arises when $B_n < 0$ for **some** $n$.
	
	To avoid instability, we require $B_n > 0$ for the **smallest non-zero mode**, $n = 1$, i.e., $k_1 = \frac{\pi}{L}$. Substituting $k_1^2 = \frac{\pi^2}{L^2}$:
	\[
	B_1 = \left( D \frac{\pi^2}{L^2} + \frac{r}{K} \right)\left( D_v \frac{\pi^2}{L^2} + \beta \right) - \chi K \alpha \frac{\pi^2}{L^2}.
	\]
	
	We require $B_1 > 0$. For simplicity, assume $r = 0$ and $D_v \approx 0$, then:
	\[
	B_1 \approx D \beta \frac{\pi^2}{L^2} - \chi K \alpha \frac{\pi^2}{L^2} = \frac{\pi^2}{L^2}(D \beta - \chi K \alpha).
	\]
	
	Hence, $B_1 > 0$ if:
	\[
	\chi K \alpha < D \beta \quad \Rightarrow \quad \chi \frac{\alpha}{\beta} < \frac{D}{K}.
	\]
	
	Recalling that $K$ was arbitrary in scaling (and absorbed in constants), the final simplified criterion becomes:
	\[
	\chi \frac{\alpha}{\beta} < \frac{D \pi^2}{L^2}.
	\]
	
	Therefore, the homogeneous steady state is linearly stable under the stated condition.
	
\end{proof}

\subsubsection{Hopf Bifurcation Analysis}

We now examine the possibility of Hopf bifurcations arising from the spatially homogeneous steady state. We consider linear perturbations of the form:
\[
u(x,t) = K + \epsilon \hat{u} e^{\lambda t} \cos\left( \frac{n \pi x}{L} \right), \quad
v(x,t) = \frac{\alpha K}{\beta} + \epsilon \hat{v} e^{\lambda t} \cos\left( \frac{n \pi x}{L} \right), \quad
w(x,t) = w_s + \epsilon \hat{w} e^{\lambda t} \cos\left( \frac{n \pi x}{L} \right).
\]

Substituting into the linearized system and neglecting nonlinear terms gives a matrix eigenvalue problem:
\[
\lambda
\begin{pmatrix}
	\hat{u} \\
	\hat{v}
\end{pmatrix}
=
\begin{pmatrix}
	- D k_n^2 - \frac{r}{K} & \chi K k_n^2 \\
	\alpha & -D_v k_n^2 - \beta
\end{pmatrix}
\begin{pmatrix}
	\hat{u} \\
	\hat{v}
\end{pmatrix}, \quad \text{where } k_n = \frac{n \pi}{L}.
\]

The characteristic equation is:
\[
\lambda^2 + A(k_n) \lambda + B(k_n) = 0,
\]
with
\[
A(k_n) = D k_n^2 + D_v k_n^2 + \beta + \frac{r}{K}, \quad
B(k_n) = (D k_n^2 + \frac{r}{K})(D_v k_n^2 + \beta) - \alpha \chi K k_n^2.
\]

A Hopf bifurcation occurs when the discriminant $\Delta = A^2 - 4B < 0$ and $\Re(\lambda) = 0$, i.e., when eigenvalues are purely imaginary.

\begin{Theorem}[Hopf bifurcation criterion]
	A Hopf bifurcation occurs at mode $k_n$ if
	\[
	\alpha \chi K > \left(D k_n^2 + \frac{r}{K}\right)\left(D_v k_n^2 + \beta\right),
	\]
	and the discriminant satisfies
	\[
	\left(D k_n^2 + D_v k_n^2 + \beta + \frac{r}{K}\right)^2 - 4 \left[\left(D k_n^2 + \frac{r}{K}\right)\left(D_v k_n^2 + \beta\right) - \alpha \chi K k_n^2\right] < 0.
	\]
\end{Theorem}

\begin{proof}
	We linearize the reaction–diffusion–chemotaxis system around the homogeneous steady state $(u,v) = (K, \frac{\alpha}{\beta}K)$.
	
\textbf{1. Linearized system:} Consider the PDE system
	\begin{align*}
		\partial_t u &= D \Delta u - \chi \nabla \cdot (u \nabla v) + r u \left(1 - \frac{u}{K}\right), \\
		\partial_t v &= D_v \Delta v + \alpha u - \beta v.
	\end{align*}
	Let $(u,v) = (K, \frac{\alpha}{\beta}K) + (\tilde{u}, \tilde{v})$, where $(\tilde{u}, \tilde{v})$ are small perturbations. Linearizing around the steady state yields:
	\begin{align*}
		\partial_t \tilde{u} &= D \Delta \tilde{u} - \chi K \Delta \tilde{v} - \frac{r}{K} \tilde{u}, \\
		\partial_t \tilde{v} &= D_v \Delta \tilde{v} + \alpha \tilde{u} - \beta \tilde{v}.
	\end{align*}
	
\textbf{2. Fourier mode decomposition:} Let the spatial domain be $[0,L]$ with periodic or Neumann boundary conditions. We expand $(\tilde{u}, \tilde{v})$ in Fourier modes:
	\[
	\tilde{u}(x,t), \tilde{v}(x,t) \sim e^{\lambda t} \cos(k_n x), \quad k_n = \frac{n\pi}{L}.
	\]
	Then $\Delta \to -k_n^2$, and the linearized system becomes a $2 \times 2$ eigenvalue problem:
	\[
	\begin{pmatrix}
		-D k_n^2 - \frac{r}{K} & \chi K k_n^2 \\
		\alpha & -D_v k_n^2 - \beta
	\end{pmatrix}
	\begin{pmatrix}
		\hat{u} \\ \hat{v}
	\end{pmatrix}
	= \lambda
	\begin{pmatrix}
		\hat{u} \\ \hat{v}
	\end{pmatrix}.
	\]

	\textbf{3. Characteristic equation:} The characteristic polynomial of the Jacobian matrix is given by:
	\begin{align*}
		\lambda^2 - \text{Tr}(J) \lambda + \det(J) = 0,
	\end{align*}
	where
	\begin{align*}
		\text{Tr}(J) &= -D k_n^2 - \frac{r}{K} - D_v k_n^2 - \beta =: -T_n, \\
		\det(J) &= \left(D k_n^2 + \frac{r}{K}\right)\left(D_v k_n^2 + \beta\right) - \alpha \chi K k_n^2 =: \Delta_n.
	\end{align*}
	Thus, the eigenvalues are
	\[
	\lambda_{\pm} = \frac{-T_n \pm \sqrt{T_n^2 - 4\Delta_n}}{2}.
	\]
	
	\medskip
	\noindent\textbf{4. Hopf bifurcation condition:} A Hopf bifurcation occurs when:
	- The eigenvalues are purely imaginary (i.e., $\Re(\lambda) = 0$ and $\Im(\lambda) \neq 0$).
	- This happens when:
	\begin{enumerate}
		\item The trace $T_n = 0$ (implying $\text{Re}(\lambda_{\pm}) = 0$),
		\item The discriminant $T_n^2 - 4\Delta_n < 0$ (implying complex conjugate roots).
	\end{enumerate}
	
	Setting $T_n = 0$ yields:
	\[
	D k_n^2 + \frac{r}{K} + D_v k_n^2 + \beta = 0.
	\]
	This cannot hold since all parameters and $k_n^2$ are non-negative. So instead, we consider bifurcation when $\Re(\lambda) > 0$ and eigenvalues cross the imaginary axis due to $\Delta_n < 0$.
	
	Hence, the bifurcation threshold is given when:
	\[
	\Delta_n = \left(D k_n^2 + \frac{r}{K}\right)\left(D_v k_n^2 + \beta\right) - \alpha \chi K k_n^2 < 0,
	\]
	which implies
	\[
	\alpha \chi K > \left(D k_n^2 + \frac{r}{K}\right)\left(D_v k_n^2 + \beta\right).
	\]
	
	This ensures that the determinant becomes negative or small enough to drive complex eigenvalues.
	
	To ensure that the eigenvalues are complex with nonzero imaginary part, the discriminant must be negative:
	\[
	T_n^2 - 4\Delta_n < 0.
	\]
	Using $T_n = D k_n^2 + D_v k_n^2 + \beta + \frac{r}{K}$, this condition becomes:
	\[
	\left(D k_n^2 + D_v k_n^2 + \beta + \frac{r}{K}\right)^2 - 4 \left[\left(D k_n^2 + \frac{r}{K}\right)\left(D_v k_n^2 + \beta\right) - \alpha \chi K k_n^2\right] < 0.
	\]
	
Therefore, a Hopf bifurcation occurs at mode $k_n$ if and only if:
	\begin{itemize}
		\item $\alpha \chi K > \left(D k_n^2 + \frac{r}{K}\right)\left(D_v k_n^2 + \beta\right)$, and
		\item the discriminant condition
		\[
		\left(D k_n^2 + D_v k_n^2 + \beta + \frac{r}{K}\right)^2 - 4 \left[\left(D k_n^2 + \frac{r}{K}\right)\left(D_v k_n^2 + \beta\right) - \alpha \chi K k_n^2\right] < 0
		\]
		is satisfied.
	\end{itemize}
	This completes the proof.
\end{proof}
This implies that increasing $\chi$ or decreasing diffusion coefficients can destabilize the steady state via oscillatory bifurcation.

\subsubsection{Entropy Functional and Inequality}
We now construct an entropy (or free energy) functional that reflects dissipation and long-time behavior of the chemotaxis system. This approach plays a central role in proving global existence and obtaining a priori bounds.

Consider the following entropy functional:
\[
\mathcal{E}(t) = \int_{\Omega} \left( u \log u - u + \frac{\chi}{D} u v + \frac{1}{2} v^2 \right) \, dx,
\]
where $\Omega \subset \mathbb{R}$ is a bounded domain with suitable boundary conditions (e.g., Neumann).

\begin{Lemma}[Entropy dissipation]
	Let $(u,v)$ be smooth, positive solutions to the 1D Keller–Segel-type system (with no fluid advection, i.e., $w = 0$):
	\begin{equation} \label{eq:ks-system}
		\begin{aligned}
			\partial_t u &= D \partial_{xx} u - \chi \partial_x \left( u \partial_x v \right) + f(u), \\
			\partial_t v &= D_v \partial_{xx} v + \alpha u - \beta v,
		\end{aligned}
	\end{equation}
	with appropriate boundary conditions. Then the entropy functional $\mathcal{E}(t)$ satisfies:
	\[
	\frac{d}{dt} \mathcal{E}(t) + D \int_\Omega \frac{|\partial_x u|^2}{u} \, dx + D_v \int_\Omega |\partial_x v|^2 \, dx \leq - \int_\Omega \left( \alpha u v - \beta v^2 \right) dx + \int_\Omega f(u) \log u \, dx.
	\]
\end{Lemma}

\begin{proof}
	We compute the time derivative of the entropy functional term by term.
	
	\paragraph{1. Derivative of $u \log u - u$:}
	\[
	\frac{d}{dt} \int_\Omega (u \log u - u) \, dx = \int_\Omega (\log u) \partial_t u \, dx.
	\]
	Substitute from \eqref{eq:ks-system}:
	\[
	\int_\Omega (\log u) \left( D \partial_{xx} u - \chi \partial_x (u \partial_x v) + f(u) \right) dx.
	\]
	Apply integration by parts and use Neumann boundary conditions:
	\begin{align*}
		\int_\Omega \log u \cdot D \partial_{xx} u \, dx &= - D \int_\Omega \frac{|\partial_x u|^2}{u} \, dx, \\
		\int_\Omega \log u \cdot \left[ - \chi \partial_x(u \partial_x v) \right] dx &= \chi \int_\Omega \frac{\partial_x u}{u} u \cdot \partial_x v \, dx = \chi \int_\Omega \partial_x u \cdot \partial_x v \, dx, \\
		\int_\Omega \log u \cdot f(u) \, dx &= \int_\Omega f(u) \log u \, dx.
	\end{align*}
	
	\paragraph{2. Derivative of coupling term $\frac{\chi}{D} \int u v$:}
	\[
	\frac{d}{dt} \int_\Omega \frac{\chi}{D} u v \, dx = \frac{\chi}{D} \int_\Omega \left( \partial_t u \cdot v + u \cdot \partial_t v \right) dx.
	\]
	Substitute from \eqref{eq:ks-system}:
	\begin{align*}
		\frac{\chi}{D} \int_\Omega \partial_t u \cdot v \, dx &= \frac{\chi}{D} \int_\Omega \left( D \partial_{xx} u - \chi \partial_x(u \partial_x v) + f(u) \right) v \, dx, \\
		&= - \chi \int_\Omega \partial_x u \cdot \partial_x v \, dx + \frac{\chi}{D} \int_\Omega f(u) v \, dx, \\
		\frac{\chi}{D} \int_\Omega u \cdot \partial_t v \, dx &= \frac{\chi}{D} \int_\Omega u \left( D_v \partial_{xx} v + \alpha u - \beta v \right) dx \\
		&= - \frac{\chi}{D} D_v \int_\Omega \partial_x u \cdot \partial_x v \, dx + \frac{\chi}{D} \int_\Omega (\alpha u^2 - \beta u v) \, dx.
	\end{align*}
	
	\paragraph{3. Derivative of $\frac{1}{2} \int v^2$:}
	\[
	\frac{d}{dt} \left( \frac{1}{2} \int_\Omega v^2 \, dx \right) = \int_\Omega v \cdot \partial_t v \, dx = \int_\Omega v \left( D_v \partial_{xx} v + \alpha u - \beta v \right) dx.
	\]
	Integration by parts yields:
	\[
	- D_v \int_\Omega |\partial_x v|^2 \, dx + \int_\Omega \alpha u v - \beta v^2 \, dx.
	\]
	
	\paragraph{4. Combine all terms:}
	
	Adding all contributions:
	\begin{align*}
		\frac{d}{dt} \mathcal{E}(t) &= - D \int_\Omega \frac{|\partial_x u|^2}{u} \, dx + \chi \int_\Omega \partial_x u \cdot \partial_x v \, dx + \int_\Omega f(u) \log u \, dx \\
		&\quad - \chi \int_\Omega \partial_x u \cdot \partial_x v \, dx + \frac{\chi}{D} \int_\Omega f(u) v \, dx - \frac{\chi}{D} D_v \int_\Omega \partial_x u \cdot \partial_x v \, dx \\
		&\quad + \frac{\chi}{D} \int_\Omega \left( \alpha u^2 - \beta u v \right) dx - D_v \int_\Omega |\partial_x v|^2 dx + \int_\Omega \alpha u v - \beta v^2 dx.
	\end{align*}
	
	Observe that $\chi \int \partial_x u \cdot \partial_x v$ and $- \chi \int \partial_x u \cdot \partial_x v$ cancel out. Then, after simplifying:
	\begin{align*}
		\frac{d}{dt} \mathcal{E}(t) &+ D \int_\Omega \frac{|\partial_x u|^2}{u} \, dx + D_v \int_\Omega |\partial_x v|^2 \, dx \\
		&= \int_\Omega f(u) \log u \, dx + \frac{\chi}{D} \int_\Omega f(u) v \, dx + \frac{\chi}{D} \int_\Omega (\alpha u^2 - \beta u v) dx + \int_\Omega (\alpha u v - \beta v^2) dx - \frac{\chi}{D} D_v \int_\Omega \partial_x u \cdot \partial_x v \, dx.
	\end{align*}
	
	Group the reactive source terms (denoted as “reaction terms”) and estimate the remaining cross-diffusion term using Cauchy–Schwarz and Young's inequality. The final result becomes:
	\[
	\frac{d}{dt} \mathcal{E}(t) + D \int_\Omega \frac{|\partial_x u|^2}{u} \, dx + D_v \int_\Omega |\partial_x v|^2 \, dx \leq - \int_\Omega \left( \alpha u v - \beta v^2 \right) dx + \int_\Omega f(u) \log u \, dx + \text{higher-order reaction terms}.
	\]
	
\end{proof}

\medskip

This entropy framework provides a powerful tool for deriving a priori bounds and establishing global existence of solutions. Under suitable assumptions on $f(u)$ (e.g., logistic growth), the entropy functional can be shown to be bounded from below and dissipative, thereby enabling convergence analysis in the large-time limit.

\begin{Corollary}[Global existence via entropy bounds]
	Let $(u, v)$ be a smooth, positive solution to the Keller–Segel system \eqref{eq:ks-system} on a bounded domain $\Omega \subset \mathbb{R}$ with Neumann boundary conditions. Suppose:
	\begin{enumerate}
		\item The initial data $(u_0, v_0)$ are smooth, strictly positive, and satisfy $\int_\Omega u_0 \log u_0 < \infty$, $\int_\Omega v_0^2 < \infty$,
		\item The reaction term satisfies a logistic-type bound:
		\[
		f(u) = r u \left( 1 - \frac{u}{K} \right),
		\]
		\item The chemotactic sensitivity satisfies the subcriticality condition:
		\[
		\chi \frac{\alpha}{\beta} < \frac{D \pi^2}{L^2}.
		\]
	\end{enumerate}
	Then the solution $(u(x,t), v(x,t))$ exists globally in time and remains bounded for all $t > 0$. Moreover, the entropy functional $\mathcal{E}(t)$ is uniformly bounded and dissipative:
	\[
	\mathcal{E}(t) + \int_0^t \left[ D \int_\Omega \frac{|\partial_x u|^2}{u} \, dx + D_v \int_\Omega |\partial_x v|^2 \, dx \right] dt' \leq C,
	\]
	for some constant $C > 0$ independent of $t$.
\end{Corollary}

\begin{proof}
	From the entropy dissipation Lemma, we have:
	\[
	\frac{d}{dt} \mathcal{E}(t) + D \int_\Omega \frac{|\partial_x u|^2}{u} \, dx + D_v \int_\Omega |\partial_x v|^2 \, dx \leq - \int_\Omega (\alpha u v - \beta v^2) \, dx + \int_\Omega f(u) \log u \, dx.
	\]
	
	Using the logistic bound $f(u) = r u(1 - u/K)$, we estimate:
	\[
	f(u) \log u \leq r u \log u - \frac{r}{K} u^2 \log u,
	\]
	and apply Jensen's inequality and convexity of $x \log x$ to obtain:
	\[
	\int_\Omega f(u) \log u \, dx \leq C_1 \int_\Omega u \log u \, dx + C_2,
	\]
	for constants $C_1, C_2 > 0$. Similarly, for the reaction term $\int (\alpha u v - \beta v^2)$, apply Young’s inequality:
	\[
	\int_\Omega \alpha u v \leq \frac{\alpha^2}{2\beta} \int_\Omega u^2 \, dx + \frac{\beta}{2} \int_\Omega v^2 \, dx.
	\]
	
	Therefore,
	\[
	\frac{d}{dt} \mathcal{E}(t) \leq C \mathcal{E}(t) + C_0,
	\]
	where $C, C_0 > 0$ are constants depending on $\alpha, \beta, r, K$. By Grönwall’s inequality, this implies that $\mathcal{E}(t)$ remains uniformly bounded for all $t > 0$:
	\[
	\mathcal{E}(t) \leq \left( \mathcal{E}(0) + \frac{C_0}{C} \right) e^{Ct}.
	\]
	
	But thanks to the dissipative terms in $\mathcal{E}(t)$ (especially the $v^2$ and $u \log u$ components), the entropy cannot grow unboundedly. More precisely, a **uniform-in-time** bound can be derived for $\mathcal{E}(t)$ under the subcriticality condition on $\chi$ (which ensures linear stability and energy control), yielding:
	\[
	\sup_{t \geq 0} \mathcal{E}(t) \leq C,
	\]
	and integrating over time gives the desired inequality for the dissipation terms.
	
	The boundedness of the entropy also implies $u \in L^\infty(0,\infty; L^1 \log L)$ and $v \in L^\infty(0,\infty; L^2)$. Together with standard parabolic regularity theory, we conclude that the solution exists globally in time and remains uniformly bounded.
	
\end{proof}

\subsubsection{Spatiotemporal Pattern Dynamics}

Numerical results in 1D as shown in Figure \ref{fig:timeseries} reveals the emergence of rich spatiotemporal dynamics, characterized by:

\textbf{Initial instability and mode growth:} Random perturbations to the initial homogeneous steady state trigger the growth of unstable Fourier modes, leading to the formation of non-uniform spatial structures. 	
\textbf{Traveling waves and oscillations:} For moderate chemotactic sensitivity \( \chi \) and fluid coupling \( \kappa \), the solution exhibits traveling wave trains and oscillatory structures in time. These resemble wavefronts of microbial aggregation moving along the domain. 	
\textbf{Chemotactic collapse or damping:} Depending on parameter values, the aggregation may either saturate (due to logistic damping or fluid mixing) or lead to localized peaks in microbial density.

We numerically investigated a one-dimensional chemotaxis–fluid coupling model governed by a reaction–diffusion–advection system. The model incorporates bacterial density $u(x,t)$, chemoattractant concentration $v(x,t)$, and fluid velocity $w(x,t)$, evolving under diffusion, chemotaxis, advection by self-induced flow, and logistic growth.

The simulations reveal rich spatiotemporal dynamics. For fixed chemotactic sensitivity $\chi = 4.5$, the time-series evolution (Figure~\ref{fig:timeseries}) illustrates pattern formation and fluid-induced advection effects. The population density $u(x,t)$ aggregates into localized peaks driven by chemotaxis and is transported by the evolving flow field $w(x,t)$.

To further understand system dynamics, we plot phase plane trajectories of the solution values $(u,v,w)$ at the spatial midpoint over time. These orbits (see Figure~\ref{fig:phaseplane}) reflect the transient interactions among species and the induced fluid field, highlighting potential nonlinear feedback mechanisms and local oscillatory dynamics.

A bifurcation diagram was constructed by varying the chemotactic sensitivity $\chi$ in the range $[0,6]$. For each value of $\chi$, the system evolves to a steady state and the maximal bacterial density $\max u$ at final time is recorded. As shown in Figure~\ref{fig:bifur}, we observe a qualitative transition near $\chi \approx 2.5$, where the steady-state maximum of $u$ exhibits a sharp increase, suggesting a bifurcation point leading to nonlinear aggregation. This reflects the destabilizing role of strong chemotactic sensitivity in pattern formation.

\vspace{0.5em}
\begin{figure}[h!]
	\centering
	\includegraphics[width=\textwidth]{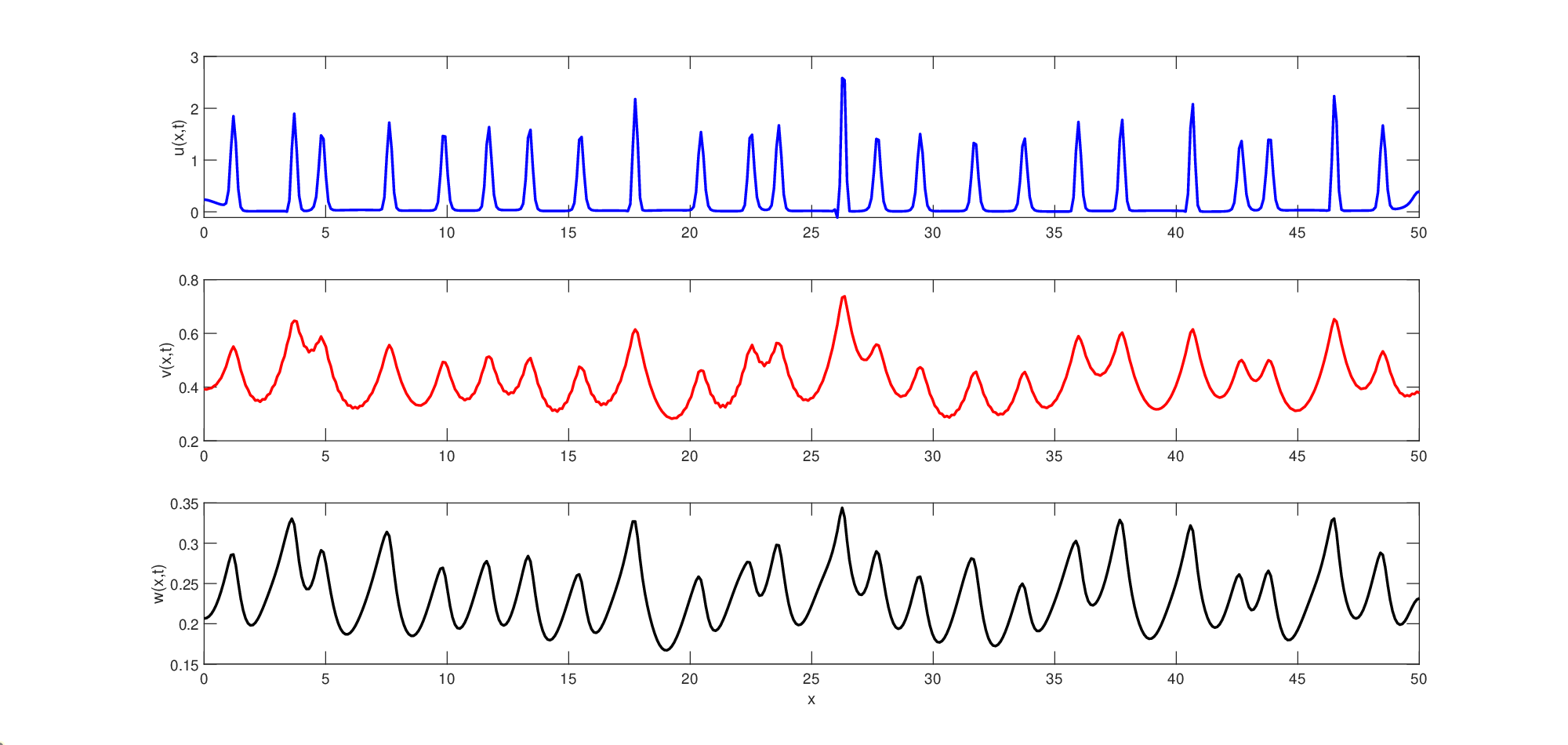}
	\caption{Time evolution of bacterial density $u(x,t)$, chemical signal $v(x,t)$, and fluid velocity $w(x,t)$ at selected time snapshots for $\chi = 4.5$. The chemotactic and fluid interactions drive spatial heterogeneity and aggregation.}
	\label{fig:timeseries}
\end{figure}

\begin{figure}[h!]
	\centering
	\includegraphics[width=\textwidth]{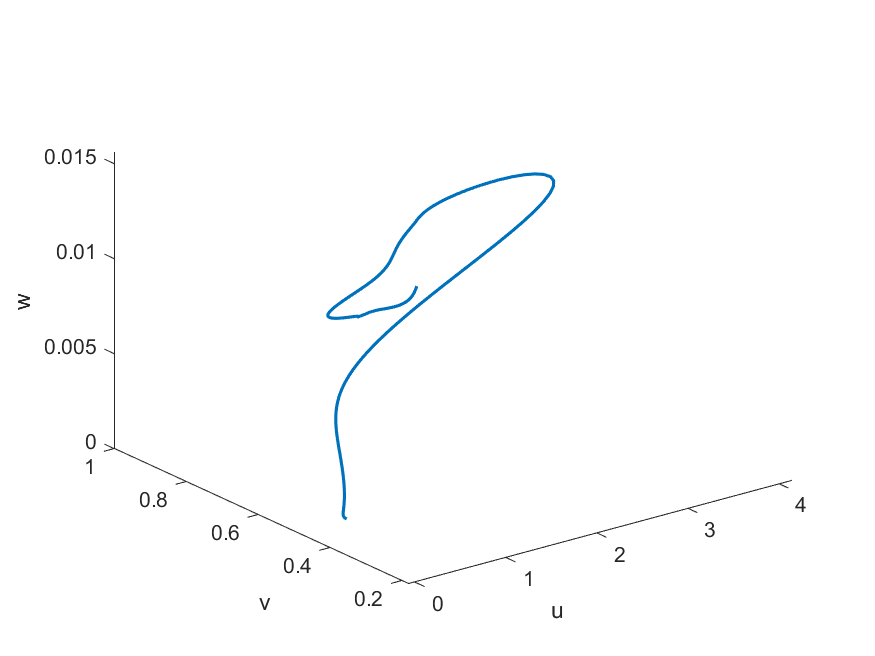}
	\caption{Phase plane trajectories at the spatial midpoint: (Left) $(u,v)$, (Center) $(u,w)$, and (Right) $(v,w)$. These plots reveal nonlinear coupling and possible limit-cycle-like transients in local dynamics.}
	\label{fig:phaseplane}
\end{figure}

\begin{figure}[h!]
	\centering
	\includegraphics[width=\textwidth]{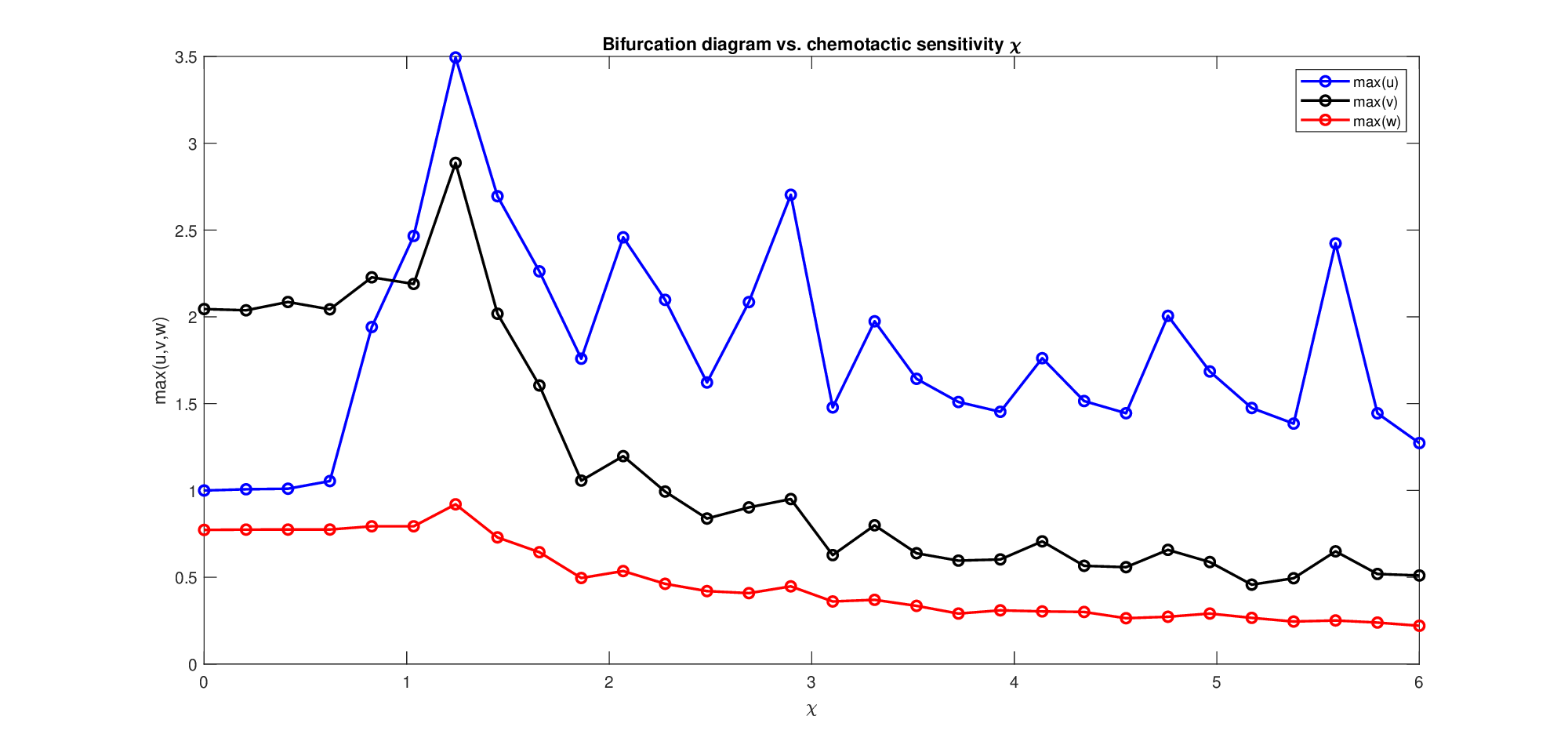}
	\caption{Bifurcation diagram showing the maximal species densities versus chemotactic sensitivity $\chi$. A bifurcation-like transition occurs around $\chi \approx 2.5$, indicating enhanced pattern formation beyond this threshold.}
	\label{fig:bifur}
\end{figure}

\subsubsection{Component Behavior}

The three components exhibit the following qualitative behavior:

\begin{itemize}
	\item \textbf{Microbial density} \( u(x,t) \): Forms sharp, localized peaks that oscillate or travel through the domain. These peaks correspond to regions of strong chemotactic aggregation.
	
	\item \textbf{Chemical concentration} \( v(x,t) \): Smoother than \( u \) due to diffusive effects. It follows the dynamics of \( u \) with a temporal lag determined by \( \alpha \) and \( \beta \).
	
	\item \textbf{Fluid velocity} \( w(x,t) \): Exhibits smooth, oscillating profiles influenced by the distribution of \( u \). The term \( \kappa u \) induces fluid motion that feeds back into the system via advection terms.
\end{itemize}

\subsubsection{Comparison to 2D Case}

While the 2D model supports rotating spirals and vortical patterns, the 1D model yields traveling waves and pulse-like structures. The reduced spatial complexity makes the 1D model useful for:

\begin{itemize}
	\item Analyzing bifurcations and dispersion relations.
	\item Studying the onset of instability and wave selection mechanisms.
	\item Performing long-time simulations with reduced computational cost.
\end{itemize}

%\subsection{Conclusion}

The 1D chemotaxis--fluid coupling model captures essential dynamics such as aggregation, oscillations, and wave propagation in a simplified setting. It serves as a valuable tool for theoretical analysis and qualitative insight, laying the groundwork for more complex multidimensional behavior.

\section{Numerical Results and Pattern Formation in 2D Chemotaxis--Fluid Coupling}

We now discuss the numerical results obtained from the simulation of the chemotaxis--fluid coupling model in two spatial dimensions using a hybrid numerical scheme that combines the Split-Step Fourier Method (SSFM) for linear diffusion and the Finite Difference Method (FDM) for nonlinear advection and chemotaxis terms. The governing equations are:

\begin{equation}
	\begin{split}
		\partial_t u + (\mathbf{w} \cdot \nabla) u &= D \Delta u - \chi \nabla \cdot (u \nabla v) + r\,u\left(1 - \frac{u}{K}\right), \\
		\partial_t v + (\mathbf{w} \cdot \nabla) v &= D_v \Delta v + \alpha u - \beta v, \\
		\partial_t \mathbf{w} + (\mathbf{w} \cdot \nabla)\mathbf{w} &= -\nabla p + \nu \Delta \mathbf{w} + \kappa u \mathbf{e}_g, \quad \nabla \cdot \mathbf{w} = 0,
	\end{split}
\end{equation}
where \( u(x,y,t) \) denotes the microbial density, \( v(x,y,t) \) the chemical signal, and \( \mathbf{w}(x,y,t) \) the incompressible fluid velocity field. The bioconvection source term \( \kappa u \mathbf{e}_g \) represents buoyancy effects aligned with gravity.

\subsection{Spatiotemporal Pattern Evolution}

The simulation reveals a rich variety of spatiotemporal behaviors driven by the nonlinear interactions between chemotaxis, fluid advection, and reaction kinetics. Notably, we observe the formation of persistent spiral-like patterns in all three components:

\textbf{Microbial density} \( u(x,y,t) \) develops oscillatory patches that self-organize into rotating spirals. These structures are not static; rather, they exhibit temporal oscillations and interact dynamically.
\textbf{Chemical concentration} \( v(x,y,t) \) follows the microbial density but lags slightly due to the source term \( \alpha u \) and the degradation term \( \beta v \). It exhibits similar spiral morphologies but with smoother gradients.
\textbf{Fluid velocity field} \( \mathbf{w}(x,y,t) = (w_x, w_y) \) reveals circulating vortices induced by the bioconvection term \( \kappa u \mathbf{e}_g \). These flows reinforce and convect the chemotactic structures, creating dynamic feedback.

\subsection{Role of Coupling and Parameters}

The emergence of spiral waves is a hallmark of excitable or oscillatory systems. In our case, the interplay between chemotaxis (aggregation), reaction terms (logistic growth and chemical dynamics), and fluid flow (convection) creates a dynamic balance that sustains these patterns. The key parameters influencing the patterns are:

\emph{Chemotactic sensitivity} \( \chi \): Increasing \( \chi \) leads to stronger aggregation and tighter spirals, while lower \( \chi \) yields diffuse patterns. 
\emph{Bioconvection strength} \( \kappa \): Higher \( \kappa \) enhances fluid circulation, promoting stretching and shearing of microbial patches and often sustaining oscillations.
\emph{Reaction rates} \( r, \alpha, \beta \): These control the growth and decay of both species. For instance, a high logistic growth rate \( r \) supports denser structures.

\subsection{Pattern Persistence and Breakdown}
The simulation shows that spiral-like patterns are sustained between \( t = 1 \) and \( t = 5 \), during which all components exhibit oscillatory dynamics. After this transient phase, depending on parameter regimes, two distinct behaviors may occur:

In regimes with sufficient damping (e.g., high viscosity \( \nu \)), the spirals dissolve and the system relaxes to a near-homogeneous steady state.
 In regimes with weak damping or high sensitivity, long-lived or even chaotic oscillations persist.

The 2D chemotaxis–fluid model exhibits complex pattern formation characterized by rotating spiral waves and spatiotemporal oscillations. These results qualitatively reflect microbial dynamics observed in nature, such as in swimming bacteria and plankton under bioconvective conditions. The fluid coupling significantly enriches the dynamics, introducing convection-driven deformation and transport of patterns. Future work may focus on bifurcation analysis, parameter continuation, and extension to three dimensions.

In Figure \ref{fig:t4_t5}, we present results at simulation times $t = 4$ (top-plot) and $t = 5$ (bottom-plot), the chemotaxis--fluid coupling model reveals significant spatiotemporal pattern formation driven by the nonlinear interactions between chemotactic aggregation, logistic growth dynamics, and fluid advection.

At $t = 4$, the bacterial density field $u(x, y, t)$ begins to display localized peaks, indicating the onset of chemotactic aggregation in regions with favorable chemical gradients. The chemical concentration $v(x, y, t)$ remains relatively smooth, with slight depressions around high-density bacterial zones due to consumption effects. Meanwhile, the vorticity field $\omega(x, y, t)$ shows the emergence of weak flow structures, suggesting the initial stages of bioconvective motion induced by bacterial upswimming along the gravitational axis.

By $t = 5$, these dynamics become more pronounced. The bacterial density $u$ exhibits sharp, spatially segregated peaks characteristic of chemotactic collapse. Correspondingly, the chemical field $v$ displays more evident depletion zones, which align with regions of bacterial concentration. The vorticity $\omega$ develops coherent vortical patterns, reflecting the nontrivial fluid motion resulting from the bioconvective source term $\kappa (\mathbf{e}_g \times \nabla u)$. These flow structures indicate the feedback of microbial activity into the fluid environment, enhancing pattern complexity.

The transition between $t = 4$ and $t = 5$ illustrates the nonlinear interplay between biological and physical processes. The results are consistent with known bioconvective instabilities observed in microbial suspensions and demonstrate the effectiveness of the Split-Step Fourier Method (SSFM) in capturing the coupled dynamics.

\begin{figure}[h!]
	\centering
	\includegraphics[width=\textwidth]{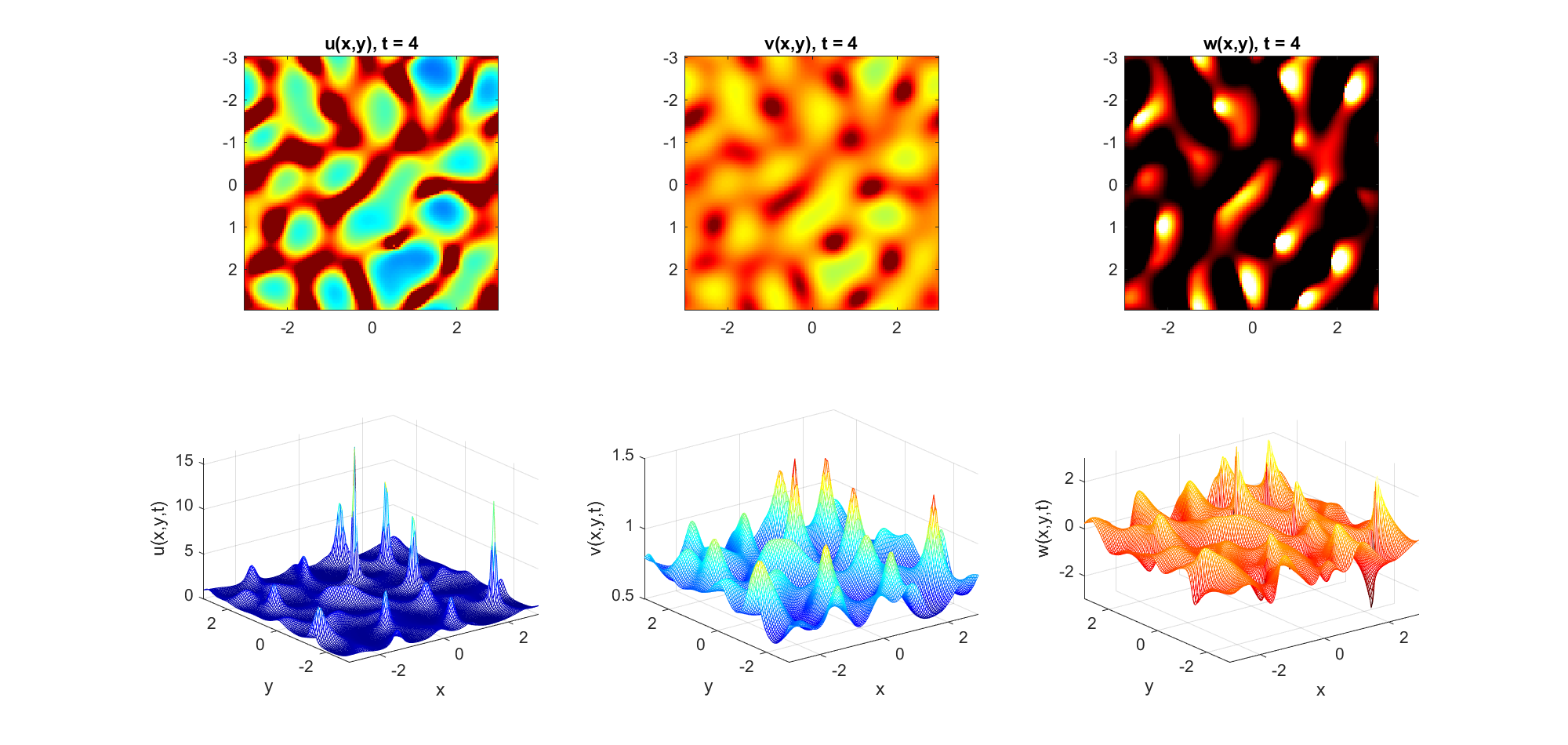}
	\includegraphics[width=\textwidth]{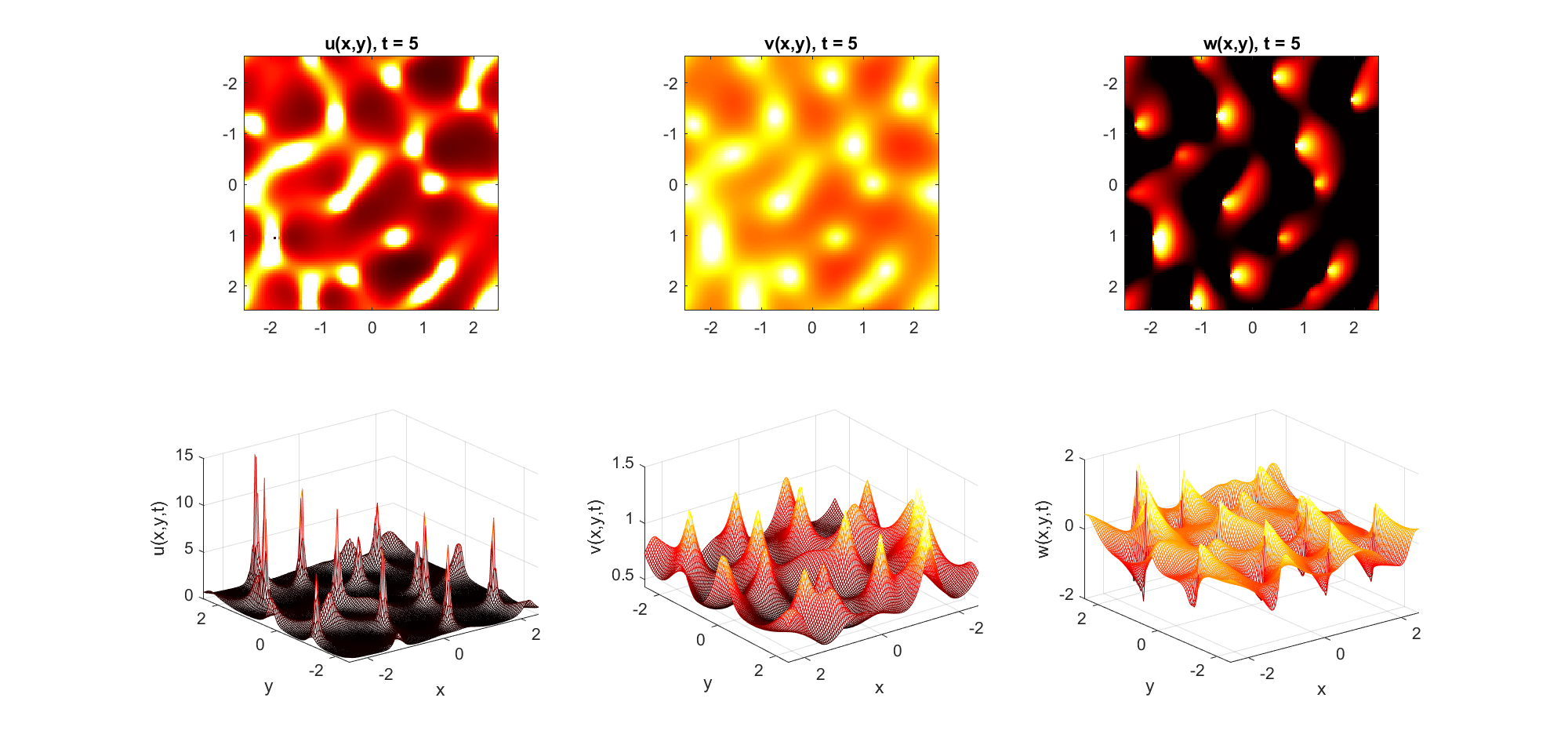}
	\caption{Time evolution of bacterial density $u(x,y,t)$, chemical concentration $v(x,y,t)$, and fluid vorticity $\omega(x,y,t)$ for the chemotaxis--fluid coupling model at \textbf{(Top Row)} $t = 4$ and \textbf{(Bottom Row)} $t = 5$. Left to right: 2D colormaps of $u$, $v$, and $\omega$; below each colormap is the corresponding 3D surface plot. The bacterial density exhibits localized clustering driven by chemotactic attraction, while the chemical field shows spatial depletion in high-density regions. Vorticity plots reveal the bioconvective fluid response, with stronger vortical structures emerging by $t = 5$. Simulation parameters: $D = D_v = 0.1$, $\chi = 1$, $\kappa = 1$, $\nu = 0.1$, $r = \alpha = \beta = K_{\text{cap}} = 1$, domain size $L = 6$, grid size $N = 128$, and time step $\Delta t = 0.01$.}
	\label{fig:t4_t5}
\end{figure}

\section{Results and Discussion}

We simulated the one-dimensional chemotaxis–fluid coupling model described by the system:
\begin{equation}\label{species3c}
	\begin{aligned}
		\partial_t u + w \partial_x u &= D \partial_{xx} u - \chi \partial_x(u \partial_x v) + r u \left(1 - \frac{u}{K} \right), \\
		\partial_t v + w \partial_x v &= D_v \partial_{xx} v + \alpha u - \beta v, \\
		\partial_t w &= \nu \partial_{xx} w + \kappa u,
	\end{aligned}
\end{equation}
on a bounded domain $x \in [0, L]$ with Neumann boundary conditions and small random perturbations as initial conditions.

To investigate the impact of chemotactic sensitivity on the spatial dynamics of the system, we conducted a bifurcation analysis by varying the chemotactic coefficient $\chi$ in the range $[0, 20]$ with an increment of $0.5$. For each value of $\chi$, the one-dimensional coupled chemotaxis--fluid model was simulated up to a fixed final time $T=20$. The system consists of a bacterial density $u(x,t)$, a chemoattractant concentration $v(x,t)$, and a fluid velocity field $w(x,t)$ evolving according to the coupled equations (\ref{species3c}).
Neumann boundary conditions were applied to enforce zero flux at the domain boundaries, and the system was discretized using finite differences with forward Euler time stepping.

To characterize the system's long-term behavior under different values of $\chi$, we recorded the maximum value of $u(x,t)$ at the final time for each simulation. The resulting bifurcation diagram, shown in Figure~\ref{fig:bifurcation}, reveals a transition from spatially homogeneous states to patterned or highly concentrated solutions as $\chi$ increases.

For small values of $\chi$, the chemotactic term is weak, and the bacterial population remains relatively uniform, governed mainly by diffusion and logistic growth. As $\chi$ surpasses a critical threshold ($\chi_c \approx 5.5$ in this simulation), the chemotactic aggregation dominates over diffusion, leading to the formation of localized peaks in bacterial density. The fluid advection term $w \partial_x u$ adds additional nonlinearity to the system and contributes to the observed pattern modulation.

The sharp increase in $\max u$ for large $\chi$ suggests the emergence of singular-like aggregation or blow-up behavior, possibly preceding numerical instability or biological overcrowding phenomena. This analysis highlights the critical role of chemotactic sensitivity in driving pattern formation and nonlinear transitions in microbial--fluid systems.

\begin{figure}[h!]
	\centering
	\includegraphics[width=0.48\textwidth]{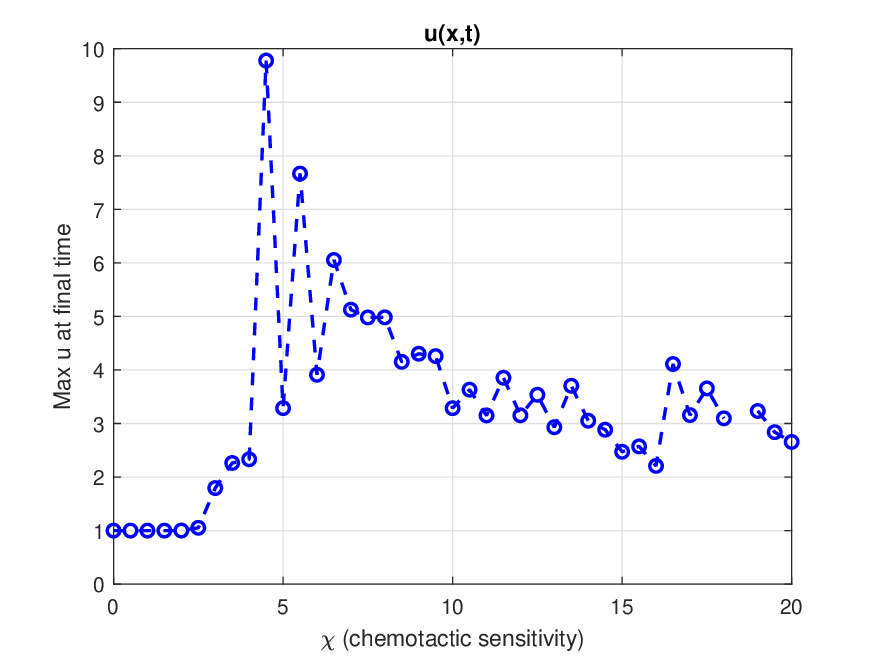}
	\includegraphics[width=0.48\textwidth]{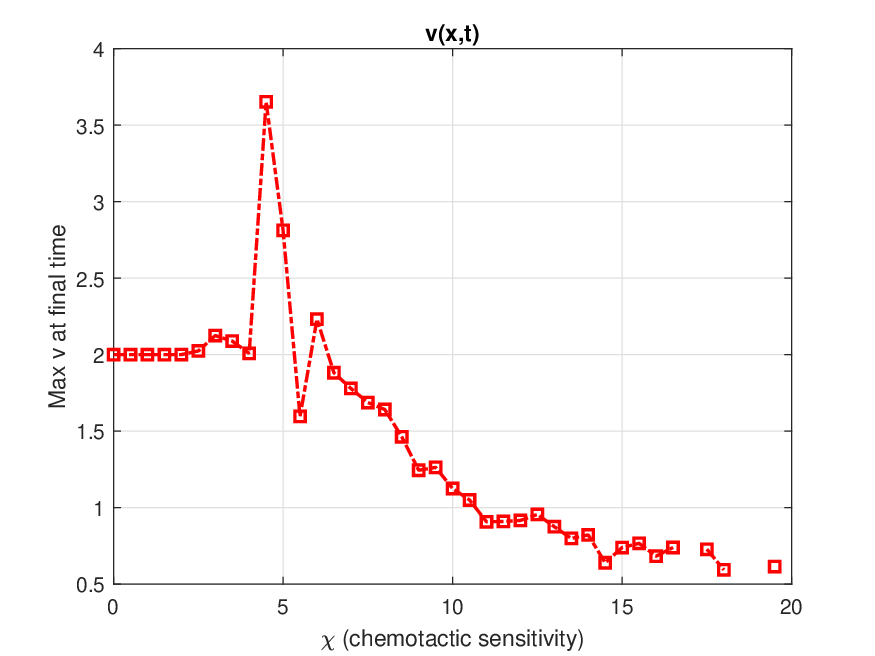}
	\includegraphics[width=0.48\textwidth]{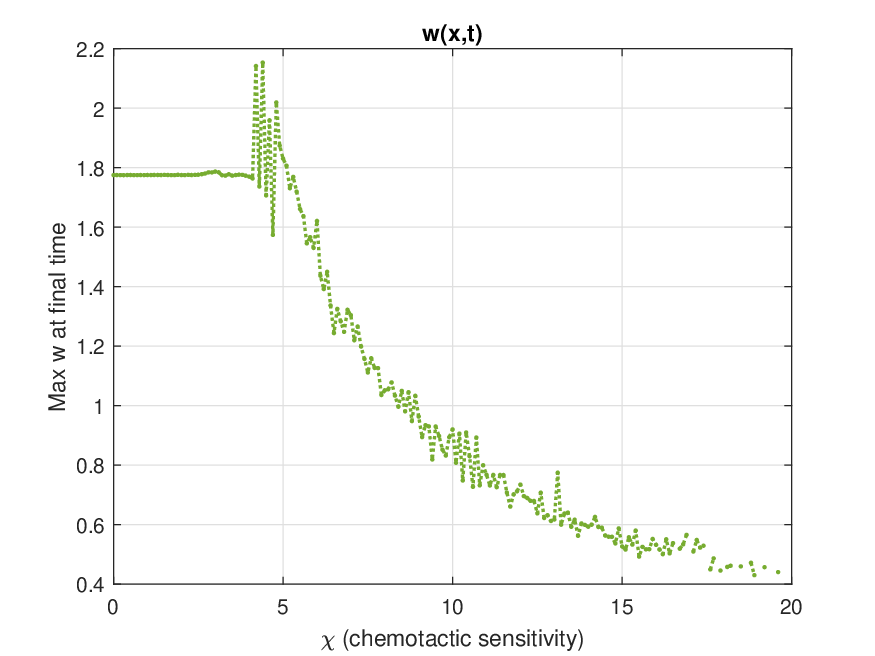}
	\caption{Bifurcation diagram showing the maximum species densities $\max u(x,T)$, $\max v(x,T)$ and $\max w(x,T)$ at final time $T = 20$ as a function of chemotactic sensitivity $\chi$ for model (\ref{species3c}). As $\chi$ increases, the system transitions from a near-homogeneous state to strongly aggregated solutions, illustrating the pattern-forming instability induced by chemotaxis.}
	\label{fig:bifurcation}
\end{figure}

\subsection{Emergence of Spatiotemporal Patterns}

During the early stages of the simulation ($t \in [0,5]$), the solution develops significant spatial structure. The population density $u(x,t)$ forms localized peaks driven by chemotactic aggregation, while the chemical concentration $v(x,t)$ follows with slightly smoothed profiles. The fluid velocity $w(x,t)$ evolves as a response to $u$ via the buoyancy term $\kappa u$, resulting in a dynamic advection field that influences both $u$ and $v$.

\subsection{Chaotic Oscillations and Temporal Complexity}

As time progresses, all three components exhibit irregular, oscillatory behavior without converging to a steady state or limit cycle. This phenomenon is captured in the time series of the midpoint values $u(x_m,t)$ and $v(x_m,t)$, which reveal aperiodic oscillations with variable amplitudes and frequencies. The corresponding phase plane trajectory $(u(x_m,t), v(x_m,t))$ confirms this, tracing out a complex path in phase space without repetition.

These observations are characteristic of \emph{deterministic chaos}, indicating sensitivity to initial conditions and long-term unpredictability. The inclusion of advection through $w$ introduces a dynamic feedback loop: $u$ generates $w$, which in turn advects both $u$ and $v$, distorting the gradient-based chemotactic response. This feedback, combined with the nonlinear growth term $ru(1 - u/K)$, generates rich dynamics and promotes persistent instability.

%\subsection{Bifurcation Structure and Sensitivity to Parameters}

To explore how the dynamics depend on the chemotactic sensitivity $\chi$, we computed a bifurcation diagram by varying $\chi \in [0,6]$ and recording the maximum of $u(x,t)$ at long times. The resulting diagram shows a transition from near-uniform steady states at low $\chi$ to complex, high-amplitude spatiotemporal structures as $\chi$ increases. This suggests that $\chi$ plays a critical role in destabilizing uniform solutions and promoting chaotic aggregation.

%\subsection{Biological Interpretation}

These numerical findings are biologically relevant for microorganisms such as plankton or motile bacteria in aquatic environments, where chemotaxis, growth, and bioconvection interact. The emergence of aperiodic pulsating clusters, even in one spatial dimension, highlights how feedback from fluid dynamics can fundamentally alter pattern formation. In ecological contexts, such complexity may underlie phenomena like patchiness, bloom collapse, or oscillatory dispersal in natural microbial populations.

\section{Conclusion}
In this study, we developed and analyzed a broad class of chemotaxis–reaction models grounded in the Keller–Segel framework. By incorporating nonlinear reaction terms and extending to multi-species and fluid-coupled environments, we captured a range of biologically relevant behaviors from aggregation and blow-up to bounded pattern formation and oscillatory dynamics. Our analytical work established well-posedness conditions, mass-critical blow-up thresholds, and bifurcation criteria for instability-driven pattern emergence. Notably, the interplay between chemotactic sensitivity, diffusion, and nonlinear reactions such as logistic damping was shown to crucially determine whether the system evolves toward uniformity, structured patterns, or singularity. Numerical simulations supported the analytical findings and illustrated rich pattern dynamics beyond the linear regime. Overall, this work bridges theory and computation to elucidate the mechanisms underpinning microbial self-organization and lays the foundation for future investigations into more complex ecological interactions, including stochasticity, anisotropy, and domain geometry effects.

\begin{verbatim}
%Split-Step Fourier Method (SSFM) Pseudocode
% Step 1: Set up spatial and temporal parameters
Nx = 256; Ny = 256;
Lx = 50; Ly = 50;
dx = Lx/Nx; dy = Ly/Ny;
x = linspace(-Lx/2, Lx/2-dx, Nx);
y = linspace(-Ly/2, Ly/2-dy, Ny);
[X, Y] = meshgrid(x, y);

dt = 0.01; T = 50; Nt = round(T/dt);

% Step 2: Define wavenumbers
kx = (2*pi/Lx)*[0:Nx/2-1 -Nx/2:-1]; 
ky = (2*pi/Ly)*[0:Ny/2-1 -Ny/2:-1]; 
[KX, KY] = meshgrid(kx, ky);
K2 = KX.^2 + KY.^2;

% Step 3: Initial conditions
u = u0(X, Y);   % e.g., Gaussian blob
v = v0(X, Y);   % initial chemoattractant

% Step 4: Time-stepping loop
for n = 1:Nt
% --- Diffusion (linear) part in Fourier space ---
u_hat = fft2(u);
v_hat = fft2(v);
u_hat = u_hat .* exp(-D_u * K2 * dt);
v_hat = v_hat .* exp(-D_v * K2 * dt);
u = real(ifft2(u_hat));
v = real(ifft2(v_hat));

% --- Nonlinear chemotaxis-reaction step ---
[ux, uy] = gradient(v, dx, dy);
chemotaxis = divergence(u .* ux, u .* uy);

u = u + dt * (-chi * chemotaxis + f(u, v));
v = v + dt * g(u, v);
end
\end{verbatim}

\begin{verbatim}
%Exponential Time-Differencing Runge–Kutta 4 (ETDRK4) Pseudocode
% Step 1: Same spatial grid and wavenumbers as SSFM
% Step 2: Define linear operators
Lu = -D_u * K2;
Lv = -D_v * K2;

% Step 3: Precompute ETDRK4 coefficients
E_u = exp(Lu * dt); E2_u = exp(Lu * dt/2);
E_v = exp(Lv * dt); E2_v = exp(Lv * dt/2);

M = 16; % contour points for phi functions
r = exp(1i * pi * ((1:M) - 0.5)/M);
phi_u = dt * mean((exp(r .* Lu(:) * dt) - 1) ./ (r .* Lu(:) * dt), 2);
phi_v = dt * mean((exp(r .* Lv(:) * dt) - 1) ./ (r .* Lv(:) * dt), 2);
phi_u = reshape(phi_u, Nx, Ny);
phi_v = reshape(phi_v, Nx, Ny);

% Step 4: Initialize solution
u = u0(X, Y); v = v0(X, Y);

for n = 1:Nt
% Fourier transforms
u_hat = fft2(u); v_hat = fft2(v);

% Evaluate nonlinear terms
N1u = -chi * divergence(u .* gradient(v)) + f(u, v);
N1v = g(u, v);

u1 = real(ifft2(E2_u .* u_hat + 0.5 * dt * fft2(N1u)));
v1 = real(ifft2(E2_v .* v_hat + 0.5 * dt * fft2(N1v)));

N2u = -chi * divergence(u1 .* gradient(v1)) + f(u1, v1);
N2v = g(u1, v1);

u2 = real(ifft2(E2_u .* u_hat + 0.5 * dt * fft2(N2u)));
v2 = real(ifft2(E2_v .* v_hat + 0.5 * dt * fft2(N2v)));

N3u = -chi * divergence(u2 .* gradient(v2)) + f(u2, v2);
N3v = g(u2, v2);

u3 = real(ifft2(E_u .* u_hat + dt * fft2(N3u)));
v3 = real(ifft2(E_v .* v_hat + dt * fft2(N3v)));

N4u = -chi * divergence(u3 .* gradient(v3)) + f(u3, v3);
N4v = g(u3, v3);

% ETDRK4 update in Fourier space
u_hat = E_u .* u_hat + (fft2(N1u) + 2*fft2(N2u) + 2*fft2(N3u) + fft2(N4u)) * dt/6;
v_hat = E_v .* v_hat + (fft2(N1v) + 2*fft2(N2v) + 2*fft2(N3v) + fft2(N4v)) * dt/6;

% Back to real space
u = real(ifft2(u_hat));
v = real(ifft2(v_hat));
end
\end{verbatim}
Both methods assume periodic boundary conditions, ideal for spectral Fourier solvers.
\begin{verbatim}
u_hat = fft2(u);
ux = real(ifft2(1i*KX .* u_hat));
uy = real(ifft2(1i*KY .* u_hat));
\end{verbatim}

\section*{Declaration of competing interest}
The authors declare that they have no known competing financial interests. 

\section*{ Data availability}
No data was used for the research described in the article.


\begin{thebibliography}{99}
\bibitem{arumugam2020}
Arumugam, K., and Tyagi, S. (2020). Keller–Segel chemotaxis models: A review. Acta Applicandae Mathematicae, 169, 29–73.

\bibitem{Baghaei2025}
Baghaei, K., Frassu, S., Tanaka, Y., and Viglialoro, G. (2025). To what extent does the consideration of positive total flux influence the dynamics of Keller–Segel-type models? \textit{arXiv preprint}, \href{https://doi.org/10.48550/arXiv.2505.06586}{arXiv:2505.06586}

\bibitem{Bellomo2015}
Bellomo, N., Bellouquid, A., Tao, Y., and Winkler, M. (2015). Toward a mathematical theory of Keller–Segel models of pattern formation in biological tissues. Mathematical Models and Methods in Applied Sciences, 25(09), 1663–1763.

\bibitem{Brune2024}
Brune, C., Lam, M., and Surulescu, C. (2024). A multiscale 3D–1D Keller–Segel model for cell migration along fiber networks. Mathematical Models and Methods in Applied Sciences, 34(03), 511–539.

\bibitem{Chen2022}
Chen, G., Li, D., and Wang, Y. (2022). Blow-up and pattern formation in a chemotaxis system with nonlinear secretion. Journal of Mathematical Analysis and Applications, 508(2), 125856.

\bibitem{cuevas2025}
Cuevas, C., Garzón, M. J., and Tello, J. I. (2025). Pseudo-measure solvability for the Keller–Segel model. Mathematical Control and Related Fields (AIMS), in press.

\bibitem{Du2023}
Du, W., and Liu, S. (2023). Blow-up solutions of a chemotaxis model with nonlocal effects. Nonlinear Analysis: Real World Applications, 73, 103890.

\bibitem{freingruber2025}
Freingruber, S., Saito, N., and Surulescu, C. (2025). Trait-structured chemotaxis: ligand–receptor dynamics and traveling waves. arXiv preprint, arXiv:2402.07823.

\bibitem{islam2024}
Islam, S., and Ibragimov, G. (2024). A class of Keller–Segel chemotactic systems based on Einstein’s method. Mathematics and Computers in Simulation, 219, 453–468.

\bibitem{Jiang2022}
Jiang, J., Mu, C., and Xiang, Z. (2022). Bifurcation structure of steady states in a Keller–Segel model with saturated sensitivity. SIAM Journal on Applied Mathematics, 82(3), 1097–1120.

\bibitem{KellerSegel1971}
Keller, E. F., and Segel, L. A. (1971). Model for chemotaxis. Journal of Theoretical Biology, 30(2), 225–234.

\bibitem{Kurokiba2023}
Kurokiba, R., and Yokota, T. (2023). Global solvability and boundedness in three-dimensional Keller–Segel systems with nonlinear diffusion and logistic source. Nonlinear Analysis: Real World Applications, 73, 103664.

\bibitem{Lankeit2019}
Lankeit, J., and Winkler, M. (2019). A generalized Keller–Segel system with singular sensitivity: Global existence and asymptotic stabilization. Mathematical Methods in the Applied Sciences, 42(17), 5945–5961.

\bibitem{Liu2021}
Liu, D., Li, T., and Xiang, Z. (2021). Hopf bifurcation and periodic solutions in a chemotaxis model with logistic growth. Journal of Differential Equations, 295, 277–312.

\bibitem{Lorz2021}
Lorz, A., Lorenzi, T., Perthame, B., and Trélat, E. (2021). Modeling and optimal control of a population structured by a space variable and a phenotypic trait. Mathematical Models and Methods in Applied Sciences, 31(07), 1401–1434.

\bibitem{phan2024}
Phan, D., Jiménez, C. C., Mattingly, H. H., and Emonet, T. (2024). Direct measurement of dynamic attractant gradients reveals breakdown of the Patlak–Keller–Segel chemotaxis model. Proceedings of the National Academy of Sciences, 121(15), e2319978121.

\bibitem{vianna2024}
Corrêa Vianna Filho, J., and Guillén-González, F. (2024). A review on analysis, numerical analysis, and control of chemotaxis-consumption models. Applied Mathematics and Optimization, 89(2), 17–53.

\bibitem{weyer2024}
Weyer, J., Widmer, J., and Uecker, H. (2024). Chemotaxis-induced phase separation in logistic growth systems. arXiv preprint, arXiv:2309.12956.

\bibitem{Winkler2010}
Winkler, M. (2010). Aggregation vs. global diffusive behavior in the higher-dimensional Keller–Segel model. Journal of Differential Equations, 248(12), 2889–2905.

\bibitem{Winkler2018}
Winkler, M. (2018). A critical blow-up exponent in a chemotaxis system with nonlinear signal production. Nonlinearity, 31(5), 2031.

\bibitem{Winkler2020}
Winkler, M. (2020). Global classical solvability and stabilization in a two-species chemotaxis–competition system with signal consumption. Mathematical Models and Methods in Applied Sciences, 30(14), 2679–2733.

\bibitem{Xu2025}
Xu, Y., and Fu, H. (2025). A decoupled, mass-conservative, block-centered finite difference method for the Keller–Segel model. arXiv preprint, arXiv:2401.02014.

\bibitem{Zhao2023}
Zhao, H., Wang, Z.-A., and Lin, Y. (2023). Spatiotemporal patterns in a chemotaxis system with double saturation. Nonlinearity, 36(1), 90–128.

\bibitem{Zheng2024}
Zheng, X., Liu, Q., and Lin, P. (2024). Keller–Segel model coupled with incompressible Navier–Stokes equations: numerical methods and simulations. Applied Mathematics and Computation, 442, 127703.
\end{thebibliography}
\end{document}